\documentclass[12pt]{amsart}
\usepackage{amsmath, amssymb, amsfonts}
\usepackage{graphicx, overpic}
\usepackage{fullpage}
\usepackage{color}
\usepackage{enumerate}
\usepackage{tikz} 
\usetikzlibrary{knots}
\usepackage{pst-knot}
\usepackage{epstopdf}
\usepackage{subfigure}
\usepackage[all]{xy}
\usepackage{extarrows}
\usetikzlibrary{patterns}
\usetikzlibrary{arrows}
\usepackage{abstract}
\usepackage{here}

\usetikzlibrary{decorations.pathreplacing,decorations.markings}
\usetikzlibrary{patterns}
\usetikzlibrary{arrows}
\usetikzlibrary{decorations.markings}

\theoremstyle{plain}
\newtheorem{thm}{Theorem}[section]
\newtheorem{lemma}[thm]{Lemma}
\newtheorem{prop}[thm]{Proposition}
\newtheorem{cor}[thm]{Corollary}
\newtheorem{definition}[thm]{Definition}
\newtheorem*{theorem*}{Theorem}
\newtheorem*{prop*}{Proposition}
\newtheorem*{main theorem*}{Main Theorem}

\theoremstyle{definition}
\newtheorem{remark}[thm]{Remark}
\newtheorem{example}[thm]{Example}

\numberwithin{equation}{section}
\numberwithin{figure}{section}

\newcommand\blfootnote[1]{%
  \begingroup
  \renewcommand\thefootnote{}\footnote{#1}%
  \addtocounter{footnote}{-1}%
  \endgroup
}

\begin{document}

\tikzset{middlearrow/.style={
		decoration={markings,
			mark= at position 0.55 with {\arrow{#1}} ,
		},
		postaction={decorate}
	}
}

\title{Identifying Dehn Functions of Bestvina--Brady Groups From Their Defining Graphs}
\author{Yu-Chan Chang}
\address{Oxford College of Emory University \\ 801 Emory Street, Oxford, GA 30054}
\email{yuchanchang74321@gmail.com \\ yu-chan.chang@emory.edu}

\maketitle

\begin{abstract}
Let $\Gamma$ be a finite simplicial graph such that the flag complex on $\Gamma$ is a $2$-dimensional triangulated disk. We show that with some assumptions, the Dehn function of the associated Bestvina--Brady group is either quadratic, cubic, or quartic. Furthermore, we can identify the Dehn function from the defining graph $\Gamma$.
\end{abstract}

\blfootnote{The author gratefully acknowledges the support from the NSF Grant DMS-1812061.}

\section{introduction}

Dehn functions are one of the quasi-isometry invariants of finitely presented groups, and they have been studied by many people. One of the reasons people study Dehn functions is that they are related to the solvability of the word problem for finitely presented groups. That is, a finitely presented group has a solvable word problem if and only if its Dehn function is recursive. Besides the solvability of the word problem, Dehn functions can also detect certain structures in groups. For example, a group is hyperbolic if and only if it has a linear Dehn function \cite{gromovhyperbolicgroup}. We refer to \cite{bridsongeometryofwordproblem} for background on Dehn functions. In this paper, we study Dehn functions of Bestvina--Brady groups, which are a class of subgroups of right-angled Artin groups. 

Given a finite simplicial graph $\Gamma$, the associated right-angled Artin group $A_{\Gamma}$ is generated by the vertices of $\Gamma$. The relators are commutators: two generators $u$, $v$ commute if and only if they are adjacent vertices of $\Gamma$. Right-angled Artin groups have become an important objects that people study in geometric group theory; see \cite{charneyanintroductiontoraag} for a general survey. They are known to be $\mathrm{CAT}(0)$ groups; both categories of groups have at most quadratic Dehn functions. But subgroups can have larger Dehn functions. Brady and Forester \cite{bradyandforester} gave examples of $\mathrm{CAT}(0)$ groups that contain finitely presented subgroups whose Dehn functions are of the form $n^{\rho}$, for a dense set of $\rho\in[2,\infty)$. Brady and Soroko \cite{bradyandsorokodehnfunctionsofsubgroupsofraags} proved that for each positive integer $\rho$, there is a right-angled Artin group that contains a finitely presented subgroup whose Dehn function is $n^{\rho}$.

For a right-angled Artin group $A_{\Gamma}$, the associated Bestvina--Brady group $H_{\Gamma}$ is defined to be the kernel of the homomorphism $A_{\Gamma}\rightarrow\mathbb{Z}$, which sends all the generators of $A_{\Gamma}$ to $1$. This kernel had been studied prior to Bestvina--Brady. Stallings \cite{stallingsafinitelypresentedgroupwhose3dimentionalintegralhomologyisnotfinitelygenerated} constructed a group that is finitely generated but not finitely presented. This group can be realized as the Bestvina--Brady group $H_{\Gamma}=\ker(F_{2}\times F_{2}\rightarrow\mathbb{Z})$, where $\Gamma$ is a cycle graph $C_{4}$. When $\Gamma$ is taken to be the $(n+1)$-fold join of two vertices, the right-angled Artin group on $\Gamma$ is the $n$-copies of $F_{2}$, and Bieri \cite{bierinormalsubgroupsindualitygroupsandingroupsofcohomologicaldimension2} proved that the associated Bestvina--Brady group $H_{\Gamma}=\ker(F_{2}\times\cdots\times F_{2}\rightarrow\mathbb{Z})$ satisfies the finiteness property $\mathbf{FP}_{n}$ but not $\mathbf{FP}_{n+1}$. Bestvina and Brady \cite{bestvinaandbrady} gave a systematic construction of groups that satisfy some finiteness properties but not others.  Moreover, there are Bestvina--Brady groups that are either counterexamples to the Eilenberg--Ganea Conjecture or counterexamples to the Whitehead Conjecture (\cite{bestvinaandbrady}, Theorem 8.7).

Dison proved that the Dehn functions of Bestvina--Brady groups are bounded above by quartic polynomials \cite{dison}. We are interested in knowing whether all the Bestvina--Brady groups have Dehn functions of the form $n^{\alpha}$, $\alpha=1,2,3$ or $4$. Our motivation is Brady's examples in (\cite{bradyrileyshortthegeometryofthewordproblemforfinitelygeneratedgroups}, Part I), where he gave Bestvina--Brady groups that realize quadratic, cubic, and quartic Dehn functions. In each of those examples, the flag complex on the defining graph $\Gamma$, denoted by $\Delta_{\Gamma}$, is a $2$-dimensional triangulated disk with square boundary. In this paper, we prove that the Dehn functions of Bestvina--Brady groups $H_{\Gamma}$ with such restrictions on the defining graphs $\Gamma$ are of the form $n^{\alpha}$, $\alpha=1,2,3$ or $4$. Furthermore, we provide a way to identify the Dehn functions of those Bestvina--Brady groups by their defining graphs. A simplex of a $2$-dimensional triangulated disk $D$ is called an \emph{interior simplex} if none of its faces are on $\partial D$. If $D$ has interior $d$-simplices, and no other interior $k$-simplices, $k>d$, then we say that $D$ has the \emph{interior dimension} $d$, denoted by $\dim_{I}(D)=d$. If $D$ has no interior simplices, then we define $\dim_{I}(D)=0$. Our main result is the following theorem:

\begin{thm}\label{Dehn functions for square boundary introduction}
Let $\Gamma$ be a finite simplicial graph such that $\Delta_{\Gamma}$ is a $2$-dimensional triangulated disk whose boundary is a square. If $\dim_{I}(\Delta_{\Gamma})=d$ for $d\in\lbrace 0,1,2 \rbrace$, then $\delta_{H_{\Gamma}}(n)\cong n^{d+2}$.
\end{thm}

\begin{example}\label{example of interior dimension}
The flag complexes on the graphs shown in Figure \ref{Disks of interior dimensions $0,1,2$, respectively, and with square boundaries} are $2$-dimensional triangulated disks with square boundary, and they have interior dimensions $0$, $1$, and $2$, respectively. 

\begin{figure}[H]
\begin{center}
\begin{tikzpicture}[scale=0.8]
\draw (-6.5,0)--(-3.5,0);
\draw (-5,2)--(-6.5,0);
\draw (-5,2)--(-3.5,0);
\draw (-5,-2)--(-6.5,0);
\draw (-5,-2)--(-3.5,0);
\draw (-5,-2)--(-5,2);

\draw [fill] (-5,2) circle [radius=0.05];
\draw [fill] (-6.5,0) circle [radius=0.05];
\draw [fill] (-5,0) circle [radius=0.05];
\draw [fill] (-3.5,0) circle [radius=0.05];
\draw [fill] (-5,-2) circle [radius=0.05];

\node [below] at (-5,-2.2) {(a)};

\draw (-1.5,0)--(1.5,0);
\draw (0,2)--(-1.5,0);
\draw (0,2)--(1.5,0);
\draw (0,2)--(-0.5,0);
\draw (0,2)--(0.5,0);
\draw (0,-2)--(-1.5,0);
\draw (0,-2)--(1.5,0);
\draw (0,-2)--(-0.5,0);
\draw (0,-2)--(0.5,0);
		
\draw [fill] (1.5,0) circle [radius=0.05];
\draw [fill] (-1.5,0) circle [radius=0.05];
\draw [fill] (0,2) circle [radius=0.05];
\draw [fill] (0,-2) circle [radius=0.05];
\draw [fill] (-0.5,0) circle [radius=0.05];
\draw [fill] (0.5,0) circle [radius=0.05];
		
\node [below] at (0,-2.2) {(b)};

\draw (3.5,0)--(6.5,0);
\draw (5,2)--(3.5,0)--(5,-2)--(6.5,0)--(5,2);
\draw (5,1)--(3.5,0)--(5,-1)--(6.5,0)--(5,1);
\draw (5,1)--(4.5,0)--(5,-1)--(5.5,0)--(5,1);
\draw (5,2)--(5,1);
\draw (5,-2)--(5,-1);

\draw [fill] (3.5,0) circle [radius=0.05];
\draw [fill] (4.5,0) circle [radius=0.05];
\draw [fill] (5,2) circle [radius=0.05];
\draw [fill] (5,-2) circle [radius=0.05];
\draw [fill] (5.5,0) circle [radius=0.05];
\draw [fill] (6.5,0) circle [radius=0.05];
\draw [fill] (5,1) circle [radius=0.05];
\draw [fill] (5,-1) circle [radius=0.05];
\draw [fill] (4.5,0) circle [radius=0.05];
\draw [fill] (5.5,0) circle [radius=0.05];
\draw [fill] (5,-2) circle [radius=0.05];
		
\node [below] at (5,-2.2) {(c)};
\end{tikzpicture}
\end{center}
\caption{}
\label{Disks of interior dimensions $0,1,2$, respectively, and with square boundaries}
\end{figure}	
\end{example}

We now briefly discuss the proof of Theorem \ref{Dehn functions for square boundary introduction}. When $\dim_{I}(\Delta_{\Gamma})=0$, we can eliminate the square boundary condition on $\Delta_{\Gamma}$; see Theorem \ref{Dehn function of interior dimension 0}. In this case, the Bestvina--Brady group $H_{\Gamma}$ has a graph of groups decomposition, where the edge groups are infinite cyclic, and the vertex groups are right-angled Artin groups. In fact, these Bestvina--Brady groups are $\mathrm{CAT}(0)$, therefore, their Dehn functions are at most quadratic. Since we assume that the flag complexes are triangulated disks, the associated Bestvina--Brady groups have $\mathbb{Z}^{2}$ subgroups. Thus, these Bestvina--Brady groups are not hyperbolic and have at least quadratic Dehn functions. We remark that some cases of $\dim_{I}(\Delta_{\Gamma})=0$ in Theorem \ref{Dehn functions for square boundary introduction} can be obtained by a result of Carter and Forester \cite{carterandforester}:

\begin{thm}\textnormal{(\cite{carterandforester}, Corollary 4.3)}\label{Join of three graphs has quadratic Dehn function}
If a finite simplicial graph $\Gamma$ is a join of three graphs $\Gamma=\Gamma_{1}\ast\Gamma_{2}\ast\Gamma_{3}$, then $\delta_{H_{\Gamma}}$ is quadratic.
\end{thm}

\begin{example}\label{Example of Carter and Forester}
Let $\Gamma$ be the graph in Figure \ref{Disks of interior dimensions $0,1,2$, respectively, and with square boundaries} (a). Label the vertices as follows:
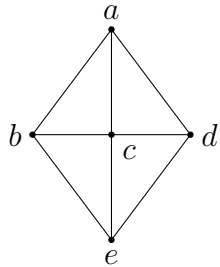
\begin{center}
\begin{tikzpicture}[scale=0.7]
\draw (-6.5,0)--(-3.5,0);
\draw (-5,2)--(-6.5,0);
\draw (-5,2)--(-3.5,0);
\draw (-5,-2)--(-6.5,0);
\draw (-5,-2)--(-3.5,0);
\draw (-5,-2)--(-5,2);

\draw [fill] (-5,2) circle [radius=0.05];
\draw [fill] (-6.5,0) circle [radius=0.05];
\draw [fill] (-5,0) circle [radius=0.05];
\draw [fill] (-3.5,0) circle [radius=0.05];
\draw [fill] (-5,-2) circle [radius=0.05];

\node [above] at (-5,2) {$a$};
\node [left] at (-6.5,0) {$b$};
\node [below right] at (-5,0) {$c$}; 
\node [right] at (-3.5,0) {$d$};
\node [below] at (-5,-2) {$e$};
\end{tikzpicture}
\end{center}
Let $\Gamma_{1}=\lbrace c\rbrace$, $\Gamma_{2}=\lbrace b,d\rbrace$, and $\Gamma_{3}=\lbrace a,e\rbrace$, then $\Gamma=\Gamma_{1}\ast\Gamma_{2}\ast\Gamma_{3}$. Therefore, $\delta_{H_{\Gamma}}$ is quadratic by Theorem \ref{Join of three graphs has quadratic Dehn function}. It also follows from Theorem \ref{Dehn functions for square boundary introduction} that $\delta_{H_{\Gamma}}$ is quadratic.
\end{example}

In \cite{bradyrileyshortthegeometryofthewordproblemforfinitelygeneratedgroups}, Brady proved that the Bestvina--Brady group on the graph shown in Figure \ref{Disks of interior dimensions $0,1,2$, respectively, and with square boundaries} ~(b) has a cubic Dehn function. This graph can be seen as the suspension of a path of length ~$3$. When $\dim_{I}(\Delta_{\Gamma})=1$, we prove that $\Gamma$ is the suspension of a path, and we show that the associated Bestvina--Brady group has a cubic Dehn function. To achieve the cubic upper bound, we use the \emph{corridor schemes} \cite{bradyandforester} to analyse the van Kampen diagrams carefully. Lemma \ref{stacks are cubic} is the main technical result of this paper. In this technical lemma, we prove that the area of a special region, called a \emph{stack}, in a van Kampen diagram is bounded above by a cubic function of the perimeter of that region. This result allows us to obtain the desired cubic upper bound; see Lemma \ref{The suspension of a path of length at least 3 is at most cubic} for a detail proof. 

Since there is a universal quartic upper bound on the Dehn functions of Bestvina--Brady groups, the remaining cases are the cubic and quartic lower bounds for $\dim_{I}(\Delta_{\Gamma})=1,2$, respectively. In \cite{abddy}, the authors introduced the \emph{height-pushing map} to obtain the lower bound on the higher Dehn functions of orthoplex groups. Their method can be adapted to our proof to obtain the desired lower bounds. We want to point out that their theorem (\cite{abddy}, Theorem ~5.1) recovers Dison's quartic upper bound in \cite{dison}. We have discussed all the cases of our main result Theorem \ref{Dehn functions for square boundary introduction}.

Denote $K_{4}$ to be the complete graph on four vertices. Note that the assumption of the flag complex $\Delta_{\Gamma}$ being $2$-dimensional is equivalent to saying that the graph $\Gamma$ does not have $K_{4}$ subgraphs. Suppose a given graph $\Gamma$ whose flag complex is a triangulated disk but not necessarily $2$-dimensional or has square boundary. If $\Gamma$ contains a subgraph that satisfies the assumptions of Theorem \ref{Dehn functions for square boundary introduction}, then we can obtain a lower bound on the Dehn function of the associated Bestvina--Brady group $H_{\Gamma}$:

\begin{prop}\label{Lower bound on K4 cases introduction}
Let $\Gamma$ be a finite simplicial graph such that $\Delta_{\Gamma}$ is simply-connected. If $\Gamma$ contains an induced subgraph $\Gamma'$ such that $\Delta_{\Gamma'}$ is a $2$-dimensional triangulated subdisk of $\Delta_{\Gamma}$ that has square boundary and $\dim_{I}(\Delta_{\Gamma'})=d$ for $d\in\lbrace 0,1,2 \rbrace$, then $n^{d+2}\preceq\delta_{H_{\Gamma}}(n)$.
\end{prop}

This paper is organized as follows. Section \ref{section preliminaries} provides some necessary background. In Section \ref{section disk with interior dimension 0}, we prove the case of $\dim_{I}(\Delta_{\Gamma})=0$ without the square boundary assumption. Section \ref{section disks with square boundaries} is devoted to the proof of Theorem \ref{Dehn functions for square boundary introduction}. In Section \ref{section graphs with K4 subgraphs}, we prove Proposition \ref{Lower bound on K4 cases introduction}.

\section*{acknowledgements}

I want to thank Pallavi Dani for introducing me to this project and her generous advice. I want to thank Tullia Dymarz, Max Forester, and Bogdan Oporowski for many helpful conversations and their support. I want to thank the referee for valuable suggestions and feedback.

\section{preliminaries}\label{section preliminaries}

\subsection{Dehn functions}\label{subsection Dehn functions}

Let $G$ be a group with a finite presentation $\mathcal{P}=\langle\mathcal{S}\vert\mathcal{R}\rangle$. Let $w$ be a word in $F(\mathcal{S})$ that represents the identity of $G$, denoted by $w\equiv_{G}1$. The $\emph{area}$ of $w$, denoted by Area$(w)$, is defined as follows:
\begin{align*}
\text{Area}(w)=\min
\bigg\lbrace N \ \bigg\vert w\overset{F(\mathcal{S})}{=}\prod^{N}_{i=1}x_{i}r^{\pm 1}_{i}x^{-1}_{i}, x_{i}\in F(\mathcal{S}), r_{i}\in\mathcal{R}\bigg\rbrace,
\end{align*}
where $F(\mathcal{S})$ is the free group generated by $S$. The \emph{Dehn function} $\delta_{G}:\mathbb{N}\rightarrow\mathbb{N}$ of a group $G$ over the presentation $\mathcal{P}=\langle\mathcal{S}\vert\mathcal{R}\rangle$ is defined by 
\begin{align*}
\delta_{\mathcal{P}}(n)=\max\bigg\lbrace\text{Area}_{G}(w)\ \bigg\vert \ w\equiv_{G}1, |w|\leq n\bigg\rbrace,
\end{align*}
where $|w|$ denotes the length of the word $w$.

\begin{definition}\label{equivalence of functions}
Let $f,g:[0,\infty)\rightarrow[0,\infty)$ be two functions. We say that $f$ is \emph{bounded above} by $g$, denoted by $f\preceq g$, if there is a number $C>0$ such that $f(n)\leq Cg(Cn+C)+Cn+C$ for all $n>0$. We say that $f$ and $g$ are \emph{$\simeq$-equivalent}, or simply \emph{equivalent}, denoted by $f\simeq g$, if $f\preceq g$ and $g\preceq f$. 
\end{definition}

If $\mathcal{P}_{1}$ and $\mathcal{P}_{2}$ are finite presentations of a group $G$, then $\delta_{\mathcal{P}_{1}}$ is equivalent to $\delta_{\mathcal{P}_{2}}$; we refer to \cite{bridsongeometryofwordproblem} for a proof of this fact. We denote the $\simeq$-equivalent class of $G$ by $\delta_{G}$, and call it the \emph{Dehn function} of $G$. We say that the Dehn function $\delta_{G}$ is \emph{linear}, \emph{quadratic}, \emph{cubic} or \emph{quartic} if for all $n\in\mathbb{N}$, $\delta_{G}(n)\simeq n$, $\delta_{G}(n)\simeq n^{2}$, $\delta_{G}(n)\simeq n^{3}$ or $\delta_{G}(n)\simeq n^{4}$, respectively.

\subsection{Right-angled Artin groups, Bestvina--Brady groups, and the Dicks--Leary presentation}\label{subsection RAAG, BBG, DLP}

Let $\Gamma$ be a finite simplicial graph and $\mathrm{V}(\Gamma)$ the set of vertices of $\Gamma$. The \emph{right-angled Artin group} $A_{\Gamma}$ associated to $\Gamma$ has the following presentation:
\begin{align*}
A_{\Gamma}=\bigg\langle\text{V}(\Gamma) \ \bigg\vert \ [v_{i},v_{j}] \ \text{whenever $v_{i}$ and $v_{j}$ are connected by an edge of $\Gamma$}\bigg\rangle.
\end{align*}
When $\Gamma$ is a complete graph $K_{n}$ on $n$ vertices, $A_{\Gamma}=\mathbb{Z}^{n}$; when $\Gamma$ is a set of $n$ distinct points, $A_{\Gamma}=F_{n}$, the free group of rank $n$.

For each finite simplicial graph $\Gamma$, its associated right-angled Artin group $A_{\Gamma}$ is the fundamental group of a cubical complex $X_{\Gamma}$, called the \emph{Salvetti complex}. It is well-known that the Salvetti complex $X_{\Gamma}$ is compact and non-positively curved, and its universal cover $\widetilde{X}_{\Gamma}$ is a $\mathrm{CAT}(0)$ cube complex. Moreover, right-angled Artin groups are $\mathrm{CAT}(0)$ groups; thus, they have at most quadratic Dehn functions. We refer to  \cite{charneyanintroductiontoraag} for more details of these facts.

Given a finite simplicial graph $\Gamma$, we define a group homomorphism $\phi:A_{\Gamma}\rightarrow\mathbb{Z}$ by sending all the generators of $A_{\Gamma}$ to $1$. The kernel of this homomorphism is called the \emph{Bestvina--Brady group} defined by $\Gamma$, and is denoted by $H_{\Gamma}$. In fact, the map $\phi:A_{\Gamma}\rightarrow\mathbb{Z}$ is induced by $l:X_{\Gamma}\rightarrow S^{1}$, and the lift of $l$, denoted by $h:\widetilde{X}_{\Gamma}\rightarrow\mathbb{R}$ is a $\phi$-equivariant Morse function; see \cite{bestvinaandbrady}, Theorem 5.12. Restricting the action of $A_{\Gamma}$ on $\widetilde{X}_{\Gamma}$ to $H_{\Gamma}$, we obtain a geometric action of $H_{\Gamma}$ on the zero level set $Z_{\Gamma}=h^{-1}(0)$.

The \emph{flag complex $\Delta_{\Gamma}$} on a finite simplicial graph $\Gamma$ is a simplicial complex such that each complete subgraph $K_{n}$ of $\Gamma$ spans an $(n-1)$-simplex in $\Delta_{\Gamma}$. When $\Delta_{\Gamma}$ is connected, $H_{\Gamma}$ is finitely generated; when $\Delta_{\Gamma}$ is simply-connected, $H_{\Gamma}$ is finitely presented; see \cite{bestvinaandbrady} for the proof of these facts. When $H_{\Gamma}$ is finitely presented, we can write down its  \emph{Dicks--Leary presentation} \cite{dicksandleary}: 

\begin{thm}\textnormal{(\cite{dicksandleary}, Corollary 3)}\label{Dicks and Leary presentation}
Let $\Gamma$ be a finite simplicial oriented graph. Suppose that $\Delta_{\Gamma}$ is simply-connected. Then the Bestvina--Brady group $H_{\Gamma}$ has the following finite presentation:
\begin{align*}
H_{\Gamma}=\bigg\langle\overline{E}(\Gamma) \ \bigg\vert \ ef=g=fe \ \text{whenever $e,f,g$ form an oriented triangle }\bigg\rangle,
\end{align*}
where $\overline{E}(\Gamma)$ is the set of oriented edges of $\Gamma$, and the oriented triangle is shown in Figure \ref{oriented triangle}
\begin{figure}[H]
\begin{center}
\begin{tikzpicture}[scale=0.5]
\draw [middlearrow={stealth}] (4,3)--(0,0);
\draw [middlearrow={stealth}] (0,0)--(6,0);
\draw [middlearrow={stealth}] (4,3)--(6,0);

\node [left] at (2,2) {$e$};
\node at (3,-1) {$f$};
\node [right] at (5,2) {$g$};

\end{tikzpicture}
\end{center}
\caption{Realator of the Dicks--Leary presentation.}
\label{oriented triangle}
\end{figure}
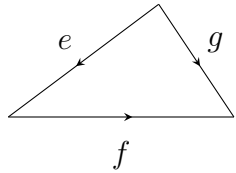
\end{thm}

We will use the fact that $H_{\Gamma}$ is finitely presented when $\Delta_{\Gamma}$ is simply-connected without specifying an orientation on $\Gamma$. We will give an orientation on $\Gamma$ when we need to work with a finite presentation for $H_{\Gamma}$. The Dicks--Leary presentation can be reduced further:

\begin{cor}\textnormal{(\cite{stefansuciualgebraicinvariantsforbbgroups}, Corollary 2.3)}\label{Generating set of BB}
If the flag complex on a finite simplicial graph $\Gamma$ is simply-connected, then $H_{\Gamma}$ has a presentation $H_{\Gamma}=F/R$, where $F$ is the free group generated by the edges in a maximal tree of $\Gamma$, and $R$ is a finitely generated normal subgroup of the commutator group $[F,F]$. 
\end{cor}

While Dehn functions of right-angled Artin groups are at most quadratic, Dison \cite{dison} proved that Dehn functions of Bestvina--Brady groups are bounded above by quartic functions. 

\begin{thm}\textnormal{(\cite{dison})}
Dehn functions of Bestvina--Brady groups are at most quartic.
\end{thm}

\subsection{Interior dimensions}

Let $D$ be a triangulated disk. An \emph{interior $i$-simplex} of $D$ is an $i$-simplex whose faces do not intersect $\partial D$. We also call an interior $0$-simplex an \emph{interior vertex}, an interior $1$-simplex an \emph{interior edge}, and an interior $2$-simplex an \emph{interior triangle}. We say that $D$ has \emph{interior dimension $d$}, denoted by $\dim_{I}(D)=d$, if $D$ contains interior $d$-simplices, and it has no interior $k$-simplices, $k>d$. If $D$ contains no interior simplices, then we define $\dim_{I}(D)=0$. We refer to Example \ref{example of interior dimension} for concrete examples.

\section{Disks with interior dimension $0$}\label{section disk with interior dimension 0}

In this section, we prove the following theorem.

\begin{thm}\label{Dehn function of interior dimension 0}
Let $\Gamma$ be a finite simplicial graph. If $\Delta_{\Gamma}$ is a $2$-dimensional triangulated disk satisfying $\dim_{I}(\Delta_{\Gamma})=0$, then $\delta_{H_{\Gamma}}(n)\simeq n^{2}$. 
\end{thm}

For such a graph $\Gamma$, we will see later that the associated Bestvina--Brady group $H_{\Gamma}$ has a graph-of-groups decomposition, where the edges groups are infinite cyclic groups; and the vertex groups are Bestvina--Brady groups on some induced subgraphs of $\Gamma$, namely, \emph{fans} and \emph{wheels}. Moreover, each of the vertex groups is isomorphic to a non-hyperbolic right-angled Artin group.  

Recall that the \emph{join} of two graphs $\Gamma_{1}$ and $\Gamma_{2}$, denoted by $\Gamma_{1}\ast\Gamma_{2}$, is the graph obtained by taking the disjoint union of $\Gamma_{1}$ and $\Gamma_{2}$ together with all the edges that connect the vertices of $\Gamma_{1}$ and the vertices of $\Gamma_{2}$.

\begin{definition}
A \emph{fan} $\mathcal{F}_{n+1}$ is the join of a vertex and a path $P_{n}$. A \emph{wheel} $\mathcal{W}_{n+1}$ is the join of a vertex and a cycle $C_{n}$. 
\end{definition}  

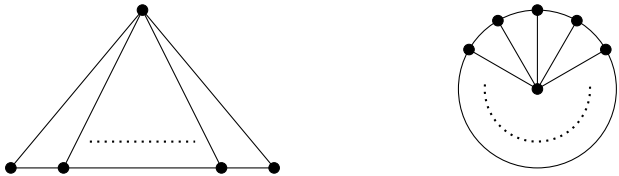
\begin{figure}[H]
\begin{center}
\begin{tikzpicture}[scale=0.7]

\draw (-10,-1.5)--(-5,-1.5);

\draw [dotted,thick] (-8.5,-1)--(-6.5,-1);

\draw (-7.5,1.5)--(-5,-1.5);
\draw (-7.5,1.5)--(-6,-1.5);
\draw (-7.5,1.5)--(-9,-1.5);
\draw (-7.5,1.5)--(-10,-1.5);

\draw [fill] (-7.5,1.5) circle [radius=0.1];
\draw [fill] (-10,-1.5) circle [radius=0.1];
\draw [fill] (-9,-1.5) circle [radius=0.1];
\draw [fill] (-6,-1.5) circle [radius=0.1];
\draw [fill] (-5,-1.5) circle [radius=0.1];

\draw (0,0) circle [radius=1.5];
	
\draw [fill] (0,0) circle [radius=0.1];
	
\draw (0,0)--(30:1.5cm);
\draw (0,0)--(60:1.5cm);
\draw (0,0)--(90:1.5cm);
\draw (0,0)--(120:1.5cm);
\draw (0,0)--(150:1.5cm);
	
\draw [dotted, thick] (175:1cm) arc [start angle=175, end angle=365, radius=1cm];
	
\draw [fill] (30:1.5cm) circle [radius=0.1];
\draw [fill] (60:1.5cm) circle [radius=0.1];
\draw [fill] (90:1.5cm) circle [radius=0.1];
\draw [fill] (120:1.5cm) circle [radius=0.1];
\draw [fill] (150:1.5cm) circle [radius=0.1];
\end{tikzpicture}
\end{center}
\caption{Fan and Wheel.}
\end{figure}

\begin{remark}
The flag complex on $\mathcal{W}_{4}$ is a tetrahedron, which is not $2$-dimensional. Throughout this paper, unless otherwise stated, all the wheels have at least five vertices, that is, $\mathcal{W}_{n}$ for $n\geq 5$. Note that a triangle is also a fan $\mathcal{F}_{3}$.
\end{remark}

Fans and wheels have a special structure: they can be decomposed as a join of a vertex and graph. When a finite simplicial graph $\Gamma$ decomposes as a join $\Gamma=\lbrace v\rbrace\ast\Gamma'$, the Dicks--Leary presentation gives $H_{\Gamma}\cong A_{\Gamma'}$ (see Example 2.5 in \cite{stefansuciualgebraicinvariantsforbbgroups}). That is, $H_{\Gamma}$ is a right-angled Artin group. Thus, $\delta_{H_{\Gamma}}$ is at most quadratic.

\begin{prop}\label{BB=RAAG}
Suppose a finite simplicial graph decomposes as a join $\Gamma=\lbrace v\rbrace\ast\Gamma'$. Then $H_{\Gamma}\cong A_{\Gamma'}$ and $\delta_{H_{\Gamma}}$ is at most quadratic. Moreover, if $\Gamma'$ contains an edge, then $A_{\Gamma'}$ is non-hyperbolic and $\delta_{H_{\Gamma}}$ is quadratic.
\end{prop}

\begin{proof}
Since $\Gamma=\lbrace v\rbrace\ast\Gamma'$, we have $A_{\Gamma}=\mathbb{Z}\times A_{\Gamma'}$. We claim that $H_{\Gamma}\cong A_{\Gamma'}$. Label the edges that have $v$ as the common endpoint by $e_{1},\cdots,e_{k}$; label the other end points of $e_{1},\cdots,e_{k}$ by $v_{1},\cdots,v_{k}$. Since $e_{1},\cdots,e_{k}$ form a maximal tree of $\Gamma$, they generate $H_{\Gamma}$; see Corollary \ref{Generating set of BB}. Meanwhile, $v_{1},\cdots,v_{k}$ generate $A_{\Gamma'}$. Define a map $\psi:H_{\Gamma}\rightarrow A_{\Gamma'}$ by sending $e_{i}$ to $v_{i}$ for $i=1,\cdots,k$. This is a bijection between the generating sets of $H_{\Gamma}$ and $A_{\Gamma'}$. We now argue that $\psi$ preserves relators. Note that the relators of $H_{\Gamma}$ and $A_{\Gamma'}$ are commutators. The generators $e_{i}$ and $e_{j}$ commute when they are two edges of the same triangle, that is, when their end points $v_{i}$ and $v_{j}$ are connected by an edge. Thus, $v_{i},v_{j}$ commute whenever $e_{i},e_{j}$ commute; the converse is true. Hence, $\psi$ is an isomorphism.	

Since $H_{\Gamma}\cong A_{\Gamma'}$ and $\delta_{A_{\Gamma'}}$ is at most quadratic, $\delta_{H_{\Gamma}}$ is at most quadratic. If $\Gamma'$ contains an edge, then $H_{\Gamma}\cong A_{\Gamma'}$ contains $\mathbb{Z}\times\mathbb{Z}$ as a subgroup. Therefore, $H_{\Gamma'}$ cannot be hyperbolic and $\delta_{H_{\Gamma}}$ has to be quadratic.
\end{proof}

\begin{remark}
In Proposition \ref{BB=RAAG} $\Gamma$ may have $K_{4}$ subgraphs. So there are infinitely many non-$2$-dimensional Bestvina--Brady groups whose Dehn functions are quadratic. 
\end{remark}

We want to point out that there are Bestvina--Brady groups that are not isomorphic to any right-angled Artin groups (\cite{stefansuciualgebraicinvariantsforbbgroups}, Proposition 9.4). The following corollary is an immediate consequence of Proposition \ref{BB=RAAG}.

\begin{cor}\label{fan and wheel have quadratic dehn functions}
If $\Gamma$ is a fan or wheel, then $\delta_{H_{\Gamma}}$ is quadratic. 
\end{cor}

\begin{lemma}\label{Decompose gamma into fans and wheels}
Let $\Gamma$ be a finite simplicial graph such that $\Delta_{\Gamma}$ is a $2$-dimensional triangulated disk with $\dim_{I}(\Delta_{\Gamma})=0$. Then $\Gamma$ can be represented as a tree $T$: each vertex of $T$ represents a fan or a wheel; two vertices $v,w$ of $T$ are adjacent if the intersection of the graphs that are represented by $v$ and $w$ is an edge.  
\end{lemma}

\begin{proof}
We observe that for any interior vertex of $\Delta_{\Gamma}$, the induced subgraph on the interior vertex, together with its adjacent vertices, is a wheel. Also, note that if $\Delta_{\Gamma}$ has two interior vertices, then they are not connected. Otherwise, the edge that connects the two interior vertices would be an interior edge of $\Delta_{\Gamma}$, which contradicts our assumption. Thus, there are three types of edges of $\Delta_{\Gamma}$: (1) edges on $\partial\Delta_{\Gamma}$, (2) edges that connect interior vertices and vertices on $\partial\Delta_{\Gamma}$, and (3) edges that intersect $\partial\Delta_{\Gamma}$ at two vertices. Cutting along all the edges of type (3), we obtain connected components of $\Delta_{\Gamma}$ whose $1$-skeletons are wheels and fans (triangles). For each connected component, we assign a vertex to it; two vertices are connected if the corresponding connected components intersect in $\Delta_{\Gamma}$. Thus, we have the desire decomposition of $\Gamma$ with an underlying graph $T$.

Now we show that $T$ is a tree. Suppose $T$ is not a tree, then $T$ contains a circle $C_{n}$ of length $n$. The flag complex on the subgraph of $\Gamma$ whose decomposition corresponds to $C_{n}$ would be a triangulated annulus. Thus, $\Delta_{\Gamma}$ would not be a triangulate disk, and this contradicts the assumption that $\Delta_{\Gamma}$ is a triangulated disk. Hence, $T$ is a tree.
\end{proof}

Note that the decomposition in Proposition \ref{Decompose gamma into fans and wheels} is not unique. Since for a fan $\mathcal{F}_{n+1}$, there are $p(n)$ such decompositions, where $p(n)$ is the partition function.

Let $\Gamma=\Gamma_{1}\cup\Gamma_{2}$ be a finite simplicial graph. If $\Gamma_{1}\cap\Gamma_{2}$ is a single edge, then the Bestvina--Brady group $H_{\Gamma}$ splits over $\mathbb{Z}$ as $H_{\Gamma_{1}}\ast_{\mathbb{Z}}H_{\Gamma_{2}}$ by the Dicks--Leary presentation (see Theorem \ref{Dicks and Leary presentation}). It is not hard to see that Lemma \ref{Decompose gamma into fans and wheels} implies that $H_{\Gamma}$ splits over $\mathbb{Z}$, and the vertex groups are $\mathrm{CAT}(0)$ groups. We summarize these statements in the following proposition.

\begin{prop}\label{Decompose BBG as a graph of groups}
Let $\Gamma$ be a finite simplicial graph such that $\Delta_{\Gamma}$ is a triangulated disk with $\dim_{I}(\Delta_{\Gamma})=0$. Then $H_{\Gamma}$ has a graph-of-groups decomposition, where the underlying graph is the tree $T$ in Lemma \ref{Decompose gamma into fans and wheels}, the edge groups are $\mathbb{Z}$, and the vertex groups are Bestvina--Brady groups defined by the graphs that are represented by the vertex of $T$ in the Lemma \ref{Decompose gamma into fans and wheels}. Moreover, the vertex groups are right-angled Artin groups, then hence, $\mathrm{CAT}(0)$ groups.
\end{prop}

To prove our main theorem in this section, we need the following proposition in \cite{bridsonandhaefliger}:

\begin{prop}\textnormal{(\cite{bridsonandhaefliger}, Chapter II.11, 11.17 Proposition)}\label{Bridson 11.17}
If each of the groups $G_{1}$ and $G_{2}$ is the fundamental group of a non-positively curved compact metric space, then so is $G_{1}\ast_{\mathbb{Z}}G_{2}$. In particular, if $G_{1}$ and $G_{2}$ are $\mathrm{CAT}(0)$ groups, so is $G_{1}\ast_{\mathbb{Z}}G_{2}$.
\end{prop}

\begin{proof}[Proof of Theorem \ref{Dehn function of interior dimension 0}]
Let $\Gamma$ be a finite simplicial graph such that $\Delta_{\Gamma}$ is a $2$-dimensional triangulated disk with $\dim_{I}(\Delta_{\Gamma})=0$. By Proposition \ref{Decompose BBG as a graph of groups}, $H_{\Gamma}$ has a tree-of-groups decomposition, where the edge groups are $\mathbb{Z}$, and the vertex groups are $\mathrm{CAT}(0)$. Repeating Proposition \ref{Bridson 11.17} gives us that $H_{\Gamma}$ is a $\mathrm{CAT}(0)$ group. Thus, $\delta_{H_{\Gamma}}$ is at most quadratic. Since $\Delta_{\Gamma}$ is a triangulated disk, $H_{\Gamma}$ contains $\mathbb{Z}\times\mathbb{Z}$ subgroups, and thus, it is not hyperbolic. Therefore, $\delta_{H_{\Gamma}}$ is not linear. Hence, $\delta_{H_{\Gamma}}$ is quadratic. 
\end{proof}

\section{Disks with square boundaries}\label{section disks with square boundaries}

The goal of this section is to prove the following:

\begin{thm}\label{Dehn functions for square boundary}
Let $\Gamma$ be a finite simplicial graph such that $\Delta_{\Gamma}$ is a $2$-dimensional triangulated disk whose boundary is a square. If $\dim_{I}(\Delta_{\Gamma})=d$ for $d\in\lbrace 0,1,2 \rbrace$, then $\delta_{H_{\Gamma}}(n)\cong n^{d+2}$.
\end{thm}

When $d=0$, Theorem \ref{Dehn functions for square boundary} is a corollary of Theorem \ref{Dehn function of interior dimension 0}. In Subsection \ref{subsection lower bound}, we establish the lower bound for $\delta_{H_{\Gamma}}$ for $d=1,2$. Since there is a universal quartic upper bound given by Dison \cite{dison}, we only need to establish the cubic upper bound for $\delta_{H_{\Gamma}}$ for $d=1$; the detailed proof is contained in Subsection \ref{subsection upper bound}. Before we proceed with the proof, we prove the following lemma which will be used later on. 

\begin{lemma}\label{The graph of square boundary and d=1}
Let $\Gamma$ be a finite simplicial graph such that $\Delta_{\Gamma}$ is a $2$-dimensional triangulated disk with square boundary. If $\dim_{I}(\Delta_{\Gamma})=1$, then $\Gamma$ is the suspension of a path of length at least $3$.
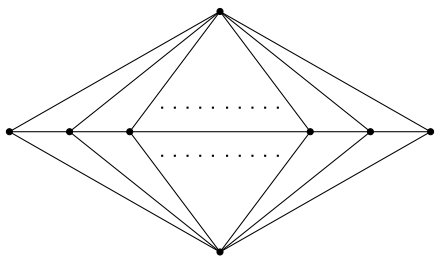
\begin{figure}[H]
\begin{center}
\begin{tikzpicture}[scale=0.4]
\draw (-7,0)--(7,0);

\draw (0,4)--(-7,0);
\draw (0,4)--(-5,0);
\draw (0,4)--(-3,0);
\draw (0,4)--(3,0);
\draw (0,4)--(5,0);
\draw (0,4)--(7,0);

\draw (0,-4)--(-7,0);
\draw (0,-4)--(-5,0);
\draw (0,-4)--(-3,0);
\draw (0,-4)--(3,0);
\draw (0,-4)--(5,0);
\draw (0,-4)--(7,0);

\draw [thick, loosely dotted] (-0.2,0.8)--(-2.3,0.8);
\draw [thick, loosely dotted] (0.2,0.8)--(2.3,0.8);
\draw [thick, loosely dotted] (-0.2,-0.8)--(-2.3,-0.8);
\draw [thick, loosely dotted] (0.2,-0.8)--(2.3,-0.8);

\draw [fill] (0,4) circle [radius=0.1];
\draw [fill] (-7,0) circle [radius=0.1];
\draw [fill] (-5,0) circle [radius=0.1];
\draw [fill] (-3,0) circle [radius=0.1];
\draw [fill] (3,0) circle [radius=0.1];
\draw [fill] (5,0) circle [radius=0.1];
\draw [fill] (7,0) circle [radius=0.1];
\draw [fill] (0,-4) circle [radius=0.1];

\end{tikzpicture}
\end{center}
\caption{The suspension of a path of length at least $3$.}
\label{The suspension of a path of length at least $3$.}
\end{figure}
\end{lemma}

\begin{proof}
Define $\Gamma'$ to be the graph whose edge set consists of all the interior edges of $\Delta_{\Gamma}$. Label the four vertices on $\partial\Delta_{\Gamma}$ by $a,b,c,d$ as shown in Figure \ref{two vertices on the boundary of Delta are not adjacent}. Since $\Delta_{\Gamma}$ is a triangulated disk, each edge of $\Gamma'$ together with two vertices on $\partial\Delta_{\Gamma}$ form two adjacent triangles, and these two vertices on $\partial\Delta_{\Gamma}$ are not adjacent. Otherwise, there would be a $K_{4}$ subgraph contained in $\Gamma$; see Figure \ref{two vertices on the boundary of Delta are not adjacent}.
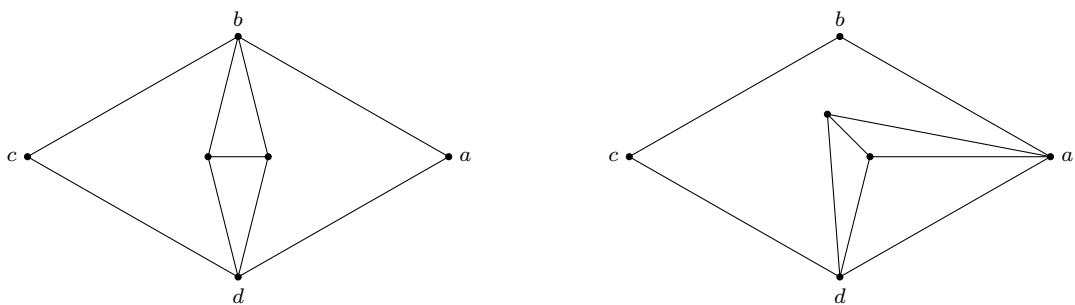
\begin{figure}[H]
\begin{center}
\begin{tikzpicture}[scale=0.4]
\draw (7,0)--(0,4)--(-7,0)--(0,-4)--(7,0);

\node [right] at (7,0) {\tiny $a$};
\node [above] at (0,4) {\tiny $b$};
\node [left] at (-7,0) {\tiny $c$};
\node [below] at (0,-4) {\tiny $d$};

\draw (-1,0)--(1,0);
\draw (0,4)--(-1,0);
\draw (0,4)--(1,0);
\draw (0,-4)--(-1,0);
\draw (0,-4)--(1,0);

\draw [fill] (0,4) circle [radius=0.1];
\draw [fill] (-7,0) circle [radius=0.1];
\draw [fill] (-1,0) circle [radius=0.1];
\draw [fill] (1,0) circle [radius=0.1];
\draw [fill] (7,0) circle [radius=0.1];
\draw [fill] (0,-4) circle [radius=0.1];

\draw (27,0)--(20,4)--(13,0)--(20,-4)--(27,0);

\node [right] at (27,0) {\tiny $a$};
\node [above] at (20,4) {\tiny $b$};
\node [left] at (13,0) {\tiny $c$};
\node [below] at (20,-4) {\tiny $d$};

\draw (19.5857864376,1.41421356237)--(21,0);
\draw (27,0)--(19.5857864376,1.41421356237);
\draw (27,0)--(21,0);
\draw (20,-4)--(19.5857864376,1.41421356237);
\draw (20,-4)--(21,0);

\draw [fill] (20,4) circle [radius=0.1];
\draw [fill] (13,0) circle [radius=0.1];
\draw [fill] (19.5857864376,1.41421356237) circle [radius=0.1];
\draw [fill] (21,0) circle [radius=0.1];
\draw [fill] (27,0) circle [radius=0.1];
\draw [fill] (20,-4) circle [radius=0.1];

\end{tikzpicture}
\end{center}
\caption{An interior edge together with two vertices on $\partial\Delta_{\Gamma}$ form two adjacent triangles. These two vertices on $\partial\Delta_{\Gamma}$ are not adjacent, see the picture on the left. Otherwise, $\Gamma$ would contain a $K_{4}$, see the picture on the right.}
\label{two vertices on the boundary of Delta are not adjacent}
\end{figure}

We make two claims on $\Gamma'$. The first claim is: $\Gamma'$ has no vertices whose valency is greater than $3$. Suppose $\Gamma'$ has a vertex whose valency is greater than $3$, that is, $\Gamma'$ contains a tripod. Since each edge of the tripod together with two non-adjacent vertices on $\partial\Delta_{\Gamma}$, say $b,d$, form two adjacent triangles, there are two triangles in $\Delta_{\Gamma}$ share two consecutive edges; see Figure \ref{Gamma' has a tripod.}. Thus, $\Gamma$ is not a simplicial graph, we get a contradiction. Hence, no vertex of $\Gamma'$ has valency greater than $3$. This proves the first claim.
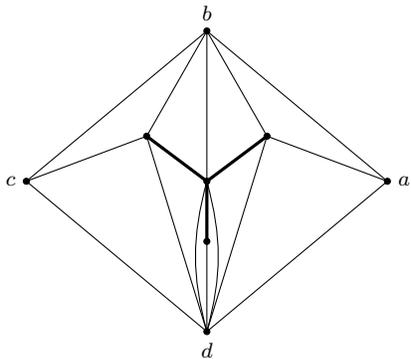
\begin{figure}[H]
\begin{center}
\begin{tikzpicture}[scale=0.4]

\draw (-1,0)--(5,5)--(11,0)--(5,-5)--(-1,0);

\node [right] at (11,0) {\tiny $a$};
\node [above] at (5,5) {\tiny $b$};
\node [left] at (-1,0) {\tiny $c$};
\node [below] at (5,-5) {\tiny $d$};

\draw [very thick] (3,1.5)--(5,0);
\draw [very thick] (7,1.5)--(5,0);
\draw [very thick] (5,-2)--(5,0);

\draw (3,1.5)--(5,5);
\draw (5,0)--(5,5);
\draw (7,1.5)--(5,5);

\draw (3,1.5)--(-1,0);
\draw (7,1.5)--(11,0);

\draw (3,1.5)--(5,-5);
\draw (5,0) to [out=255,in=105] (5,-5);

\draw (5,-2)--(5,-5);

\draw (7,1.5)--(5,-5);
\draw (5,0) to [out=285,in=75] (5,-5);

\draw [fill] (5,5) circle [radius=0.1];
\draw [fill] (3,1.5) circle [radius=0.1];
\draw [fill] (7,1.5) circle [radius=0.1];
\draw [fill] (-1,0) circle [radius=0.1];
\draw [fill] (5,0) circle [radius=0.1];
\draw [fill] (11,0) circle [radius=0.1];
\draw [fill] (5,-2) circle [radius=0.1];
\draw [fill] (5,-5) circle [radius=0.1];

\end{tikzpicture}
\end{center}
\caption{The picture shows the simplest case: $\Gamma'$ is a tripod, drown in thick lines. There are two edges connecting the central vertex of the tripod and the vertex $d$.}
\label{Gamma' has a tripod.}
\end{figure}

The second claim is that $\Gamma'$ is a connected path. Suppose $\Gamma'$ is not connected, say $\Gamma'$ has $k$ connected components $\Gamma'_{1},\cdots,\Gamma'_{k}$. Since none of the vertices of $\Gamma'$ has valency greater than $3$, each of $\Gamma'_{1},\cdots,\Gamma'_{k}$ is a connected path. Also, all the pairs of adjacent vertices of $\Gamma'$ are connected either to vertices $a,c$ or vertices $b,d$ simultaneously, say vertices $b,d$:
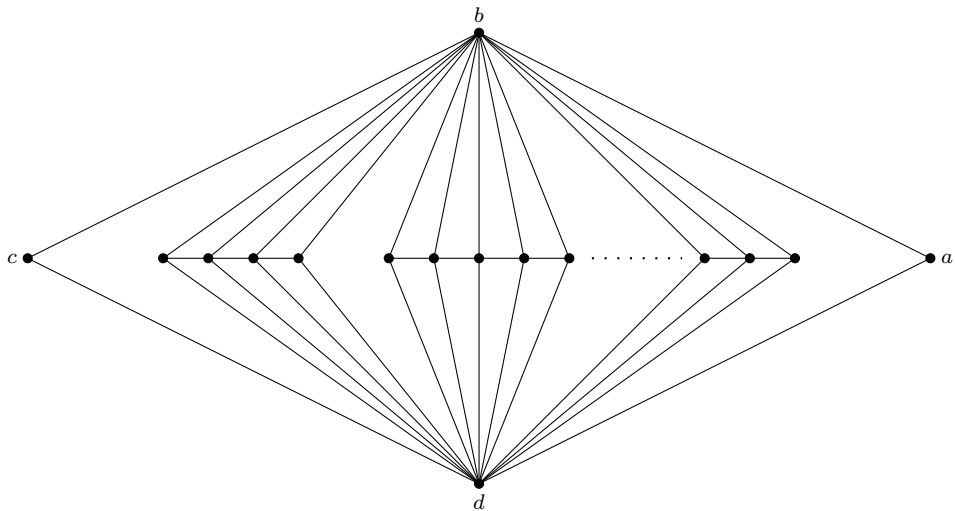
\begin{figure}[H]
\begin{center}
\begin{tikzpicture}[scale=0.6]

\draw (-10,0)--(0,5)--(10,0)--(0,-5)--(-10,0);

\node [right] at (10,0) {\tiny $a$};
\node [above] at (0,5) {\tiny $b$};
\node [left] at (-10,0) {\tiny $c$};
\node [below] at (0,-5) {\tiny $d$};

\draw [fill] (-10,0) circle [radius=0.1];
\draw [fill] (0,5) circle [radius=0.1];
\draw [fill] (10,0) circle [radius=0.1];
\draw [fill] (0,-5) circle [radius=0.1];

\draw (-7,0)--(-4,0);

\draw (0,5)--(-7,0);
\draw (0,5)--(-6,0);
\draw (0,5)--(-5,0);
\draw (0,5)--(-4,0);
\draw (0,-5)--(-7,0);
\draw (0,-5)--(-6,0);
\draw (0,-5)--(-5,0);
\draw (0,-5)--(-4,0);

\draw [fill] (-7,0) circle [radius=0.1];
\draw [fill] (-6,0) circle [radius=0.1];
\draw [fill] (-5,0) circle [radius=0.1];
\draw [fill] (-4,0) circle [radius=0.1];

\draw (-2,0)--(2,0);

\draw (0,5)--(-2,0);
\draw (0,5)--(-1,0);
\draw (0,5)--(0,0);
\draw (0,5)--(1,0);
\draw (0,5)--(2,0);
\draw (0,-5)--(-2,0);
\draw (0,-5)--(-1,0);
\draw (0,-5)--(0,0);
\draw (0,-5)--(1,0);
\draw (0,-5)--(2,0);

\draw [fill] (-2,0) circle [radius=0.1];
\draw [fill] (-1,0) circle [radius=0.1];
\draw [fill] (0,0) circle [radius=0.1];
\draw [fill] (1,0) circle [radius=0.1];
\draw [fill] (2,0) circle [radius=0.1];

\draw [thick, loosely dotted] (2.5,0)--(4.5,0);

\draw (5,0)--(7,0);

\draw (0,5)--(5,0);
\draw (0,5)--(6,0);
\draw (0,5)--(7,0);
\draw (0,-5)--(5,0);
\draw (0,-5)--(6,0);
\draw (0,-5)--(7,0);

\draw [fill] (5,0) circle [radius=0.1];
\draw [fill] (6,0) circle [radius=0.1];
\draw [fill] (7,0) circle [radius=0.1];
\end{tikzpicture}
\end{center}
\caption{Connected components of $\Gamma'$. Each pair of adjacent vertices of $\Gamma'$ are connected to a pair of non-adjacent vertices on $\partial\Delta_{\Gamma}$.}
\label{Gamma' with k components}
\end{figure}

In Figure \ref{Gamma' with k components}, label from the left-most connected component to the right-most component of $\Gamma'$ by $\Gamma'_{1},\cdots,\Gamma'_{k}$, and define $\Gamma_{i}$ to be the join of $\Gamma'_{i}$ and $\lbrace b,d\rbrace$. Since $\Delta_{\Gamma}$ is a triangulated disk, there must be an edge connecting $b$ and $d$ between $\Gamma_{i}$ and $\Gamma_{i+1}$, $i=1,\cdots k-1$:
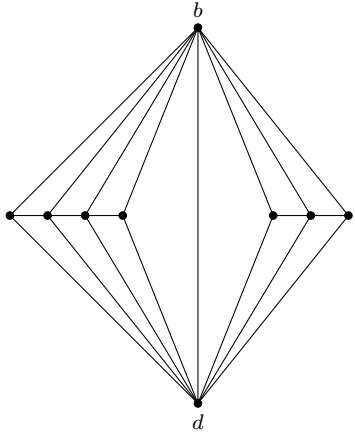
\begin{figure}[H]
\begin{center}
\begin{tikzpicture}[scale=0.5]

\node [above] at (5,5) {\tiny $b$};
\node [below] at (5,-5) {\tiny $d$};

\draw [fill] (5,5) circle [radius=0.1];
\draw [fill] (5,-5) circle [radius=0.1];

\draw (0,0)--(3,0);
\draw (5,5)--(0,0);
\draw (5,5)--(1,0);
\draw (5,5)--(2,0);
\draw (5,5)--(3,0);
\draw (5,-5)--(0,0);
\draw (5,-5)--(1,0);
\draw (5,-5)--(2,0);
\draw (5,-5)--(3,0);

\draw [fill] (0,0) circle [radius=0.1];
\draw [fill] (1,0) circle [radius=0.1];
\draw [fill] (2,0) circle [radius=0.1];
\draw [fill] (3,0) circle [radius=0.1];

\draw (5,5)--(5,-5);

\draw (7,0)--(9,0);
\draw (5,5)--(7,0);
\draw (5,5)--(8,0);
\draw (5,5)--(9,0);
\draw (5,-5)--(7,0);
\draw (5,-5)--(8,0);
\draw (5,-5)--(9,0);

\draw [fill] (7,0) circle [radius=0.1];
\draw [fill] (8,0) circle [radius=0.1];
\draw [fill] (9,0) circle [radius=0.1];

\end{tikzpicture}
\end{center}
\caption{Between $\Gamma_{i}$ and $\Gamma_{i+1}$, vertices $b,d$ are connected by an edge.}
\label{Edge between Gamma_i and Gamma_i+1}
\end{figure}
\noindent In Figure \ref{Edge between Gamma_i and Gamma_i+1} we see that connecting $b,d$ by an edge creates $K_{4}$ subgraphs in $\Gamma$. This contradicts our assumption. Thus, the graph $\Gamma'$ is a connected path. 

Now, we have shown that $\Gamma'$ is a connected path, and all pairs of adjacent vertices of $\Gamma'$ are connected either to vertices $a,c$ or $b,d$ simultaneously. Again, suppose $b,d$ are connected to all vertices of $\Gamma'$. The two end vertices of $\Gamma'$ have no choices but to be connected to $a$ and $c$. Hence, $\Gamma$ is a suspension of $\Gamma'$. This proves the lemma. 
\end{proof}

\begin{remark}\label{a family of graph that have cubic Dehn function}
In \cite{bradyrileyshortthegeometryofthewordproblemforfinitelygeneratedgroups}, Brady gave an example of a cubic Dehn function $\delta_{H_{\Gamma}}$, where $\Gamma$ is the suspension of a path of length $3$. 	Theorem \ref{Dehn functions for square boundary} plus Lemma \ref{The graph of square boundary and d=1} gives a concrete family of graphs $\Gamma$ satisfying $\delta_{H_{\Gamma}}(n)\simeq n^{3}$. 
\end{remark}

\subsection{Lower bound}\label{subsection lower bound} 

In \cite{abddy}, the authors introduced \emph{height-pushing maps} to give a lower bound on the higher dimensional Dehn functions of \emph{orthoplex groups}. The same technique can be used here to obtain the desired lower bound. We give necessary background here; see \cite{abddy} for more details.

Let $\Gamma$ be a finite simplical graph. We use the notations $X_{\Gamma}$, $\widetilde{X}_{\Gamma}$, $Z_{\Gamma}$, $h$ from Subsection ~\ref{subsection RAAG, BBG, DLP}. Recall that $\widetilde{X}_{\Gamma}$ is a $\mathrm{CAT}(0)$ cube complex. We put the $\ell^{1}$-metric on each cube and extend it to the whole space $\widetilde{X}_{\Gamma}$. Let $B_{r}(x)$ and $S_{r}(x)$ be the $\ell^{1}$-ball and $\ell^{1}$-sphere in $\widetilde{X}_{\Gamma}$ centered at $x$ with radius $r$, respectively. 

A subspace $F\subseteq\widetilde{X}_{\Gamma}$ is called a \emph{$k$-flat} if it is isometric to the Euclidean space $\mathbb{E}^{k}$. 

\begin{thm}\textnormal{(\cite{abddy}, Theorem 4.2)} \label{height-pushing map}
There is an $H_{\Gamma}$-equivariant retraction, called the \emph{height-pushing map}
\begin{align*}
\mathbf{P}:\widetilde{X}_{\Gamma}\setminus\bigcup\limits_{v\notin Z_{\Gamma}}B_{1/4}(v)\rightarrow Z_{\Gamma},
\end{align*}
where the union is over the vertices that are not in $Z_{\Gamma}$, such that for all $t>0$, when $\mathbf{P}$ is restricted to $h^{-1}([-t,t])$, $\mathbf{P}$ is a $(ct+c)$-Lipschitz map, where $c$ is a constant which only depends on $\Gamma$.
\end{thm}

In the next lemma, we establish the lower bound for $\delta_{H_{\Gamma}}$ in Theorem \ref{Dehn functions for square boundary}.

\begin{lemma}\label{lower bound of square boundary}
Let $\Gamma$ be a finite simplicial graph such that $\Delta_{\Gamma}$ is a $2$-dimensional triangulated disk with square boundary. If $\dim_{I}(\Delta_{\Gamma})=d$, then $\delta_{H_{\Gamma}}(n)\succcurlyeq n^{d+2}$.
\end{lemma}

We give the idea of the proof of Lemma \ref{lower bound of square boundary}. Recall that Bestvina--Brady groups $H_{\Gamma}$ act geometrically on the zero level set $Z_{\Gamma}=h^{-1}(0)$. So the filling function of $Z_{\Gamma}$ is equivalent to $\delta_{H_{\Gamma}}$. For each $r$, we construct a loop $S_{r}$ in $Z_{\Gamma}$ whose length is $\simeq r$, and it bounds a disk in the ambient space $\widetilde{X}_{\Gamma}$ and a disk in $Z_{\Gamma}$. Then we use the height-pushing map to push the filling of the disk in $\widetilde{X}_{\Gamma}$ to get a lower bound of the filling of the loop $S_{r}$ in $Z_{\Gamma}$.

\begin{proof}[Proof of Lemma \ref{lower bound of square boundary}]
The case $d=0$ follows from Theorem \ref{Dehn function of interior dimension 0}; the case $d=2$ appears in \cite{abddy}. We prove the case when $d=1$ which uses the same strategy as the case $d=2$. When $d=1$, recall from Lemma \ref{The graph of square boundary and d=1} that $\Gamma$ is the suspension of a path of length at least $3$. Label the four vertices on the boundary of $\Delta_{\Gamma}$ as follows:
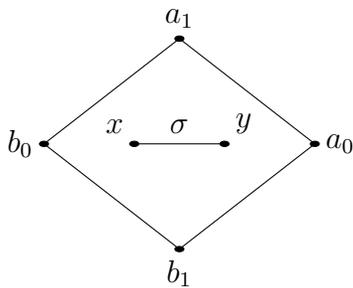
\begin{figure}[H]
\begin{center}
\begin{tikzpicture}[xscale=1.2, yscale=0.7]
\draw (1.5,0)--(0,2)--(-1.5,0)--(0,-2)--(1.5,0);
\draw (-0.5,0)--(0.5,0);

\draw [fill] (1.5,0) circle [radius=0.05];
\node [right] at (1.5,0) {$a_{0}$};

\draw [fill] (-1.5,0) circle [radius=0.05];
\node [left] at (-1.5,0) {$b_{0}$};

\draw [fill] (0,2) circle [radius=0.05];
\node [above] at (0,2) {$a_{1}$};

\draw [fill] (0,-2) circle [radius=0.05];
\node [below] at (0,-2) {$b_{1}$};

\draw [fill] (-0.5,0) circle [radius=0.05];
\node [above left] at (-0.5,0) {$x$};

\node [above] at (0,0) {$\sigma$};

\draw [fill] (0.5,0) circle [radius=0.05];
\node [above right] at (0.5,0) {$y$};

\end{tikzpicture}
\end{center}
\caption{The square boundary of $\Delta_{\Gamma}$ and an interior $1$-simplex $\sigma$.}
\label{The square boundary of a $2$-dimensional disk and an interior $1$-simplex.}
\end{figure}

Define bi-infinite geodesic rays $\gamma_{i}$, $i=0,1$, in the $1$-skeleton of $\widetilde{X}_{\Gamma}$:
\begin{align*}
\gamma_{i}(0)=e, \ \ \ 
\gamma_{i}\vert_{\mathbb{R}^{+}}=a_{i}b_{i}a_{i}b_{i}\cdots, \ \ \ 
\gamma_{i}\vert_{\mathbb{R}^{-}}=b_{i}a_{i}b_{i}a_{i}\cdots.
\end{align*}
Define a map $F:\mathbb{Z}\times\mathbb{Z}\rightarrow A_{\Gamma}$ by 
\begin{align*}
F(x)=\gamma_{0}(x_{0})\gamma_{1}(x_{1}), \ \ \ x=(x_{0},x_{1})\in\mathbb{Z}\times\mathbb{Z}.
\end{align*}
The image of $F$ consists of elements in the non-abelian group $\langle a_{0},a_{1},b_{0},b_{1}\rangle$. Let $\overline{F}$ be the non-standard $2$-dimensional flat in $\widetilde{X}_{\Gamma}$ such that the set of its vertices is the image of $F$; see Figure \ref{The non-standard 2-dimensional flat}. Since 
\begin{align*}
h(x)=h(F(x_{0},x_{1}))=|x_{0}|+|x_{1}|, \ \ x=(x_{0},x_{1})\in\mathbb{Z}\times\mathbb{Z},
\end{align*}
the flat $\overline{F}$ has a unique vertex $F(0,0)=\gamma_{0}(0)\gamma_{1}(0)=e$ at height $0$. We think of $\overline{F}$ as the boundary of a reversed infinite square pyramid with the apex $F(0,0)$. For each $r>0$, the intersection $\overline{F}\cap h^{-1}(r)$ is homeomorphic to $S^{1}$, which is a loop $L_{r}$ at height $r$. This loop $L_{r}$ bounds a $2$-disk $\overline{F}\cap h^{-1}([0,r])$ in $\overline{F}$.
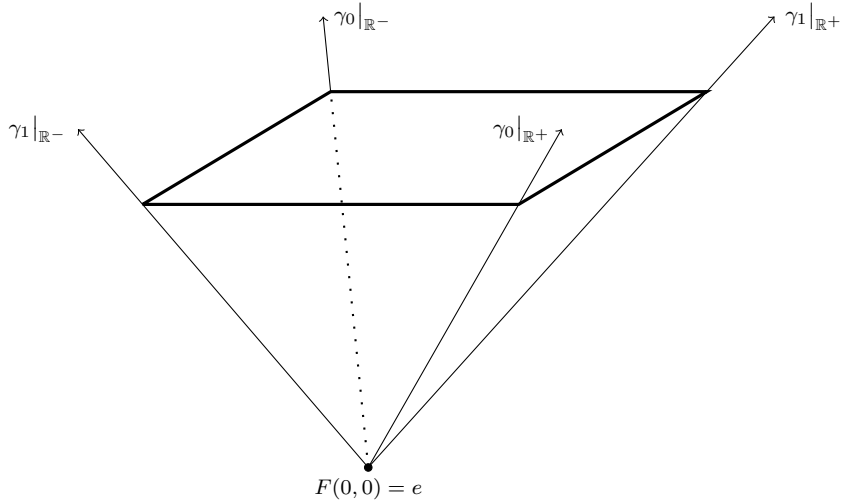
\begin{figure}[H]
\begin{center}
\begin{tikzpicture}[scale=0.5]


\draw [very thick] (0,0)--(10,0)--(15,3)--(5,3)--(0,0);

\draw [->] (6,-7)--(-1.71428571428,2);
\node [left] at (-1.71428571428,2) {\tiny $\gamma_{1}\big\vert_{\mathbb{R}^{-}}$};

\draw [->] (6,-7)--(11.1428571428,2);
\node [left] at (11.1428571428,2) {\tiny $\gamma_{0}\big\vert_{\mathbb{R}^{+}}$};

\draw [thick, loosely dotted] (6,-7)--(5,3);
\draw [->] (5,3)--(4.8,5);
\node [right] at (4.8,5) {\tiny $\gamma_{0}\big\vert_{\mathbb{R}^{-}}$};

\draw [->](6,-7)--(16.8,5);
\node [right] at (16.8,5) {\tiny $\gamma_{1}\big\vert_{\mathbb{R}^{+}}$};

\node [below] at (6,-7) {\tiny $F(0,0)=e$};
\draw [fill] (6,-7) circle [radius=0.1];

\end{tikzpicture}
\end{center}
\caption{The non-standard $2$-dimensional flat $\overline{F}$. The parallelogram drown in the thick lines is the intersection $\overline{F}\cap h^{-1}(r)$.}
\label{The non-standard 2-dimensional flat}
\end{figure}
\noindent Using the group action $A_{\Gamma}$ on $\widetilde{X}_{\Gamma}$, we translate the $2$-disk $\overline{F}\cap h^{-1}([0,r])$ to $(a_{0})^{-r}\overline{F}\cap h^{-1}([-r,0])$. We denote the new $2$-disk by $D_{r}$ and its interior by $\mathring{D}_{r}$. The loop $L_{r}$ at height $r$ is translated to a loop $S_{r}$ in $Z_{\Gamma}$ that is also homeomorphic to $S^{1}$:
\begin{align*}
S_{r}:=[(a_{0})^{-r}\overline{F}]\cap h^{-1}(0)=[(a_{0})^{-r}\overline{F}]\cap Z_{\Gamma}.
\end{align*}  
Notice that after performing this translation, the unique vertex $(a_{0})^{-r}F(0,0)$ is at height $-r$. Now we have that $S_{r}$ is a loop in $Z_{\Gamma}$ and it bounds a $2$-disk $D_{r}$. The next step is to use the height-pushing map $\mathbf{P}$ to push this filling $D_{r}$ of $S_{r}$ in $\widetilde{X}_{\Gamma}$ to fill $S_{r}$ in $Z_{\Gamma}$. 

In order to use the height-pushing map, we need to do some surgery on $D_{r}$. At each vertex $v\in\mathring{D}_{r}$, replace $B_{1/4}(v)$ by a copy of $\Delta_{\Gamma}$. We now show that after applying the height-pushing map to the surgered $\mathring{D}_{r}$, the scaled copies of $1$-simplices in the interior of $\Delta_{\Gamma}$ at each vertex of the surgered $\mathring{D}_{r}$ do not intersect much in $Z_{\Gamma}$. We first consider two different $1$-simplices $\sigma_{1}$ and $\sigma_{2}$ in the interior of $\Delta_{\Gamma}$ at a vertex $v$ in $\mathring{D}_{r}$. Since $\sigma_{1}$ and $\sigma_{2}$ either are disjoint or intersect at a $0$-simplex, the images $\mathbf{P}(\sigma_{1})$ and $\mathbf{P}(\sigma_{2})$ intersect at most at a $0$-simplex. We next consider $1$-simplices in the interior of $\Delta_{\Gamma}$ at different vertices $v_{1},v_{2}$ in $\mathring{D}_{r}$. Let $\sigma=[x,y]$ be an $1$-simplex in the interior of $\Delta_{\Gamma}$; see Figure \ref{The square boundary of a $2$-dimensional disk and an interior $1$-simplex.}. Let $v'_{1}$ and $v'_{2}$ be vertices of the scaled copies $\mathbf{P}(\sigma)$ in $Z_{\Gamma}$, based at vertices $v_{1}$ and $v_{2}$ in $\mathring{D}_{r}$, respectively. Then we have
\begin{align*}
v'_{1}=(a_{0})^{-r}v_{1}x^{s_{1}}y^{t_{1}} \ \ \text{and} \ \ v'_{2}=(a_{0})^{-r}v_{2}x^{s_{2}}y^{t_{2}}
\end{align*}
for some $s_{1},t_{2},s_{2},t_{2}\in\mathbb{Z}_{\geq0}$. If $v'_{1}=v'_{2}$, that is, the intersection of scaled copies $\mathbf{P}(\sigma)$ based at $v_{1}$ and $v_{2}$ is non-empty, then we have $v_{1}=v_{2}$ since $\langle a_{0}, a_{1}, b_{0}, b_{1}\rangle\cap\langle x,y\rangle=\lbrace0\rbrace$. Thus, no vertices of different scaled copies of $\sigma$ are the same. That is, scaled copies of $\sigma$ in $Z_{\Gamma}$ are disjoint. 
Each interior $1$-simplex is an intersection of two $2$-simplices in $\Delta_{\Gamma}$, so the image of the surgered $\mathring{D}_{r}$ under $\mathbf{P}$ has the area at least $2\cdot|\lbrace\text{number of $1$-simplices}\rbrace|$. Since we only need a lower bound, we don't need to count all the $1$-simplices. For each vertex in $\mathring{D}_{r}$, we take one $1$-simplex $\sigma$ in $\Delta_{\Gamma}$ based at that vertex. As we have seen that the scaled copies $\mathbf{P}(\sigma)$ don't cancel out. So the filling of $S_{r}$ in $Z_{\Gamma}$ will be at least the sum of the length of these scaled copies $\mathbf{P}(\sigma)$. For each interior $1$-simplex $\sigma$ in $\Delta_{\Gamma}$ based at a vertex $v$ of $\mathring{D}_{r}$, it follows by Theorem \ref{height-pushing map} that the length of the scaled copied $\mathbf{P}(\sigma)$ in $Z_{\Gamma}$ grows linearly in terms of $r$, that is, the length of $\mathbf{P}(\sigma)$ depends on how far is $\sigma$ pushed into $Z_{\Gamma}$ by the height-pushing map. Thus, the further the $\sigma$ is from $Z_{\Gamma}$, the longer the length of $\mathbf{P}(\sigma)$. At height $-r$, there is only one vertex in $\mathring{D}_{r}$, namely, the apex $(a_{0})^{-r}F(0,0)$. The image of the $1$-simplex $\sigma$ based at the apex under $\mathbf{P}$ is $cr+c$, by Theorem \ref{height-pushing map}. At each height $i=-(r-1),\cdots,-1$, there are $4i$ vertices and the image of $1$-simplex $\sigma$ under $\mathbf{P}$ has length $c(r-i)+c$. We have
\begin{align*}
\text{Area}(S_{r})\geq(cr+c)+\sum^{r-1}_{i=1}4i[c(r-i)+c]\simeq r^{3}.
\end{align*}
Thus, the filling function of $Z_{\Gamma}$ is at least cubic. Hence, $\delta_{H_{\Gamma}}$ is at least cubic. This proves the lemma. 
\end{proof}

\subsection{Upper bound}\label{subsection upper bound}

As we mentioned at the beginning of this section, we only need to establish the cubic upper bound for $d=1$ in Theorem \ref{Dehn functions for square boundary}. The proof relies on analyzing a van Kampen diagram of an arbitrary word $w\in H_{\Gamma}$ that represents the identity. All the van Kampen diagrams in this section are assumed to be minimal. The tool that we will be using is \emph{corridor schemes}.  Here we only give the definition; more details can be found in \cite{bradyandforester}. Let $\Gamma$ be a finite simplicial graph with labelings on the edges such that $\Delta_{\Gamma}$ is a $2$-dimensional triangulated disk. A \emph{corridor scheme} for $\Gamma$ is a collection $\tau$ of labeled edges of $\Gamma$ such that every triangle of $\Delta_{\Gamma}$ has either zero or two edges in $\tau$. Given a van Kampen diagram $\Delta$, a \emph{$\tau$-corridor} is a corridor in $\Delta$ that consists of triangles; each triangle has exactly two edges in $\sigma$, and every pair of adjacent triangles intersect at an edge in $\sigma$. If $\alpha$ and $\beta$ are two disjoint corridor schemes, then $\alpha$-corridors and $\beta$-corridors never cross each other. If $\Gamma$ is oriented, then each triangle in a diagram $\Delta$ is orientated. 

Recalling the fact we only have to establish the cubic upper bound for $d=1$ in Theorem \ref{Dehn functions for square boundary} and Lemma \ref{The graph of square boundary and d=1}, the graph $\Gamma$ in the rest of this subsection will be the suspension of a path of length at least $3$ with orientations and labels on the edges given by Figure \ref{The suspension of a path of length at least $3$ with orientation.}

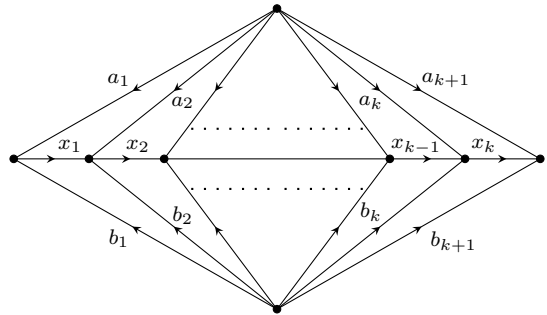
\begin{figure}[H]
\begin{center}
	\begin{tikzpicture}[scale=0.5]
	
	\draw [middlearrow={stealth}] (-7,0)--(-5,0);
	\draw [middlearrow={stealth}] (-5,0)--(-3,0);
	\draw (-3,0)--(0,0);
	\draw (0,0)--(3,0);
	\draw [middlearrow={stealth}] (3,0)--(5,0);
	\draw [middlearrow={stealth}] (5,0)--(7,0);
	
	\draw [middlearrow={stealth}] (0,4)--(-7,0);
	\draw [middlearrow={stealth}] (0,4)--(-5,0);
	\draw [middlearrow={stealth}] (0,4)--(-3,0);
	\draw [middlearrow={stealth}] (0,4)--(3,0);
	\draw [middlearrow={stealth}] (0,4)--(5,0);
	\draw [middlearrow={stealth}] (0,4)--(7,0);
	
	\draw [middlearrow={stealth}] (0,-4)--(-7,0);
	\draw [middlearrow={stealth}] (0,-4)--(-5,0);
	\draw [middlearrow={stealth}] (0,-4)--(-3,0);
	\draw [middlearrow={stealth}] (0,-4)--(3,0);
	\draw [middlearrow={stealth}] (0,-4)--(5,0);
	\draw [middlearrow={stealth}] (0,-4)--(7,0);
	
	\draw [thick, loosely dotted] (-0.2,0.8)--(-2.3,0.8);
	\draw [thick, loosely dotted] (0.2,0.8)--(2.3,0.8);
	\draw [thick, loosely dotted] (-0.2,-0.8)--(-2.3,-0.8);
	\draw [thick, loosely dotted] (0.2,-0.8)--(2.3,-0.8);

	\draw [fill] (0,4) circle [radius=0.1];
	\draw [fill] (-7,0) circle [radius=0.1];
	\draw [fill] (-5,0) circle [radius=0.1];
	\draw [fill] (-3,0) circle [radius=0.1];
	\draw [fill] (3,0) circle [radius=0.1];
	\draw [fill] (5,0) circle [radius=0.1];
	\draw [fill] (7,0) circle [radius=0.1];
	\draw [fill] (0,-4) circle [radius=0.1];

	\node at (-4.2,2.1) {\tiny $a_{1}$};
	\node at (-2.5,1.5) {\tiny $a_{2}$};
	\node at (2.5,1.5) {\tiny $a_{k}$};
	\node at (4.5,2.1) {\tiny $a_{k+1}$};

	\node at (-5.5,0.35) {\tiny $x_{1}$};
	\node at (-3.7,0.35) {\tiny $x_{2}$};
	\node at (3.7,0.32) {\tiny $x_{k-1}$};
	\node at (5.5,0.35) {\tiny $x_{k}$};
	
	\node at (-4.2,-2.1) {\tiny $b_{1}$};
	\node at (-2.5,-1.5) {\tiny $b_{2}$};
	\node at (2.5,-1.4) {\tiny $b_{k}$};
	\node at (4.7,-2.2) {\tiny $b_{k+1}$};
	
	\end{tikzpicture}
\end{center}
\caption{The graph $\Gamma$ with orientations and labels.}
\label{The suspension of a path of length at least $3$ with orientation.}
\end{figure}
The Dicks--Leary presentation (see Theorem \ref{Dicks and Leary presentation}) for $H_{\Gamma}$ is $\langle\mathcal{A}, \mathcal{B}, \mathcal{X}\vert\mathcal{R}\rangle$, where

\begin{align*}
\mathcal{A}&=\lbrace a_{1},\cdots,a_{k+1}\rbrace  \\
\mathcal{B}&=\lbrace b_{1},\cdots,b_{k+1}\rbrace  \\
\mathcal{X}&=\lbrace x_{1},\cdots,x_{k}\rbrace    \\
\mathcal{R}&=\lbrace a_{i}x_{i}=a_{i+1}=x_{i}a_{i}, \  b_{i}x_{i}=b_{i+1}=x_{i}b_{i}\rbrace, \ i=1,\cdots,k 
\end{align*}

Let $w$ be a freely reduced word of length at most $n$ that represents the identity in $H_{\Gamma}$ and $\Delta$ a minimal van Kampen diagram for $w$. Choose two disjoint corridor schemes $\alpha=\lbrace a_{1},\cdots,a_{k+1}\rbrace$ and $\beta=\lbrace b_{1},\cdots,b_{k+1}\rbrace$ for $\Gamma$. Note that $\alpha$-corridors and $\beta$-corridors never cross each other. For each corridor in $\Delta$, fix a vertex $p$ of the corridor that is on $\partial\Delta$. When reading along the boundary of a corridor starting from $p$, we get a boundary word $u'v''u''v'$ or its cyclic permutation, where $u',u''$ are letters in $\alpha$ or $\beta$; $v',v''$ are words in the free group $F(x_{1},\cdots,x_{k})$. Note that $u'$ and $u''$ are labelings of edges on $\partial\Delta$.

\begin{figure}[H]
\begin{center}
\begin{tikzpicture}[scale=0.5]
\draw (0,0)--(10,0)--(10,2)--(0,2)--(0,0);

\node [below] at (5,0) {\tiny $v''$};
\node [above] at (5,2) {\tiny $v'$};
\node [left] at (0,1) {\tiny $u'$};
\node [right] at (10,1) {\tiny $u''$};

\draw [fill] (0,2) circle [radius=0.1];

\node [above left] at (0,2) {\tiny $p$};

\end{tikzpicture}
\end{center}
\caption{Boundary of a single corridor.}
\label{boundary of a single corridor}
\end{figure}
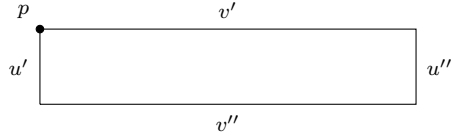

\begin{remark}
Since $\Gamma$ is an oriented simplicial graph, all the corridors and triangles of a diagram $\Delta$ are oriented. For the pictures in the rest of this subsection, we omit the arrows on some of the edges for simplicity. Since each edge is labeled with a letter, and consecutive edges form a path, we don't distinguish letters and edges, and words and paths.  
\end{remark}

\begin{definition}\label{definition of a stack}
Let $C_{i}$, $1\leq i\leq h$, be a corridor in $\Delta$ with the boundary word $u'_{i}v''_{i}u''_{i}v'_{i}$ (or its cyclic permutation) where $u'_{i},u''_{i}$ are letters in $\alpha$ or $\beta$; $v'_{i},v''_{i}$ are words in the free group $F(x_{1},\cdots,x_{k})$; see Figure \ref{boundary of a single corridor}. If $u'_{1},\cdots,u'_{h}$ and $u''_{1},\cdots,u''_{h}$ are two sets of consecutive letters on $\partial\Delta$ and $v''_{i}=v'_{i+1}$ for $i=1,\cdots h-1$, then the subdiagram $T$ of $\Delta$ obtained by gluing corridors $C_{1},\cdots,C_{h}$ along the words $v''_{1}=v'_{2},\cdots, v''_{h-1}=v'_{h}$ is called a \emph{stack}. The shorter word of the words $v'_{1}$ and $v''_{h}$ is called the \emph{top} of $T$; the longer word of the words $v'_{1}$ and $v''_{h}$ is called the \emph{bottom} of $T$. If $v'_{1}$ and $v''_{h}$ have the same length, then we call either one the top and the other one the bottom. The words $u'_{1}\cdots u'_{h}$ and $u''_{1}\cdots u''_{h}$ are called the \emph{legs} of $T$, and the number $h$ is called the \emph{height} of $T$.
\end{definition}

\begin{remark}
For two consecutive corridors $C_{i}$ and $C_{i+1}$ as defined in the Definition \ref{definition of a stack}, every $2$-cell of $\Delta$ with an edge labeled with a letter in the word $v''_{i}=v'_{i+1}$ must be part of either $C_{i}$ or $C_{i+1}$. Otherwise, it would violate the minimality assumption on $\Delta$.
\end{remark}

Roughly speaking, a stack is a pile of corridors where one corridor sits on top of another. Note that the top and bottom of a stack are not parts of $\partial\Delta$, and they might have different lengths; the legs are parts of $\partial\Delta$, and they are not necessarily straight. Figure \ref{A simple example of a stack} shows an example of a stack. For simplicity, we draw a stack with straight legs.

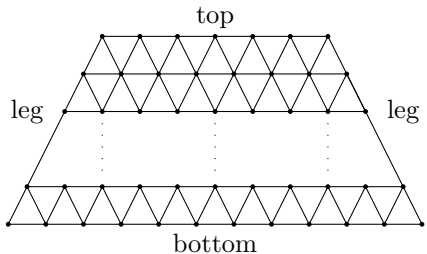
\begin{figure}[H]
\begin{center}
\begin{tikzpicture}[scale=0.5]
\draw (0,0)--(6,0);

\draw (0,0)--(-0.5,-1);
\draw (6,0)--(6.5,-1);

\draw (-0.5,-1)--(6.5,-1);

\draw (-0.5,-1)--(-1,-2);
\draw (6.5,-1)--(7,-2);

\draw (-1,-2)--(7,-2);

\draw (-1,-2)--(-2,-4);
\draw (7,-2)--(8,-4);

\draw (-2,-4)--(8,-4);

\draw (-2,-4)--(-2.5,-5);
\draw (8,-4)--(8.5,-5);

\draw (-2.5,-5)--(8.5,-5);

\draw [fill] (0,0) circle [radius=0.05];
\draw [fill] (1,0) circle [radius=0.05];
\draw [fill] (2,0) circle [radius=0.05];
\draw [fill] (3,0) circle [radius=0.05];
\draw [fill] (4,0) circle [radius=0.05];
\draw [fill] (5,0) circle [radius=0.05];
\draw [fill] (6,0) circle [radius=0.05];

\draw (0,0)--(0.5,-1);
\draw (0.5,-1)--(1,0);
\draw (1,0)--(1.5,-1);
\draw (1.5,-1)--(2,0);
\draw (2,0)--(2.5,-1);
\draw (2.5,-1)--(3,0);
\draw (3,0)--(3.5,-1);
\draw (3.5,-1)--(4,0);
\draw (4,0)--(4.5,-1);
\draw (4.5,-1)--(5,0);
\draw (5,0)--(5.5,-1);
\draw (5.5,-1)--(6,0);

\draw [fill] (-0.5,-1) circle [radius=0.05];
\draw [fill] (0.5,-1) circle [radius=0.05];
\draw [fill] (1.5,-1) circle [radius=0.05];
\draw [fill] (2.5,-1) circle [radius=0.05];
\draw [fill] (3.5,-1) circle [radius=0.05];
\draw [fill] (4.5,-1) circle [radius=0.05];
\draw [fill] (5.5,-1) circle [radius=0.05];
\draw [fill] (6.5,-1) circle [radius=0.05];

\draw (-0.5,-1)--(0,-2);
\draw (0,-2)--(0.5,-1);
\draw (0.5,-1)--(1,-2);
\draw (1,-2)--(1.5,-1);
\draw (1.5,-1)--(2,-2);
\draw (2,-2)--(2.5,-1);
\draw (2.5,-1)--(3,-2);
\draw (3,-2)--(3.5,-1);
\draw (3.5,-1)--(4,-2);
\draw (4,-2)--(4.5,-1);
\draw (4.5,-1)--(5,-2);
\draw (5,-2)--(5.5,-1);
\draw (5.5,-1)--(6,-2);
\draw (6,-2)--(6.5,-1);
\draw (6.5,-1)--(7,-2);

\draw [fill] (-1,-2) circle [radius=0.05];
\draw [fill] (0,-2) circle [radius=0.05];
\draw [fill] (1,-2) circle [radius=0.05];
\draw [fill] (2,-2) circle [radius=0.05];
\draw [fill] (3,-2) circle [radius=0.05];
\draw [fill] (4,-2) circle [radius=0.05];
\draw [fill] (5,-2) circle [radius=0.05];
\draw [fill] (6,-2) circle [radius=0.05];
\draw [fill] (7,-2) circle [radius=0.05];

\draw [loosely dotted] (0,-2)--(0,-3.8);
\draw [loosely dotted] (3,-2)--(3,-3.8);
\draw [loosely dotted] (6,-2)--(6,-3.8);

\draw [fill] (-2,-4) circle [radius=0.05];
\draw [fill] (-1,-4) circle [radius=0.05];
\draw [fill] (0,-4) circle [radius=0.05];
\draw [fill] (1,-4) circle [radius=0.05];
\draw [fill] (2,-4) circle [radius=0.05];
\draw [fill] (3,-4) circle [radius=0.05];
\draw [fill] (4,-4) circle [radius=0.05];
\draw [fill] (5,-4) circle [radius=0.05];
\draw [fill] (6,-4) circle [radius=0.05];
\draw [fill] (7,-4) circle [radius=0.05];
\draw [fill] (8,-4) circle [radius=0.05];

\draw (-2,-4)--(-1.5,-5);
\draw (-1.5,-5)--(-1,-4);
\draw (-1,-4)--(-0.5,-5);
\draw (-0.5,-5)--(0,-4);
\draw (0,-4)--(0.5,-5);
\draw (0.5,-5)--(1,-4);
\draw (1,-4)--(1.5,-5);
\draw (1.5,-5)--(2,-4);
\draw (2,-4)--(2.5,-5);
\draw (2.5,-5)--(3,-4);
\draw (3,-4)--(3.5,-5);
\draw (3.5,-5)--(4,-4);
\draw (4,-4)--(4.5,-5);
\draw (4.5,-5)--(5,-4);
\draw (5,-4)--(5.5,-5);
\draw (5.5,-5)--(6,-4);
\draw (6,-4)--(6.5,-5);
\draw (6.5,-5)--(7,-4);
\draw (7,-4)--(7.5,-5);
\draw (7.5,-5)--(8,-4);

\draw [fill] (-2.5,-5) circle [radius=0.05];
\draw [fill] (-1.5,-5) circle [radius=0.05];
\draw [fill] (-0.5,-5) circle [radius=0.05];
\draw [fill] (0.5,-5) circle [radius=0.05];
\draw [fill] (1.5,-5) circle [radius=0.05];
\draw [fill] (2.5,-5) circle [radius=0.05];
\draw [fill] (3.5,-5) circle [radius=0.05];
\draw [fill] (4.5,-5) circle [radius=0.05];
\draw [fill] (5.5,-5) circle [radius=0.05];
\draw [fill] (6.5,-5) circle [radius=0.05];
\draw [fill] (7.5,-5) circle [radius=0.05];
\draw [fill] (8.5,-5) circle [radius=0.05];

\node at (3,0.5) {\footnotesize top};
\node at (8,-2) {\footnotesize leg};
\node at (-2,-2) {\footnotesize leg};
\node at (3,-5.5) {\footnotesize bottom};
\end{tikzpicture}
\end{center}
\caption{A simple example of a stack.}
\label{A simple example of a stack}
\end{figure}

\begin{definition}
Let $C$ be either a single $\alpha$-corridor or $\beta$-corridor with boundary word $u'v''u''v'$ (or its cyclic permutation), as shown in Figure \ref{boundary of a single corridor}. We say that a vertex $q$ on the boundary of $C$ labeled by the word $v'$ (respectively $v''$) is a \emph{$j$-vertex}, or $q$ has \emph{type $j$}, if there are $j+1$ edges connecting $q$ to $j+1$ distinct consecutive vertices on the part of boundary of $C$ labeled by the word $v''$ (respectively $v'$).
\begin{figure}[H]
\begin{center}
\begin{tikzpicture}[scale=0.4]

\draw (-10,4)--(-8,4);
\draw (-8,4)--(0,4)--(8,4);
\draw (8,4)--(10,4);

\draw (-8,4)--(-5,0);
\draw (8,4)--(5,0);
\draw (0,4)--(-5,0);
\draw (0,4)--(-3,0);
\draw (0,4)--(3,0);
\draw (0,4)--(5,0);

\draw [fill] (-8,4) circle [radius=0.1];
\draw [fill] (0,4) circle [radius=0.1];
\draw [fill] (8,4) circle [radius=0.1];
\draw [fill] (-5,0) circle [radius=0.1];
\draw [fill] (-3,0) circle [radius=0.1];
\draw [fill] (3,0) circle [radius=0.1];
\draw [fill] (5,0) circle [radius=0.1];

\draw [thick, loosely dotted] (-2,1)--(2,1);

\draw (-10,0)--(-5,0);
\draw (-5,0)--(5,0);
\draw (5,0)--(10,0);

\node [above] at (0,4) {\tiny $j$};
\node [below] at (-5,0) {\tiny $v_{1}$};
\node [below] at (-3,0) {\tiny $v_{2}$};
\node [below] at (3,0) {\tiny $v_{j}$};
\node [below] at (5,0) {\tiny $v_{j+1}$};
\end{tikzpicture}
\end{center}
\caption{Picture of a $j$-vertex. The number $j$ around the vertex indicates that the vertex is a $j$-vertex.}
\label{j-vertex}
\end{figure}
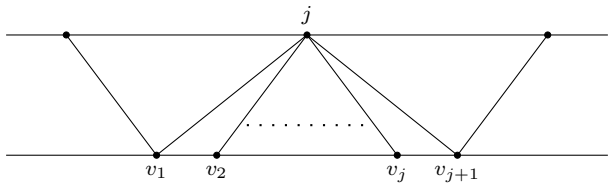
We say that a pair of adjacent vertices $q_{1}$ and $q_{2}$ on the boundary of $C$ labeled by the word $v'$ (respectively $v''$) \emph{generate} a vertex $q_{3}$ in $v''$ (respectively $v'$) if there is a triangle shown as follows:

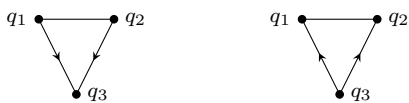
\begin{figure}[H]
\begin{center}
\begin{tikzpicture}[scale=0.5]


\draw (5,2)--(7,2);

\draw [middlearrow={stealth}] (5,2)--(6,0);
\draw [middlearrow={stealth}] (7,2)--(6,0);

\draw [fill] (5,2) circle [radius=0.1];
\draw [fill] (6,0) circle [radius=0.1];
\draw [fill] (7,2) circle [radius=0.1];

\node [left] at (5,2) {\tiny $q_{1}$};
\node [right] at (7,2) {\tiny $q_{2}$};
\node [right] at (6,0) {\tiny $q_{3}$};


\draw (12,2)--(14,2);

\draw [middlearrow={stealth reversed}] (12,2)--(13,0);
\draw [middlearrow={stealth reversed}] (14,2)--(13,0);

\draw [fill] (12,2) circle [radius=0.1];
\draw [fill] (13,0) circle [radius=0.1];
\draw [fill] (14,2) circle [radius=0.1];

\node [left] at (12,2) {\tiny $q_{1}$};
\node [right] at (14,2) {\tiny $q_{2}$};
\node [right] at (13,0) {\tiny $q_{3}$};

\end{tikzpicture}
\end{center}
\caption{A pair of adjacent vertices $p_{1}$ and $p_{2}$ generates $w$.}
\end{figure}
\end{definition}

\begin{lemma}\label{stacks are cubic}
Let $\Gamma$ be the graph as shown in Figure \ref{The suspension of a path of length at least $3$ with orientation.}. Let $w\in H_{\Gamma}$ be a freely reduced word that represents the identity. Let $T$ be a stack in a minimal van Kampen diagram $\Delta$ for $w$ such that the top of $T$ has length $l$ and the height of $T$ is $h$. Then the area of $T$ satisfies
\begin{align*}
\text{Area}(T)\leq K(lh^{2}+h^{3}),
\end{align*} 
where $K$ is a positive constant that does not depend on $l$ and $h$.
\end{lemma}

The proof of Lemma \ref{stacks are cubic} relies on a series of lemmas that analyze the boundaries of corridors of a stack. Since all the van Kampen diagrams are assumed to be minimal, in Lemma \ref{stacks are cubic} each vertex on the boundary of a corridor is of type at most $k$ for a fixed $k\geq 3$. We label the boundary words of corridors of a stack $T$ as shown in Figure \ref{Boundary words of corridors in a stack $T$.}:
\begin{figure}[H]
\begin{center}
\begin{tikzpicture}[scale=0.6]
\draw (0,10)--(10,10);
\node [above] at (5,10) {\tiny $t_{0}$};
\draw [fill] (0,10) circle [radius=0.05];
\draw [fill] (10,10) circle [radius=0.05];

\draw (0,10)--(0,8.5);
\node [left] at (0,9.25) {\tiny $u'_{1}$};
\draw (10,10)--(10,8.5);
\node [right] at (10,9.25) {\tiny $u''_{1}$};

\draw (0,8.5)--(10,8.5);
\node [above] at (5,8.5) {\tiny $t_{1}$};
\draw [fill] (0,8.5) circle [radius=0.05];
\draw [fill] (10,8.5) circle [radius=0.05];

\draw (0,8.5)--(0,7);
\node [left] at (0,7.75) {\tiny $u'_{2}$};
\draw (10,8.5)--(10,7);
\node [right] at (10,7.75) {\tiny $u''_{2}$};

\draw (0,7)--(10,7);
\node [above] at (5,7) {\tiny $t_{2}$};
\draw [fill] (0,7) circle [radius=0.05];
\draw [fill] (10,7) circle [radius=0.05];

\draw [thick, loosely dotted] (-0.5,6.5)--(-0.5,3.5);
\draw (0,7)--(0,3);
\draw [thick, loosely dotted] (5,6.5)--(5,4);
\draw (10,7)--(10,3);
\draw [thick, loosely dotted] (10.5,6.5)--(10.5,3.5);

\draw (0,3)--(10,3);
\node [above] at (5,3) {\tiny $t_{h-2}$};
\draw [fill] (0,3) circle [radius=0.05];
\draw [fill] (10,3) circle [radius=0.05];

\draw (0,3)--(0,1.5);
\node [left] at (0,2.25) {\tiny $u'_{h-1}$};
\draw (10,3)--(10,1.5);
\node [right] at (10,2.25) {\tiny $u''_{h-1}$};

\draw (0,1.5)--(10,1.5);
\node [above] at (5,1.5) {\tiny $t_{h-1}$};
\draw [fill] (0,1.5) circle [radius=0.05];
\draw [fill] (10,1.5) circle [radius=0.05];

\draw (0,1.5)--(0,0);
\node [left] at (0,0.75) {\tiny $u'_{h}$};
\draw (10,1.5)--(10,0);
\node [right] at (10,0.75) {\tiny $u''_{h}$};

\draw (0,0)--(10,0);
\node [above] at (5,0) {\tiny $t_{h}$};
\draw [fill] (0,0) circle [radius=0.05];
\draw [fill] (10,0) circle [radius=0.05];
\end{tikzpicture}
\end{center}
\caption{Boundary words of corridors in a stack $T$.}
\label{Boundary words of corridors in a stack $T$.}
\end{figure}
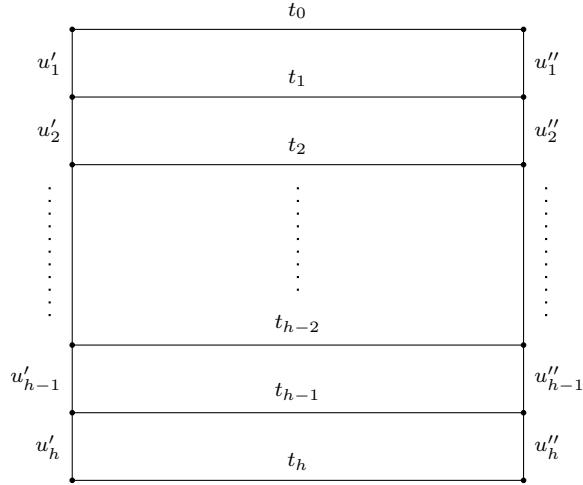

Let's look at a single corridor $C_{i}$ as shown in Figure \ref{A single corridor with boundary words.}:

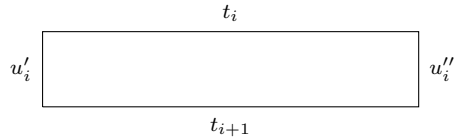
\begin{figure}[H]
\begin{center}
\begin{tikzpicture}[scale=0.5]
\draw (0,0)--(10,0)--(10,2)--(0,2)--(0,0);

\node [below] at (5,0) {\tiny $t_{i+1}$};
\node [above] at (5,2) {\tiny $t_{i}$};
\node [left] at (0,1) {\tiny $u'_{i}$};
\node [right] at (10,1) {\tiny $u''_{i}$};
\end{tikzpicture}
\end{center}
\caption{A single corridor with boundary words.}
\label{A single corridor with boundary words.}
\end{figure}

\noindent Each edge $x_{m}$ of $t_{i}$, $1\leq m\leq k$, is a part of the boundary of a triangle in the interior of $C_{i}$:

\begin{figure}[H]
\begin{center}
\begin{tikzpicture}[scale=0.5]
\draw (0,0)--(10,0)--(10,2)--(0,2)--(0,0);

\node [left] at (0,1) {\tiny $u'_{i}$};
\node [right] at (10,1) {\tiny $u''_{i}$};

\draw (4,2)--(5,0);
\draw (6,2)--(5,0);

\node [above] at (5,2) {\tiny $x_{m}$};

\draw [fill] (4,2) circle [radius=0.1];
\draw [fill] (6,2) circle [radius=0.1];
\draw [fill] (5,0) circle [radius=0.1];
\end{tikzpicture}
\end{center}
\caption{}
\label{}
\end{figure}
\noindent Depending on $C_{i}$ is an $\alpha$-corridor or $\beta$-corridor, the orientation of $x_{m}$, and the orientation of $u'_{i}, u''_{i}$ (they have the same orientation), we have the following eight possibilities:
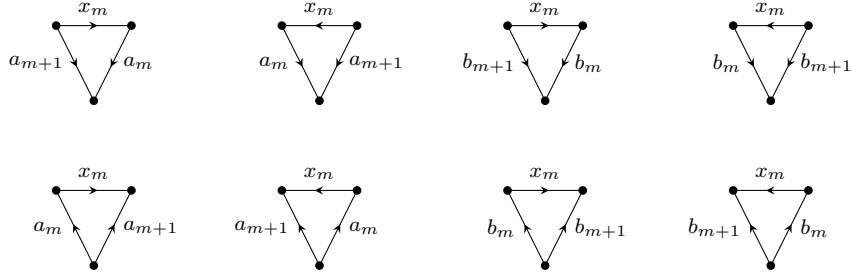
\begin{figure}[H]
\begin{center}
\begin{tikzpicture}[scale=0.5]
\draw [middlearrow={stealth}] (0,2)--(2,2);
\draw [middlearrow={stealth}] (0,2)--(1,0);
\draw [middlearrow={stealth}] (2,2)--(1,0);

\node [above] at (1,2) {\tiny $x_{m}$};
\node [left] at (0.5,1) {\tiny $a_{m+1}$};
\node [right] at (1.5,1) {\tiny $a_{m}$};

\draw [fill] (0,2) circle [radius=0.1];
\draw [fill] (2,2) circle [radius=0.1];
\draw [fill] (1,0) circle [radius=0.1];

\draw [middlearrow={stealth reversed}] (6,2)--(8,2);
\draw [middlearrow={stealth}] (6,2)--(7,0);
\draw [middlearrow={stealth}] (8,2)--(7,0);

\node [above] at (7,2) {\tiny $x_{m}$};
\node [left] at (6.5,1) {\tiny $a_{m}$};
\node [right] at (7.5,1) {\tiny $a_{m+1}$};

\draw [fill] (6,2) circle [radius=0.1];
\draw [fill] (8,2) circle [radius=0.1];
\draw [fill] (7,0) circle [radius=0.1];

\draw [middlearrow={stealth}] (12,2)--(14,2);
\draw [middlearrow={stealth}] (12,2)--(13,0);
\draw [middlearrow={stealth}] (14,2)--(13,0);

\node [above] at (13,2) {\tiny $x_{m}$};
\node [left] at (12.5,1) {\tiny $b_{m+1}$};
\node [right] at (13.5,1) {\tiny $b_{m}$};

\draw [fill] (12,2) circle [radius=0.1];
\draw [fill] (14,2) circle [radius=0.1];
\draw [fill] (13,0) circle [radius=0.1];

\draw [middlearrow={stealth reversed}] (18,2)--(20,2);
\draw [middlearrow={stealth}] (18,2)--(19,0);
\draw [middlearrow={stealth}] (20,2)--(19,0);

\node [above] at (19,2) {\tiny $x_{m}$};
\node [left] at (18.5,1) {\tiny $b_{m}$};
\node [right] at (19.5,1) {\tiny $b_{m+1}$};

\draw [fill] (18,2) circle [radius=0.1];
\draw [fill] (20,2) circle [radius=0.1];
\draw [fill] (19,0) circle [radius=0.1];

\end{tikzpicture}
\end{center}

\

\begin{center}
\begin{tikzpicture}[scale=0.5]
\draw [middlearrow={stealth}] (0,2)--(2,2);
\draw [middlearrow={stealth reversed}] (0,2)--(1,0);
\draw [middlearrow={stealth reversed}] (2,2)--(1,0);

\node [above] at (1,2) {\tiny $x_{m}$};
\node [left] at (0.5,1) {\tiny $a_{m}$};
\node [right] at (1.5,1) {\tiny $a_{m+1}$};

\draw [fill] (0,2) circle [radius=0.1];
\draw [fill] (2,2) circle [radius=0.1];
\draw [fill] (1,0) circle [radius=0.1];

\draw [middlearrow={stealth reversed}] (6,2)--(8,2);
\draw [middlearrow={stealth reversed}] (6,2)--(7,0);
\draw [middlearrow={stealth reversed}] (8,2)--(7,0);

\node [above] at (7,2) {\tiny $x_{m}$};
\node [left] at (6.5,1) {\tiny $a_{m+1}$};
\node [right] at (7.5,1) {\tiny $a_{m}$};

\draw [fill] (6,2) circle [radius=0.1];
\draw [fill] (8,2) circle [radius=0.1];
\draw [fill] (7,0) circle [radius=0.1];

\draw [middlearrow={stealth}] (12,2)--(14,2);
\draw [middlearrow={stealth reversed}] (12,2)--(13,0);
\draw [middlearrow={stealth reversed}] (14,2)--(13,0);

\node [above] at (13,2) {\tiny $x_{m}$};
\node [left] at (12.5,1) {\tiny $b_{m}$};
\node [right] at (13.5,1) {\tiny $b_{m+1}$};

\draw [fill] (12,2) circle [radius=0.1];
\draw [fill] (14,2) circle [radius=0.1];
\draw [fill] (13,0) circle [radius=0.1];

\draw [middlearrow={stealth reversed}] (18,2)--(20,2);
\draw [middlearrow={stealth reversed}] (18,2)--(19,0);
\draw [middlearrow={stealth reversed}] (20,2)--(19,0);

\node [above] at (19,2) {\tiny $x_{m}$};
\node [left] at (18.5,1) {\tiny $b_{m+1}$};
\node [right] at (19.5,1) {\tiny $b_{m}$};

\draw [fill] (18,2) circle [radius=0.1];
\draw [fill] (20,2) circle [radius=0.1];
\draw [fill] (19,0) circle [radius=0.1];

\end{tikzpicture}
\end{center}
\caption{The first row shows the cases where the arrows on $u'_{i}$ and $u''_{i}$ are pointing from $t_{i}$ to $t_{i+1}$; the second row shows the cases where the arrows on $u'_{i}$ and $u''_{i}$ are pointing from $t_{i+1}$ to $t_{i}$. The left two columns are the cases where $C_{i}$ is an $\alpha$-corridor; the right two columns are the cases where $C_{i}$ is a $\beta$-corridor.}
\label{cases for oriented 2-cells}
\end{figure}
\noindent Note that the boundary word of each triangle in Figure \ref{cases for oriented 2-cells} is a relator of $H_{\Gamma}$.

\begin{lemma}\label{Properties of one corridor in a stack}
Let $C_{i}$ be either an $\alpha$-corridor or $\beta$-corridor in a stack $T$ as shown in Figure \ref{Boundary words of corridors in a stack $T$.} and Figure \ref{A single corridor with boundary words.}.
\begin{enumerate}
\item[(1)] The area of $C_{i}$ is $|t_{i}|+|t_{i+1}|$.
\item[(2)] Suppose $|t_{i}|<|t_{i+1}|$, then 
\begin{align*}
|t_{i+1}|=\sum_{j}j\cdot\big|\lbrace\text{$j$-vertices on $t_{i}$}\rbrace\big|.
\end{align*}
\end{enumerate}
\end{lemma}

\begin{proof} \
\begin{enumerate}
\item[(1)] Since each edge on $t_{i}$ and $t_{i+1}$ is part of a unique triangle in $C_{i}$, the result follows. 
\item[(2)] Since each $j$-vertex on $t_{i}$ contributes $j$ edges on $t_{i+1}$, the statement follows immediately. 
\end{enumerate}
\end{proof}

\begin{lemma}\label{Cases of 1-vertices}
Let $C_{i}$ be either an $\alpha$-corridor or $\beta$-corridor in a stack $T$ as shown in Figure \ref{Boundary words of corridors in a stack $T$.} and Figure \ref{A single corridor with boundary words.}. There are eight combinations of edges on $t_{i}$ (respectively $t_{i+1}$) that will create $1$-vertices on the $t_{i}$ (respectively $t_{i+1}$) for a given $m$, $2\leq m\leq k-1$, as shown in Figure \ref{Cases of 1 vertex}.
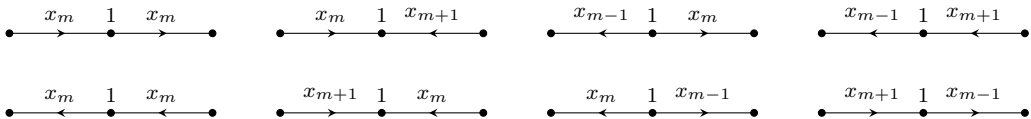
\begin{figure}[H]
\begin{center}
\begin{tikzpicture}[scale=0.9]
\draw [middlearrow={stealth}] (0,0)--(1.5,0);
\draw [middlearrow={stealth}] (1.5,0)--(3,0);

\node [above] at (0.75,0) {\tiny $x_{m}$};
\node [above] at (1.5,0) {\tiny $1$};
\node [above] at (2.25,0) {\tiny $x_{m}$};

\draw [fill] (0,0) circle [radius=0.05];
\draw [fill] (1.5,0) circle [radius=0.05];
\draw [fill] (3,0) circle [radius=0.05];

\draw [middlearrow={stealth}] (4,0)--(5.5,0);
\draw [middlearrow={stealth reversed}](5.5,0)--(7,0);

\node [above] at (4.75,0) {\tiny $x_{m}$};
\node [above] at (5.5,0) {\tiny $1$};
\node [above] at (6.25,0) {\tiny $x_{m+1}$};

\draw [fill] (4,0) circle [radius=0.05];
\draw [fill] (5.5,0) circle [radius=0.05];
\draw [fill] (7,0) circle [radius=0.05];

\draw [middlearrow={stealth reversed}] (8,0)--(9.5,0);
\draw [middlearrow={stealth}] (9.5,0)--(11,0);

\node [above] at (8.75,0) {\tiny $x_{m-1}$};
\node [above] at (9.5,0) {\tiny $1$};
\node [above] at (10.25,0) {\tiny $x_{m}$};

\draw [fill] (8,0) circle [radius=0.05];
\draw [fill] (9.5,0) circle [radius=0.05];
\draw [fill] (11,0) circle [radius=0.05];

\draw [middlearrow={stealth reversed}] (12,0)--(13.5,0);
\draw [middlearrow={stealth reversed}] (13.5,0)--(15,0);

\node [above] at (12.75,0) {\tiny $x_{m-1}$};
\node [above] at (13.5,0) {\tiny $1$};
\node [above] at (14.25,0) {\tiny $x_{m+1}$};

\draw [fill] (12,0) circle [radius=0.05];
\draw [fill] (13.5,0) circle [radius=0.05];
\draw [fill] (15,0) circle [radius=0.05];
\end{tikzpicture}

\
\

\begin{tikzpicture}[scale=0.9]
\draw [middlearrow={stealth reversed}] (0,0)--(1.5,0);
\draw [middlearrow={stealth reversed}] (1.5,0)--(3,0);

\node [above] at (0.75,0) {\tiny $x_{m}$};
\node [above] at (1.5,0) {\tiny $1$};
\node [above] at (2.25,0) {\tiny $x_{m}$};

\draw [fill] (0,0) circle [radius=0.05];
\draw [fill] (1.5,0) circle [radius=0.05];
\draw [fill] (3,0) circle [radius=0.05];

\draw [middlearrow={stealth}] (4,0)--(5.5,0);
\draw [middlearrow={stealth reversed}](5.5,0)--(7,0);

\node [above] at (4.75,0) {\tiny $x_{m+1}$};
\node [above] at (5.5,0) {\tiny $1$};
\node [above] at (6.25,0) {\tiny $x_{m}$};

\draw [fill] (4,0) circle [radius=0.05];
\draw [fill] (5.5,0) circle [radius=0.05];
\draw [fill] (7,0) circle [radius=0.05];

\draw [middlearrow={stealth reversed}] (8,0)--(9.5,0);
\draw [middlearrow={stealth}] (9.5,0)--(11,0);

\node [above] at (8.75,0) {\tiny $x_{m}$};
\node [above] at (9.5,0) {\tiny $1$};
\node [above] at (10.25,0) {\tiny $x_{m-1}$};

\draw [fill] (8,0) circle [radius=0.05];
\draw [fill] (9.5,0) circle [radius=0.05];
\draw [fill] (11,0) circle [radius=0.05];

\draw [middlearrow={stealth}] (12,0)--(13.5,0);
\draw [middlearrow={stealth}] (13.5,0)--(15,0);

\node [above] at (12.75,0) {\tiny $x_{m+1}$};
\node [above] at (13.5,0) {\tiny $1$};
\node [above] at (14.25,0) {\tiny $x_{m-1}$};

\draw [fill] (12,0) circle [radius=0.05];
\draw [fill] (13.5,0) circle [radius=0.05];
\draw [fill] (15,0) circle [radius=0.05];
\end{tikzpicture}
\end{center}
\caption{Possible combinations of edges that create $1$-vertices.}
\label{Cases of 1 vertex}
\end{figure}
\end{lemma}

\begin{proof} We prove the case when $C_{i}$ is an $\alpha$-corridor; the case for a single $\beta$-corridor is similar. To see the conclusion, we only need to know how to fill $C_{i}$. Take the left-most picture in the first row of Figure \ref{Cases of 1 vertex}. Suppose the $x_{m}$'s are edges on either $t_{i}$ or $t_{i+1}$. By Figure \ref{cases for oriented 2-cells}, we have two choice to complete the triangles based on the $x_{m}$'s: 
\begin{figure}[H]
\begin{center}
\begin{tikzpicture}
\draw [middlearrow={stealth}] (0,0)--(1.5,0);
\draw [middlearrow={stealth}] (1.5,0)--(3,0);

\node [above] at (0.75,0) {\tiny $x_{m}$};
\node [above] at (1.5,0) {\tiny $1$};
\node [above] at (2.25,0) {\tiny $x_{m}$};

\draw [fill] (0,0) circle [radius=0.05];
\draw [fill] (1.5,0) circle [radius=0.05];
\draw [fill] (3,0) circle [radius=0.05];

\draw [middlearrow={stealth}] (0,0)--(0.75,-1.5);
\draw [middlearrow={stealth}] (1.5,0)--(0.75,-1.5);
\draw [middlearrow={stealth}] (1.5,0)--(2.25,-1.5);
\draw [middlearrow={stealth}] (3,0)--(2.25,-1.5);

\node [left] at (1.275,-0.6) {\tiny $a_{m}$};
\node [right] at (1.7,-0.6) {\tiny $a_{m+1}$};

\draw [fill] (0.75,-1.5) circle [radius=0.05];
\draw [fill] (2.25,-1.5) circle [radius=0.05];

\draw [middlearrow={stealth}] (0.75,-1.5)--(2.25,-1.5);

\node [below] at (1.5, -1.5) {\tiny $x_{m}$};

\draw [middlearrow={stealth}] (5,0)--(6.5,0);
\draw [middlearrow={stealth}] (6.5,0)--(8,0);

\node [above] at (5.75,0) {\tiny $x_{m}$};
\node [above] at (6.5,0) {\tiny $1$};
\node [above] at (7.25,0) {\tiny $x_{m}$};

\draw [fill] (5,0) circle [radius=0.05];
\draw [fill] (6.5,0) circle [radius=0.05];
\draw [fill] (8,0) circle [radius=0.05];

\draw [middlearrow={stealth reversed}] (5,0)--(5.75,-1.5);
\draw [middlearrow={stealth reversed}] (6.5,0)--(5.75,-1.5);
\draw [middlearrow={stealth reversed}] (6.5,0)--(7.25,-1.5);
\draw [middlearrow={stealth reversed}] (8,0)--(7.25,-1.5);

\node [left] at (6.275,-0.6) {\tiny $a_{m+1}$};
\node [right] at (6.7,-0.6) {\tiny $a_{m}$};

\draw [fill] (5.75,-1.5) circle [radius=0.05];
\draw [fill] (7.25,-1.5) circle [radius=0.05];

\draw [middlearrow={stealth}] (5.75,-1.5)--(7.25,-1.5);

\node [below] at (6.5, -1.5) {\tiny $x_{m}$};
\end{tikzpicture}
\end{center}
\caption{}
\end{figure}
\noindent Thus, the vertex is a $1$-vertex. The arguments for other combinations are similar. 
\end{proof}

\begin{lemma}\label{two consecutive corridor with the same orientation}
Let $C_{i}$ and $C_{i+1}$ be two consecutive corridors in a stack $T$. Suppose that the arrows on the edges $u'_{i}, u''_{i}, u'_{i+1}, u''_{i+1}$ have the same orientation, as shown in Figure \ref{picture of two consecutive corridor with the same orientation}. Assume that there are no $0$-vertices on $t_{i}$. Then
\begin{enumerate}
\item[(1)] All the vertices on $t_{i+1}$ are $1$-vertices, $2$-vertices, or $3$-vertices, except possibly the two vertices at the ends of $t_{i+1}$.

\item[(2)] All the vertices on $t_{i+2}$ are either $1$-vertices or $2$-vertices, except possibly the two vertices at the ends of $t_{i+1}$.

\item[(3)] We have $|t_{i}|\leq|t_{i+1}|\leq|t_{i+2}|$.
\end{enumerate} 
\begin{figure}[H]
\begin{center}
\begin{tikzpicture}[scale=0.9]

\node [above] at (3,3) {\tiny $t_{i}$};
\draw (0,3)--(6,3);

\draw [fill] (0,3) circle [radius=0.05];
\draw [fill] (6,3) circle [radius=0.05];

\node [left] at (0,2.25) {\tiny $u'_{i}$};
\draw [middlearrow={stealth}] (0,3)--(0,1.5);
\node at (3,2.25) {\footnotesize Corridor $C_{i}$};
\draw [middlearrow={stealth}] (6,3)--(6,1.5);
\node [right] at (6,2.25) {\tiny $u''_{i}$};

\draw (0,1.5)--(6,1.5);

\node [above] at (3,1.4) {\tiny $t_{i+1}$};
\draw [fill] (0,1.5) circle [radius=0.05];
\draw [fill] (6,1.5) circle [radius=0.05];

\node [left] at (0,0.75) {\tiny $u'_{i+1}$};
\draw [middlearrow={stealth}] (0,1.5)--(0,0);
\node at (3,0.75) {\footnotesize Corridor $C_{i+1}$};
\draw [middlearrow={stealth}] (6,1.5)--(6,0);
\node [right] at (6,0.75) {\tiny $u''_{i+1}$};

\draw (0,0)--(6,0);

\node [below] at (3,0) {\tiny $t_{i+2}$};
\draw [fill] (0,0) circle [radius=0.05];
\draw [fill] (6,0) circle [radius=0.05];



\node [above] at (11,3) {\tiny $t_{i}$};
\draw (8,3)--(14,3);

\draw [fill] (8,3) circle [radius=0.05];
\draw [fill] (14,3) circle [radius=0.05];

\node [left] at (8,2.25) {\tiny $u'_{i}$};
\draw [middlearrow={stealth reversed}] (8,3)--(8,1.5);
\node at (11,2.25) {\footnotesize Corridor $C_{i}$};
\draw [middlearrow={stealth reversed}] (14,3)--(14,1.5);
\node [right] at (14,2.25) {\tiny $u''_{i}$};

\draw (8,1.5)--(14,1.5);

\node [above] at (11,1.4) {\tiny $t_{i+1}$};
\draw [fill] (8,1.5) circle [radius=0.05];
\draw [fill] (14,1.5) circle [radius=0.05];

\node [left] at (8,0.75) {\tiny $u'_{i+1}$};
\draw [middlearrow={stealth reversed}] (8,1.5)--(8,0);
\node at (11,0.75) {\footnotesize Corridor $C_{i+1}$};
\draw [middlearrow={stealth reversed}] (14,1.5)--(14,0);
\node [right] at (14,0.75) {\tiny $u''_{i+1}$};

\draw (8,0)--(14,0);

\node [below] at (11,0) {\tiny $t_{i+2}$};
\draw [fill] (8,0) circle [radius=0.05];
\draw [fill] (14,0) circle [radius=0.05];

\end{tikzpicture}
\end{center}
\caption{}
\label{picture of two consecutive corridor with the same orientation}
\end{figure}
\end{lemma}

\begin{proof} 
We prove the lemma for the case when $C_{i},C_{i+1}$ are consecutive $\alpha$-corridors, and the arrows on $u'_{i},u''_{i},u'_{i+1},u''_{i+1}$ are pointing away from $t_{i}$. Other cases are similar. 

\begin{enumerate}
\item[(1)] We claim that every pair of adjacent vertices on $t_{i}$ generates either a $1$-vertex or $3$-vertex on $t_{i+1}$. Let two adjacent vertices on $t_{i}$ be connected by an edge $x_{m}$. These two adjacent vertices generate a vertex on $t_{i+1}$; Figure \ref{A pair of adjacent vertices on $t_{i}$ generates either a $1$-vertex or a $3$-vertex on $t_{i+1}$.} lists all possible combinations of edges that meet at the vertex on $t_{i+1}$ that is generated by a pair of adjacent vertices on $t_{i}$.
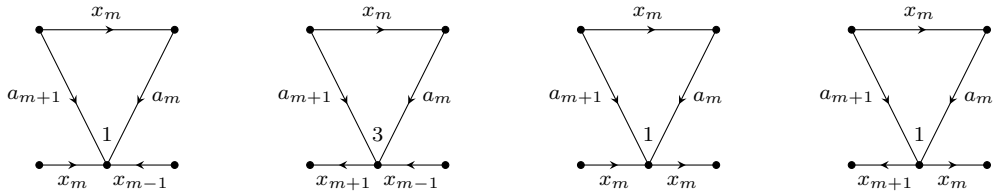
\begin{figure}[H] 
\begin{center}
\begin{tikzpicture}[scale=0.9]

\draw [middlearrow={stealth}] (0,2)--(2,2);
\draw [middlearrow={stealth}] (0,2)--(1,0);
\draw [middlearrow={stealth}] (2,2)--(1,0);
\draw [middlearrow={stealth}] (0,0)--(1,0);
\draw [middlearrow={stealth reversed}] (1,0)--(2,0);

\node [above] at (1,2) {\tiny $x_{m}$};
\node [left] at (0.5,1) {\tiny $a_{m+1}$};
\node [right] at (1.5,1) {\tiny $a_{m}$};
\node [below] at (0.5,0) {\tiny $x_{m}$};
\node [below] at (1.5,0) {\tiny $x_{m-1}$};

\node [above] at (1,0.2) {\tiny $1$};

\draw [fill] (0,2) circle [radius=0.05];
\draw [fill] (2,2) circle [radius=0.05];
\draw [fill] (0,0) circle [radius=0.05];
\draw [fill] (1,0) circle [radius=0.05];
\draw [fill] (2,0) circle [radius=0.05];

\draw [middlearrow={stealth}] (4,2)--(6,2);
\draw [middlearrow={stealth}] (4,2)--(5,0);
\draw [middlearrow={stealth}] (6,2)--(5,0);
\draw [middlearrow={stealth reversed}] (4,0)--(5,0);
\draw [middlearrow={stealth reversed}] (5,0)--(6,0);

\node [above] at (5,2) {\tiny $x_{m}$};
\node [left] at (4.5,1) {\tiny $a_{m+1}$};
\node [right] at (5.5,1) {\tiny $a_{m}$};
\node [below] at (4.5,0) {\tiny $x_{m+1}$};
\node [below] at (5.5,0) {\tiny $x_{m-1}$};

\node [above] at (5,0.2) {\tiny $3$};

\draw [fill] (4,2) circle [radius=0.05];
\draw [fill] (6,2) circle [radius=0.05];
\draw [fill] (4,0) circle [radius=0.05];
\draw [fill] (5,0) circle [radius=0.05];
\draw [fill] (6,0) circle [radius=0.05];

\draw [middlearrow={stealth}] (8,2)--(10,2);
\draw [middlearrow={stealth}] (8,2)--(9,0);
\draw [middlearrow={stealth}] (10,2)--(9,0);
\draw [middlearrow={stealth}] (8,0)--(9,0);
\draw [middlearrow={stealth}] (9,0)--(10,0);

\node [above] at (9,2) {\tiny $x_{m}$};
\node [left] at (8.5,1) {\tiny $a_{m+1}$};
\node [right] at (9.5,1) {\tiny $a_{m}$};
\node [below] at (8.5,0) {\tiny $x_{m}$};
\node [below] at (9.5,0) {\tiny $x_{m}$};

\node [above] at (9,0.2) {\tiny $1$};

\draw [fill] (8,2) circle [radius=0.05];
\draw [fill] (10,2) circle [radius=0.05];
\draw [fill] (8,0) circle [radius=0.05];
\draw [fill] (9,0) circle [radius=0.05];
\draw [fill] (10,0) circle [radius=0.05];

\draw [middlearrow={stealth}] (12,2)--(14,2);
\draw [middlearrow={stealth}] (12,2)--(13,0);
\draw [middlearrow={stealth}] (14,2)--(13,0);
\draw [middlearrow={stealth reversed}] (12,0)--(13,0);
\draw [middlearrow={stealth}] (13,0)--(14,0);

\node [above] at (13,2) {\tiny $x_{m}$};
\node [left] at (12.5,1) {\tiny $a_{m+1}$};
\node [right] at (13.5,1) {\tiny $a_{m}$};
\node [below] at (12.5,0) {\tiny $x_{m+1}$};
\node [below] at (13.5,0) {\tiny $x_{m}$};

\node [above] at (13,0.2) {\tiny $1$};

\draw [fill] (12,2) circle [radius=0.05];
\draw [fill] (14,2) circle [radius=0.05];
\draw [fill] (12,0) circle [radius=0.05];
\draw [fill] (13,0) circle [radius=0.05];
\draw [fill] (14,0) circle [radius=0.05];

\end{tikzpicture}
\end{center}
\caption{A pair of adjacent vertices on $t_{i}$ generates either a $1$-vertex or $3$-vertex on $t_{i+1}$.}
\label{A pair of adjacent vertices on $t_{i}$ generates either a $1$-vertex or a $3$-vertex on $t_{i+1}$.}
\end{figure}

\noindent The $1$-vertices in Figure \ref{A pair of adjacent vertices on $t_{i}$ generates either a $1$-vertex or a $3$-vertex on $t_{i+1}$.} are recognized by Lemma \ref{Cases of 1-vertices}. The following picture shows that the vertex in Figure \ref{A pair of adjacent vertices on $t_{i}$ generates either a $1$-vertex or a $3$-vertex on $t_{i+1}$.} is a $3$-vertex:
\begin{figure}[H]
\begin{center}
\begin{tikzpicture}
\draw [middlearrow={stealth reversed}] (0,0)--(4.5,0);
\draw [middlearrow={stealth reversed}] (4.5,0)--(9,0);

\node [above] at (2.25,0) {\tiny $x_{m+1}$};
\node [above] at (4.5,0) {\tiny $3$};
\node [above] at (6.75,0) {\tiny $x_{m-1}$};

\draw [fill] (0,0) circle [radius=0.05];
\draw [fill] (4.5,0) circle [radius=0.05];
\draw [fill] (9,0) circle [radius=0.05];


\draw [middlearrow={stealth}] (0,0)--(2.25,-2);
\draw [middlearrow={stealth}] (4.5,0)--(2.25,-2);

\draw [middlearrow={stealth}] (4.5,0)--(3.75,-2);
\draw [middlearrow={stealth}] (4.5,0)--(5.25,-2);
\draw [middlearrow={stealth}] (4.5,0)--(6.75,-2);

\draw [middlearrow={stealth}] (4.5,0)--(6.75,-2);
\draw [middlearrow={stealth}] (9,0)--(6.75,-2);

\node [left] at (1.125,-1) {\tiny $a_{m+1}$};
\node [left] at (3.275,-1) {\tiny $a_{m+2}$};
\node [left] at (4.125,-1.3) {\tiny $a_{m+1}$};
\node [right] at (4.875,-1.3) {\tiny $a_{m}$};
\node [right] at (5.625,-1) {\tiny $a_{m-1}$};
\node [right] at (7.875,-1) {\tiny $a_{m}$};


\draw [middlearrow={stealth reversed}] (2.25,-2)--(3.75,-2);
\draw [middlearrow={stealth reversed}] (3.75,-2)--(5.25,-2);
\draw [middlearrow={stealth reversed}] (5.25,-2)--(6.75,-2);

\node [below] at (3,-2) {\tiny $x_{m+1}$};
\node [below] at (4.5,-2) {\tiny $x_{m}$};
\node [below] at (6,-2) {\tiny $x_{m-1}$};

\draw [fill] (2.25,-2) circle [radius=0.05];
\draw [fill] (3.75,-2) circle [radius=0.05];
\draw [fill] (5.25,-2) circle [radius=0.05];
\draw [fill] (6.75,-2) circle [radius=0.05];
\end{tikzpicture}
\end{center}
\caption{}
\end{figure}

\noindent There are four more cases obtained by flipping the pictures in Figure \ref{A pair of adjacent vertices on $t_{i}$ generates either a $1$-vertex or a $3$-vertex on $t_{i+1}$.} such that the arrow on the edge $x_{m}$ of $t_{i}$ has the opposite orientation. This proves the claim.

Next, we show that the vertices on $t_{i+1}$ that are not generated by pairs of adjacent vertices on $t_{i}$ are $2$-vertices, except the two vertices at the end of $t_{i+1}$. We claim that each $j$-vertex on $t_{i}$ creates $j-1$ vertices on $t_{i+1}$ that are all $2$-vertices. For any $j$-vertex on $t_{i}$ (respectively $t_{i+1}$), by definition, there are $j+1$ edges connecting the $j$-vertex to $j+1$ consecutive distinct vertices on $t_{i+1}$ (respectively $t_{i}$). That is, there are $j$ distinct triangles based on $j$ consecutive edges on $t_{i+1}$ (respectively $t_{i}$) that has the $j$-vertex as the common vertex; see Figure \ref{Each $j$-vertex on $t_{i}$ create $j-1$ vertices on $t_{i+1}$ that are all $2$-vertices.}. 

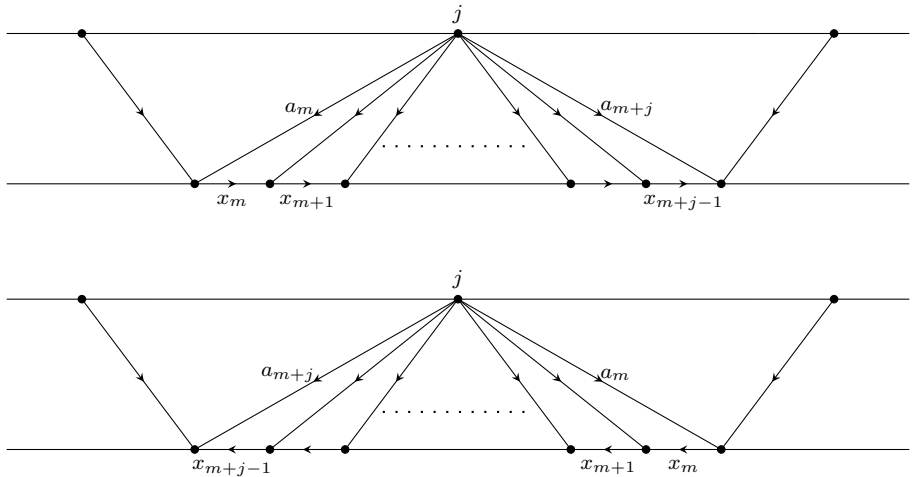
\begin{figure}[H]
\begin{center}
\begin{tikzpicture}[scale=0.5]

\draw (-12,4)--(-8,4);
\draw (-8,4)--(0,4)--(8,4);
\draw (8,4)--(12,4);

\draw [middlearrow={stealth}] (-10,4)--(-7,0);
\draw [middlearrow={stealth}] (0,4)--(-7,0);
\draw [middlearrow={stealth}] (0,4)--(-5,0);
\draw [middlearrow={stealth}] (0,4)--(-3,0);
\draw [thick, loosely dotted] (-2,1)--(2,1);
\draw [middlearrow={stealth}] (0,4)--(3,0);
\draw [middlearrow={stealth}] (0,4)--(5,0);
\draw [middlearrow={stealth}] (0,4)--(7,0);
\draw [middlearrow={stealth}] (10,4)--(7,0);

\node [left] at (-3.5,2) {\tiny $a_{m}$};
\node [right] at (3.5,2) {\tiny $a_{m+j}$};

\draw [fill] (-10,4) circle [radius=0.1];
\draw [fill] (0,4) circle [radius=0.1];
\draw [fill] (10,4) circle [radius=0.1];


\draw (-12,0)--(-7,0);
\draw [middlearrow={stealth}] (-7,0)--(-5,0);
\draw [middlearrow={stealth}] (-5,0)--(-3,0);
\draw (-3,0)--(3,0);
\draw [middlearrow={stealth}] (3,0)--(5,0);
\draw [middlearrow={stealth}] (5,0)--(7,0);
\draw (7,0)--(12,0);

\node [above] at (0,4) {\tiny $j$};
\node [below] at (-6,0) {\tiny $x_{m}$};
\node [below] at (-4,0) {\tiny $x_{m+1}$};
\node [below] at (6,0) {\tiny $x_{m+j-1}$};

\draw [fill] (-7,0) circle [radius=0.1];
\draw [fill] (-5,0) circle [radius=0.1];
\draw [fill] (-3,0) circle [radius=0.1];
\draw [fill] (3,0) circle [radius=0.1];
\draw [fill] (5,0) circle [radius=0.1];
\draw [fill] (7,0) circle [radius=0.1];

\end{tikzpicture}


\


\begin{tikzpicture}[scale=0.5]

\draw (-12,4)--(-8,4);
\draw (-8,4)--(0,4)--(8,4);
\draw (8,4)--(12,4);

\draw [middlearrow={stealth}] (-10,4)--(-7,0);
\draw [middlearrow={stealth}] (0,4)--(-7,0);
\draw [middlearrow={stealth}] (0,4)--(-5,0);
\draw [middlearrow={stealth}] (0,4)--(-3,0);
\draw [thick, loosely dotted] (-2,1)--(2,1);
\draw [middlearrow={stealth}] (0,4)--(3,0);
\draw [middlearrow={stealth}] (0,4)--(5,0);
\draw [middlearrow={stealth}] (0,4)--(7,0);
\draw [middlearrow={stealth}] (10,4)--(7,0);

\node [left] at (-3.5,2) {\tiny $a_{m+j}$};
\node [right] at (3.5,2) {\tiny $a_{m}$};

\draw [fill] (-10,4) circle [radius=0.1];
\draw [fill] (0,4) circle [radius=0.1];
\draw [fill] (10,4) circle [radius=0.1];


\draw (-12,0)--(-7,0);
\draw [middlearrow={stealth reversed}] (-7,0)--(-5,0);
\draw [middlearrow={stealth reversed}] (-5,0)--(-3,0);
\draw (-3,0)--(3,0);
\draw [middlearrow={stealth reversed}] (3,0)--(5,0);
\draw [middlearrow={stealth reversed}] (5,0)--(7,0);
\draw (7,0)--(12,0);

\node [above] at (0,4) {\tiny $j$};
\node [below] at (-6,0) {\tiny $x_{m+j-1}$};
\node [below] at (4,0) {\tiny $x_{m+1}$};
\node [below] at (6,0) {\tiny $x_{m}$};

\draw [fill] (-7,0) circle [radius=0.1];
\draw [fill] (-5,0) circle [radius=0.1];
\draw [fill] (-3,0) circle [radius=0.1];
\draw [fill] (3,0) circle [radius=0.1];
\draw [fill] (5,0) circle [radius=0.1];
\draw [fill] (7,0) circle [radius=0.1];

\end{tikzpicture}
\end{center}
\caption{Each $j$-vertex on $t_{i}$ (respectively $t_{i+1}$) creates $j-1$ vertices on $t_{i+1}$ respectively $t_{i}$) that are all $2$-vertices.}
\label{Each $j$-vertex on $t_{i}$ create $j-1$ vertices on $t_{i+1}$ that are all $2$-vertices.}
\end{figure} 
\noindent Let $x_{r-1},x_{r}$ be two consecutive edges, $r\in\lbrace m,\cdots,m+j-1\rbrace$, and let the vertex $v_{r}$ be the common vertex of $x_{r-1}$ and $x_{r}$. Figure \ref{The vertex $v_{r}$ on $t_{i+1}$ is a $2$-vertex.} shows that $v_{r}$ is a $2$-vertex. This proves the claim.
\begin{figure}[H]
\begin{center}
\begin{tikzpicture}

\draw [middlearrow={stealth}] (0,0)--(2,0);
\draw [middlearrow={stealth}] (2,0)--(4,0);

\node [above] at (1,0) {\tiny $x_{r-1}$};
\node [above] at (2,0) {\tiny $v_{r}$};
\node [above] at (3,0) {\tiny $x_{r}$};

\draw [fill] (0,0) circle [radius=0.05];
\draw [fill] (2,0) circle [radius=0.05];
\draw [fill] (4,0) circle [radius=0.05];

\draw [middlearrow={stealth}] (0,0)--(1,-2);
\draw [middlearrow={stealth}] (2,0)--(1,-2);
\draw [middlearrow={stealth}] (2,0)--(2,-2);
\draw [middlearrow={stealth}] (2,0)--(3,-2);
\draw [middlearrow={stealth}] (4,0)--(3,-2);

\node [left] at (1.5,-1) {\tiny $a_{r-1}$};
\node [right] at (2.5,-1) {\tiny $a_{r+2}$};

\draw [fill] (1,-2) circle [radius=0.05];
\draw [fill] (2,-2) circle [radius=0.05];
\draw [fill] (3,-2) circle [radius=0.05];


\draw [middlearrow={stealth}] (1,-2)--(2,-2);
\draw [middlearrow={stealth}] (2,-2)--(3,-2);

\node [below] at (1.5,-2) {\tiny $x_{r-1}$};
\node [below] at (2.5,-2) {\tiny $x_{r}$};



\draw [middlearrow={stealth reversed}] (6,0)--(8,0);
\draw [middlearrow={stealth reversed}] (8,0)--(10,0);

\node [above] at (7,0) {\tiny $x_{r}$};
\node [above] at (8,0) {\tiny $v_{r}$};
\node [above] at (9,0) {\tiny $x_{r-1}$};

\draw [fill] (6,0) circle [radius=0.05];
\draw [fill] (8,0) circle [radius=0.05];
\draw [fill] (10,0) circle [radius=0.05];

\draw [middlearrow={stealth}] (6,0)--(7,-2);
\draw [middlearrow={stealth}] (8,0)--(7,-2);
\draw [middlearrow={stealth}] (8,0)--(8,-2);
\draw [middlearrow={stealth}] (8,0)--(9,-2);
\draw [middlearrow={stealth}] (10,0)--(9,-2);

\node [left] at (7.5,-1) {\tiny $a_{r+1}$};
\node [right] at (8.5,-1) {\tiny $a_{r-1}$};

\draw [fill] (7,-2) circle [radius=0.05];
\draw [fill] (8,-2) circle [radius=0.05];
\draw [fill] (9,-2) circle [radius=0.05];


\draw [middlearrow={stealth reversed}] (7,-2)--(8,-2);
\draw [middlearrow={stealth reversed}] (8,-2)--(9,-2);

\node [below] at (7.5,-2) {\tiny $x_{r}$};
\node [below] at (8.5,-2) {\tiny $x_{r-1}$};

\end{tikzpicture}
\end{center}
\caption{The vertex $v_{r}$ on $t_{i+1}$ is a $2$-vertex.}
\label{The vertex $v_{r}$ on $t_{i+1}$ is a $2$-vertex.}
\end{figure}
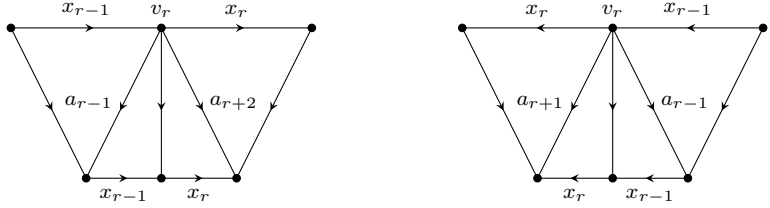

\item[(2)] We have shown that all the vertices on $t_{i+1}$ are $1$-vertices, $2$-vertices, or $3$-vertices. Each of the vertices on $t_{i+1}$ creates numbers of $0$, $1$, or $2$ vertices on $t_{i+2}$, respectively, and they are all $2$-vertices. Other vertices on $t_{i+2}$ are generated by pairs of adjacent vertices on $t_{i+1}$. We now show that these vertices on $t_{i+2}$ are $1$-vertices. We prove the claim by showing all the possible combinations of types of vertices on $t_{i+1}$. We have six cases: pairs of adjacent $1$-vertices, adjacent $2$-vertices, and adjacent $3$-vertices; pairs of adjacent $1$-vertex and $2$-vertex; pairs of adjacent $1$-vertex and $3$-vertex; and pairs of adjacent $2$-vertex and $3$-vertex. In the following pictures, all the $1$-vertices are recognized by Lemma \ref{Cases of 1-vertices}.

Figure \ref{Two adjacent $1$-vertices on $t_{i+1}$ generate a $1$-vertex on $t_{i+2}$.} shows that a pair of adjacent $1$-vertices on $t_{i+1}$ generates a $1$-vertex on $t_{i+2}$:
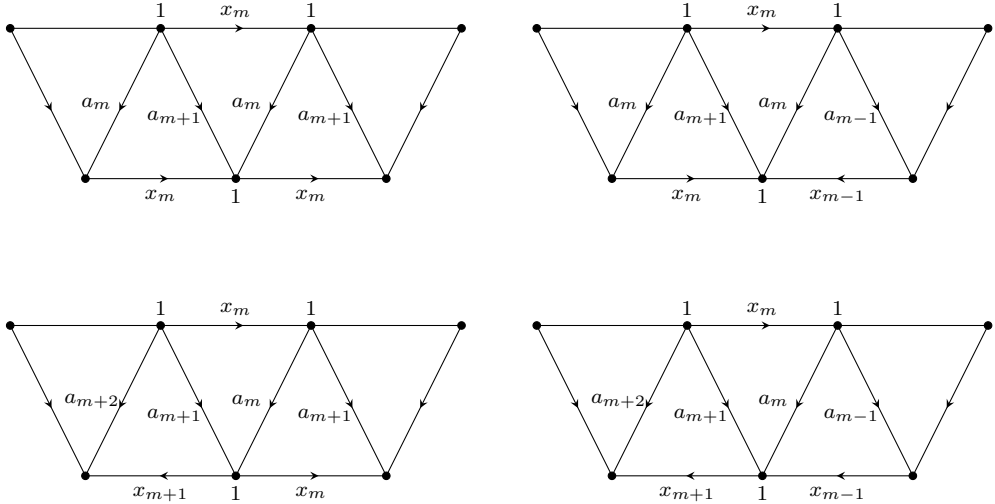
\begin{figure}[H]
\begin{center}
\begin{tikzpicture}

\draw (0,0)--(2,0);
\draw [middlearrow={stealth}] (2,0)--(4,0);
\draw (4,0)--(6,0);

\node [above] at (2,0) {\tiny $1$};
\node [above] at (3,0) {\tiny $x_{m}$};
\node [above] at (4,0) {\tiny $1$};
\draw [fill] (0,0) circle [radius=0.05];
\draw [fill] (2,0) circle [radius=0.05];
\draw [fill] (4,0) circle [radius=0.05];
\draw [fill] (6,0) circle [radius=0.05];

\draw [middlearrow={stealth}] (0,0)--(1,-2);
\draw [middlearrow={stealth}] (2,0)--(1,-2);
\draw [middlearrow={stealth}] (2,0)--(3,-2);
\draw [middlearrow={stealth}] (4,0)--(3,-2);
\draw [middlearrow={stealth}] (4,0)--(5,-2);
\draw [middlearrow={stealth}] (6,0)--(5,-2);

\node [left] at (1.5,-1) {\tiny $a_{m}$};
\node [left] at (2.7,-1.2) {\tiny $a_{m+1}$};
\node [left] at (3.5,-1) {\tiny $a_{m}$};
\node [left] at (4.7,-1.2) {\tiny $a_{m+1}$};

\draw [middlearrow={stealth}] (1,-2)--(3,-2);
\draw [middlearrow={stealth}] (3,-2)--(5,-2);

\node [below] at (2,-2) {\tiny $x_{m}$};
\node [below] at (3,-2) {\tiny $1$};
\node [below] at (4,-2) {\tiny $x_{m}$};
\draw [fill] (1,-2) circle [radius=0.05];
\draw [fill] (3,-2) circle [radius=0.05];
\draw [fill] (5,-2) circle [radius=0.05];


\draw (7,0)--(9,0);
\draw [middlearrow={stealth}] (9,0)--(11,0);
\draw (11,0)--(13,0);

\node [above] at (9,0) {\tiny $1$};
\node [above] at (10,0) {\tiny $x_{m}$};
\node [above] at (11,0) {\tiny $1$};
\draw [fill] (7,0) circle [radius=0.05];
\draw [fill] (9,0) circle [radius=0.05];
\draw [fill] (11,0) circle [radius=0.05];
\draw [fill] (13,0) circle [radius=0.05];

\draw [middlearrow={stealth}] (7,0)--(8,-2);
\draw [middlearrow={stealth}] (9,0)--(8,-2);
\draw [middlearrow={stealth}] (9,0)--(10,-2);
\draw [middlearrow={stealth}] (11,0)--(10,-2);
\draw [middlearrow={stealth}] (11,0)--(12,-2);
\draw [middlearrow={stealth}] (13,0)--(12,-2);

\node [left] at (8.5,-1) {\tiny $a_{m}$};
\node [left] at (9.7,-1.2) {\tiny $a_{m+1}$};
\node [left] at (10.5,-1) {\tiny $a_{m}$};
\node [left] at (11.7,-1.2) {\tiny $a_{m-1}$};

\draw [middlearrow={stealth}] (8,-2)--(10,-2);
\draw [middlearrow={stealth reversed}] (10,-2)--(12,-2);

\node [below] at (9,-2) {\tiny $x_{m}$};
\node [below] at (10,-2) {\tiny $1$};
\node [below] at (11,-2) {\tiny $x_{m-1}$};
\draw [fill] (8,-2) circle [radius=0.05];
\draw [fill] (10,-2) circle [radius=0.05];
\draw [fill] (12,-2) circle [radius=0.05];
\end{tikzpicture}

\

\

\begin{tikzpicture}

\draw (0,0)--(2,0);
\draw [middlearrow={stealth}] (2,0)--(4,0);
\draw (4,0)--(6,0);

\node [above] at (2,0) {\tiny $1$};
\node [above] at (3,0) {\tiny $x_{m}$};
\node [above] at (4,0) {\tiny $1$};
\draw [fill] (0,0) circle [radius=0.05];
\draw [fill] (2,0) circle [radius=0.05];
\draw [fill] (4,0) circle [radius=0.05];
\draw [fill] (6,0) circle [radius=0.05];

\draw [middlearrow={stealth}] (0,0)--(1,-2);
\draw [middlearrow={stealth}] (2,0)--(1,-2);
\draw [middlearrow={stealth}] (2,0)--(3,-2);
\draw [middlearrow={stealth}] (4,0)--(3,-2);
\draw [middlearrow={stealth}] (4,0)--(5,-2);
\draw [middlearrow={stealth}] (6,0)--(5,-2);

\node [left] at (1.6,-1) {\tiny $a_{m+2}$};
\node [left] at (2.7,-1.2) {\tiny $a_{m+1}$};
\node [left] at (3.5,-1) {\tiny $a_{m}$};
\node [left] at (4.7,-1.2) {\tiny $a_{m+1}$};

\draw [middlearrow={stealth reversed}] (1,-2)--(3,-2);
\draw [middlearrow={stealth}] (3,-2)--(5,-2);

\node [below] at (2,-2) {\tiny $x_{m+1}$};
\node [below] at (3,-2) {\tiny $1$};
\node [below] at (4,-2) {\tiny $x_{m}$};
\draw [fill] (1,-2) circle [radius=0.05];
\draw [fill] (3,-2) circle [radius=0.05];
\draw [fill] (5,-2) circle [radius=0.05];


\draw (7,0)--(9,0);
\draw [middlearrow={stealth}] (9,0)--(11,0);
\draw (11,0)--(13,0);

\node [above] at (9,0) {\tiny $1$};
\node [above] at (10,0) {\tiny $x_{m}$};
\node [above] at (11,0) {\tiny $1$};
\draw [fill] (7,0) circle [radius=0.05];
\draw [fill] (9,0) circle [radius=0.05];
\draw [fill] (11,0) circle [radius=0.05];
\draw [fill] (13,0) circle [radius=0.05];

\draw [middlearrow={stealth}] (7,0)--(8,-2);
\draw [middlearrow={stealth}] (9,0)--(8,-2);
\draw [middlearrow={stealth}] (9,0)--(10,-2);
\draw [middlearrow={stealth}] (11,0)--(10,-2);
\draw [middlearrow={stealth}] (11,0)--(12,-2);
\draw [middlearrow={stealth}] (13,0)--(12,-2);

\node [left] at (8.6,-1) {\tiny $a_{m+2}$};
\node [left] at (9.7,-1.2) {\tiny $a_{m+1}$};
\node [left] at (10.5,-1) {\tiny $a_{m}$};
\node [left] at (11.7,-1.2) {\tiny $a_{m-1}$};

\draw [middlearrow={stealth reversed}] (8,-2)--(10,-2);
\draw [middlearrow={stealth reversed}] (10,-2)--(12,-2);

\node [below] at (9,-2) {\tiny $x_{m+1}$};
\node [below] at (10,-2) {\tiny $1$};
\node [below] at (11,-2) {\tiny $x_{m-1}$};
\draw [fill] (8,-2) circle [radius=0.05];
\draw [fill] (10,-2) circle [radius=0.05];
\draw [fill] (12,-2) circle [radius=0.05];
\end{tikzpicture}
\end{center}
\caption{Two adjacent $1$-vertices on $t_{i+1}$ generate a $1$-vertex on $t_{i+2}$.}
\label{Two adjacent $1$-vertices on $t_{i+1}$ generate a $1$-vertex on $t_{i+2}$.}
\end{figure}

Next, Figure \ref{A pair of adjacent $1$-vertex and $2$-vertex on $t_{i+1}$ generates a $1$-vertex on $t_{i+2}$.} shows that a pair of adjacent $1$-vertex and $2$-vertex on $t_{i+1}$ generates a $1$-vertex on $t_{i+2}$. Recall that Figure \ref{A pair of adjacent vertices on $t_{i}$ generates either a $1$-vertex or a $3$-vertex on $t_{i+1}$.} and Figure \ref{The vertex $v_{r}$ on $t_{i+1}$ is a $2$-vertex.} are the only situations of $1$-vertex and $2$-vertex on $t_{i+1}$.
\begin{figure}[H]
\begin{center}
\begin{tikzpicture}

\draw [middlearrow={stealth}] (0,0)--(2,0);
\draw [middlearrow={stealth}] (2,0)--(4,0);
\draw [middlearrow={stealth reversed}] (4,0)--(6,0);

\node [above] at (1,0) {\tiny $x_{m-1}$};
\node [above] at (2,0) {\tiny $2$};
\node [above] at (3,0) {\tiny $x_{m}$};
\node [above] at (4,0) {\tiny $1$};
\node [above] at (5,0) {\tiny $x_{m-1}$};
\draw [fill] (0,0) circle [radius=0.05];
\draw [fill] (2,0) circle [radius=0.05];
\draw [fill] (4,0) circle [radius=0.05];
\draw [fill] (6,0) circle [radius=0.05];

\draw [middlearrow={stealth}] (0,0)--(1,-2);
\draw [middlearrow={stealth}] (2,0)--(1,-2);
\draw [middlearrow={stealth}] (2,0)--(2,-2);
\draw [middlearrow={stealth}] (2,0)--(3,-2);
\draw [middlearrow={stealth}] (4,0)--(3,-2);
\draw [middlearrow={stealth}] (4,0)--(5,-2);
\draw [middlearrow={stealth}] (6,0)--(5,-2);

\node [left] at (1.7,-0.65) {\tiny $a_{m-1}$};
\node [right] at (2.3,-0.65) {\tiny $a_{m+1}$};
\node [left] at (3.5,-1) {\tiny $a_{m}$};
\node [right] at (4.3,-0.65) {\tiny $a_{m-1}$};

\draw [middlearrow={stealth}] (1,-2)--(2,-2);
\draw [middlearrow={stealth}] (2,-2)--(3,-2);
\draw [middlearrow={stealth reversed}] (3,-2)--(5,-2);

\node [below] at (1.5,-2) {\tiny $x_{m-1}$};
\node [below] at (2.5,-2) {\tiny $x_{m}$};
\node [below] at (3,-2) {\tiny $1$};
\node [below] at (4,-2) {\tiny $x_{m-1}$};
\draw [fill] (1,-2) circle [radius=0.05];
\draw [fill] (2,-2) circle [radius=0.05];
\draw [fill] (3,-2) circle [radius=0.05];
\draw [fill] (5,-2) circle [radius=0.05];


\draw [middlearrow={stealth reversed}] (7,0)--(9,0);
\draw [middlearrow={stealth reversed}] (9,0)--(11,0);
\draw [middlearrow={stealth}](11,0)--(13,0);

\node [above] at (8,0) {\tiny $x_{m}$};
\node [above] at (9,0) {\tiny $2$};
\node [above] at (10,0) {\tiny $x_{m-1}$};
\node [above] at (11,0) {\tiny $1$};
\node [above] at (12,0) {\tiny $x_{m}$};
\draw [fill] (7,0) circle [radius=0.05];
\draw [fill] (9,0) circle [radius=0.05];
\draw [fill] (11,0) circle [radius=0.05];
\draw [fill] (13,0) circle [radius=0.05];

\draw [middlearrow={stealth}] (7,0)--(8,-2);
\draw [middlearrow={stealth}] (9,0)--(8,-2);
\draw [middlearrow={stealth}] (9,0)--(9,-2);
\draw [middlearrow={stealth}] (9,0)--(10,-2);
\draw [middlearrow={stealth}] (11,0)--(10,-2);
\draw [middlearrow={stealth}] (11,0)--(12,-2);
\draw [middlearrow={stealth}] (13,0)--(12,-2);

\node [left] at (8.7,-0.65) {\tiny $a_{m+1}$};
\node [right] at (9.3,-0.65) {\tiny $a_{m-1}$};
\node [left] at (10.5,-1) {\tiny $a_{m}$};
\node [right] at (11.3,-0.65) {\tiny $a_{m+1}$};

\draw [middlearrow={stealth reversed}] (8,-2)--(9,-2);
\draw [middlearrow={stealth reversed}] (9,-2)--(10,-2);
\draw [middlearrow={stealth reversed}] (10,-2)--(12,-2);

\node [below] at (8.5,-2) {\tiny $x_{m}$};
\node [below] at (9.5,-2) {\tiny $x_{m-1}$};
\node [below] at (10,-2) {\tiny $1$};
\node [below] at (11,-2) {\tiny $x_{m}$};
\draw [fill] (8,-2) circle [radius=0.05];
\draw [fill] (9,-2) circle [radius=0.05];
\draw [fill] (10,-2) circle [radius=0.05];
\draw [fill] (12,-2) circle [radius=0.05];
\end{tikzpicture}
\end{center}
\caption{A pair of adjacent $1$-vertex and $2$-vertex on $t_{i+1}$ generates a $1$-vertex on $t_{i+2}$.}
\label{A pair of adjacent $1$-vertex and $2$-vertex on $t_{i+1}$ generates a $1$-vertex on $t_{i+2}$.}
\end{figure} 

For a pair of adjacent $2$-vertices on $t_{i+1}$, we have one combination up to orientation:
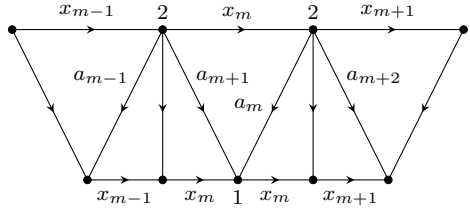
\begin{figure}[H]
\begin{center}
\begin{tikzpicture}

\draw [middlearrow={stealth}] (0,0)--(2,0);
\draw [middlearrow={stealth}] (2,0)--(4,0);
\draw [middlearrow={stealth}] (4,0)--(6,0);

\node [above] at (1,0) {\tiny $x_{m-1}$};
\node [above] at (2,0) {\tiny $2$};
\node [above] at (3,0) {\tiny $x_{m}$};
\node [above] at (4,0) {\tiny $2$};
\node [above] at (5,0) {\tiny $x_{m+1}$};
\draw [fill] (0,0) circle [radius=0.05];
\draw [fill] (2,0) circle [radius=0.05];
\draw [fill] (4,0) circle [radius=0.05];
\draw [fill] (6,0) circle [radius=0.05];

\draw [middlearrow={stealth}] (0,0)--(1,-2);
\draw [middlearrow={stealth}] (2,0)--(1,-2);
\draw [middlearrow={stealth}] (2,0)--(2,-2);
\draw [middlearrow={stealth}] (2,0)--(3,-2);
\draw [middlearrow={stealth}] (4,0)--(3,-2);
\draw [middlearrow={stealth}] (4,0)--(4,-2);
\draw [middlearrow={stealth}] (4,0)--(5,-2);
\draw [middlearrow={stealth}] (6,0)--(5,-2);

\node [left] at (1.7,-0.65) {\tiny $a_{m-1}$};
\node [right] at (2.3,-0.65) {\tiny $a_{m+1}$};
\node [left] at (3.5,-1) {\tiny $a_{m}$};
\node [right] at (4.3,-0.65) {\tiny $a_{m+2}$};

\draw [middlearrow={stealth}] (1,-2)--(2,-2);
\draw [middlearrow={stealth}] (2,-2)--(3,-2);
\draw [middlearrow={stealth}] (3,-2)--(4,-2);
\draw [middlearrow={stealth}] (4,-2)--(5,-2);

\node [below] at (1.5,-2) {\tiny $x_{m-1}$};
\node [below] at (2.5,-2) {\tiny $x_{m}$};
\node [below] at (3,-2) {\tiny $1$};
\node [below] at (3.5,-2) {\tiny $x_{m}$};
\node [below] at (4.5,-2) {\tiny $x_{m+1}$};
\draw [fill] (1,-2) circle [radius=0.05];
\draw [fill] (2,-2) circle [radius=0.05];
\draw [fill] (3,-2) circle [radius=0.05];
\draw [fill] (4,-2) circle [radius=0.05];
\draw [fill] (5,-2) circle [radius=0.05];
\end{tikzpicture}
\end{center}
\caption{A pair of adjacent $2$-vertices on $t_{i+1}$ generates a $1$-vertex on $t_{i+2}$.}
\end{figure}

For a pair of adjacent $1$-vertex and $3$-vertex on $t_{i+1}$, we have:
\begin{figure}[H] 
\begin{center}
\begin{tikzpicture}[scale=0.9]

\draw [middlearrow={stealth reversed}] (0,2)--(2,2);
\draw [middlearrow={stealth}] (0,2)--(1,0);
\draw [middlearrow={stealth}] (2,2)--(1,0);
\draw [middlearrow={stealth reversed}] (0,0)--(1,0);
\draw [middlearrow={stealth}] (1,0)--(2,0);

\node [above] at (0,2) {\tiny $3$};
\node [above] at (2,2) {\tiny $1$};
\node [above] at (1,2) {\tiny $x_{m}$};
\node [left] at (0.5,1) {\tiny $a_{m}$};
\node [right] at (1.5,1) {\tiny $a_{m+1}$};
\node [below] at (0.5,0) {\tiny $x_{m}$};
\node [below] at (1.5,0) {\tiny $x_{m+1}$};

\node [above] at (1,0.2) {\tiny $1$};

\draw [fill] (0,2) circle [radius=0.05];
\draw [fill] (2,2) circle [radius=0.05];
\draw [fill] (0,0) circle [radius=0.05];
\draw [fill] (1,0) circle [radius=0.05];
\draw [fill] (2,0) circle [radius=0.05];

\draw [middlearrow={stealth reversed}] (4,2)--(6,2);
\draw [middlearrow={stealth}] (4,2)--(5,0);
\draw [middlearrow={stealth}] (6,2)--(5,0);
\draw [middlearrow={stealth reversed}] (4,0)--(5,0);
\draw [middlearrow={stealth reversed}] (5,0)--(6,0);

\node [above] at (4,2) {\tiny $3$};
\node [above] at (6,2) {\tiny $1$};
\node [above] at (5,2) {\tiny $x_{m}$};
\node [left] at (4.5,1) {\tiny $a_{m}$};
\node [right] at (5.5,1) {\tiny $a_{m+1}$};
\node [below] at (4.5,0) {\tiny $x_{m}$};
\node [below] at (5.5,0) {\tiny $x_{m}$};

\node [above] at (5,0.2) {\tiny $1$};

\draw [fill] (4,2) circle [radius=0.05];
\draw [fill] (6,2) circle [radius=0.05];
\draw [fill] (4,0) circle [radius=0.05];
\draw [fill] (5,0) circle [radius=0.05];
\draw [fill] (6,0) circle [radius=0.05];


\node [above] at (8,2) {\tiny $1$};
\node [above] at (10,2) {\tiny $3$};
\draw [middlearrow={stealth reversed}] (8,2)--(10,2);
\draw [middlearrow={stealth}] (8,2)--(9,0);
\draw [middlearrow={stealth}] (10,2)--(9,0);
\draw [middlearrow={stealth reversed}] (8,0)--(9,0);
\draw [middlearrow={stealth reversed}] (9,0)--(10,0);

\node [above] at (9,2) {\tiny $x_{m}$};
\node [left] at (8.5,1) {\tiny $a_{m}$};
\node [right] at (9.5,1) {\tiny $a_{m+1}$};
\node [below] at (8.5,0) {\tiny $x_{m}$};
\node [below] at (9.5,0) {\tiny $x_{m}$};

\node [above] at (9,0.2) {\tiny $1$};

\draw [fill] (8,2) circle [radius=0.05];
\draw [fill] (10,2) circle [radius=0.05];
\draw [fill] (8,0) circle [radius=0.05];
\draw [fill] (9,0) circle [radius=0.05];
\draw [fill] (10,0) circle [radius=0.05];


\node [above] at (12,2) {\tiny $1$};
\node [above] at (14,2) {\tiny $3$};
\draw [middlearrow={stealth reversed}] (12,2)--(14,2);
\draw [middlearrow={stealth}] (12,2)--(13,0);
\draw [middlearrow={stealth}] (14,2)--(13,0);
\draw [middlearrow={stealth}] (12,0)--(13,0);
\draw [middlearrow={stealth reversed}] (13,0)--(14,0);

\node [above] at (13,2) {\tiny $x_{m}$};
\node [left] at (12.5,1) {\tiny $a_{m}$};
\node [right] at (13.5,1) {\tiny $a_{m+1}$};
\node [below] at (12.5,0) {\tiny $x_{m-1}$};
\node [below] at (13.5,0) {\tiny $x_{m}$};

\node [above] at (13,0.2) {\tiny $1$};

\draw [fill] (12,2) circle [radius=0.05];
\draw [fill] (14,2) circle [radius=0.05];
\draw [fill] (12,0) circle [radius=0.05];
\draw [fill] (13,0) circle [radius=0.05];
\draw [fill] (14,0) circle [radius=0.05];

\end{tikzpicture}
\end{center}
\caption{A pair of adjacent $1$-vertex and $3$-vertex on $t_{i+1}$ generates a $1$-vertex on $t_{i+2}$.}
\label{A pair of adjacent $1$-vertex and $3$-vertex on $t_{i+1}$ generates a $1$-vertex on $t_{i+2}$.}
\end{figure}
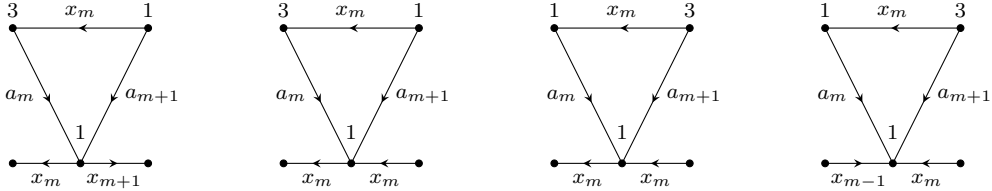

For a pair of adjacent $2$-vertex and $3$-vertex on $t_{i+1}$, since the number of combinations of edges on $t_{i+1}$ that will create $2$-vertices and $3$-vertices are limited, we have fewer cases here:
\begin{figure}[H]
\begin{center}
\begin{tikzpicture}[scale=0.9]

\draw [middlearrow={stealth}] (0,2)--(2,2);
\draw [middlearrow={stealth}] (0,2)--(1,0);
\draw [middlearrow={stealth}] (2,2)--(1,0);
\draw [middlearrow={stealth}] (0,0)--(1,0);
\draw [middlearrow={stealth}] (1,0)--(2,0);

\node [above] at (0,2) {\tiny $2$};
\node [above] at (2,2) {\tiny $3$};
\node [above] at (1,2) {\tiny $x_{m}$};
\node [left] at (0.5,1) {\tiny $a_{m+1}$};
\node [right] at (1.5,1) {\tiny $a_{m}$};
\node [below] at (0.5,0) {\tiny $x_{m}$};
\node [below] at (1.5,0) {\tiny $x_{m}$};

\node [above] at (1,0.2) {\tiny $1$};

\draw [fill] (0,2) circle [radius=0.05];
\draw [fill] (2,2) circle [radius=0.05];
\draw [fill] (0,0) circle [radius=0.05];
\draw [fill] (1,0) circle [radius=0.05];
\draw [fill] (2,0) circle [radius=0.05];

\draw [middlearrow={stealth}] (6,2)--(8,2);
\draw [middlearrow={stealth}] (6,2)--(7,0);
\draw [middlearrow={stealth}] (8,2)--(7,0);
\draw [middlearrow={stealth}] (6,0)--(7,0);
\draw [middlearrow={stealth}] (7,0)--(8,0);

\node [above] at (6,2) {\tiny $2$};
\node [above] at (8,2) {\tiny $3$};
\node [above] at (7,2) {\tiny $x_{m}$};
\node [left] at (6.5,1) {\tiny $a_{m+1}$};
\node [right] at (7.5,1) {\tiny $a_{m}$};
\node [below] at (6.5,0) {\tiny $x_{m}$};
\node [below] at (7.5,0) {\tiny $x_{m}$};

\node [above] at (7,0.2) {\tiny $1$};

\draw [fill] (6,2) circle [radius=0.05];
\draw [fill] (8,2) circle [radius=0.05];
\draw [fill] (6,0) circle [radius=0.05];
\draw [fill] (7,0) circle [radius=0.05];
\draw [fill] (8,0) circle [radius=0.05];
\end{tikzpicture}
\end{center}
\caption{A pair of adjacent $2$-vertex and $3$-vertex on $t_{i+1}$ generates a $1$-vertex on $t_{i+2}$.}
\end{figure}

Finally, for a pair of adjacent $3$-vertices, we have only one possibility up to orientation:
\begin{figure}[H]
\begin{center}
\begin{tikzpicture}[scale=0.9]

\draw [middlearrow={stealth}] (0,2)--(2,2);
\draw [middlearrow={stealth}] (0,2)--(1,0);
\draw [middlearrow={stealth}] (2,2)--(1,0);
\draw [middlearrow={stealth}] (0,0)--(1,0);
\draw [middlearrow={stealth}] (1,0)--(2,0);

\node [above] at (0,2) {\tiny $3$};
\node [above] at (2,2) {\tiny $3$};
\node [above] at (1,2) {\tiny $x_{m}$};
\node [left] at (0.5,1) {\tiny $a_{m+1}$};
\node [right] at (1.5,1) {\tiny $a_{m}$};
\node [below] at (0.5,0) {\tiny $x_{m}$};
\node [below] at (1.5,0) {\tiny $x_{m}$};

\node [above] at (1,0.2) {\tiny $1$};

\draw [fill] (0,2) circle [radius=0.05];
\draw [fill] (2,2) circle [radius=0.05];
\draw [fill] (0,0) circle [radius=0.05];
\draw [fill] (1,0) circle [radius=0.05];
\draw [fill] (2,0) circle [radius=0.05];
\end{tikzpicture}
\end{center}
\caption{A pair of adjacent $3$-vertex on $t_{i+1}$ generates a $1$-vertex on $t_{i+2}$.}
\end{figure}
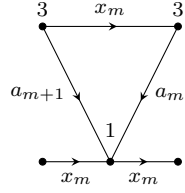

\item[(3)] Since we assume that there are no $0$-vertices on $t_{i}$, the number of vertices on $t_{i+1}$ is no less than the number of vertices on $t_{i}$, and the number of vertices on $t_{i+2}$ is also no less than the number of vertices on $t_{i+1}$. Thus, we have
\begin{align*}
|t_{i}|&=(\text{number of vertices on $t_{i}$})-1                                  \\
       &\leq(\text{number of vertices on $t_{i+1}$})-1                             \\
       &=|t_{i+1}|  
\end{align*}
and 
\begin{align*}
|t_{i+1}|&=(\text{number of vertices on $t_{i+1}$})-1                                  \\
       &\leq(\text{number of vertices on $t_{i+2}$})-1                             \\
       &=|t_{i+2}|.
\end{align*}
Hence, $|t_{i}|\leq|t_{i+1}|\leq|t_{i+2}|$. 
\end{enumerate}
\end{proof}

\begin{lemma}\label{two consecutive corridor with different orientations}
Let $C_{i}$ and $C_{i+1}$ be two consecutive corridors in a stack $T$ as in the Lemma \ref{two consecutive corridor with the same orientation}. Let the arrows on the edges labeled by $u'_{i},u''_{i}$ and $u'_{i+1},u''_{i+1}$ have different orientations. There are two cases, as shown in Figure \ref{Edges $u'_{i},u''_{i}$ and $u'_{i+1},u''_{i+1}$ have different orientations.}. Assume that there are no $0$-vertices on $t_{i}$. Then $|t'_{i+2}|\leq|t_{i+2}|$, where $t_{i+2}$ is in Figure \ref{picture of two consecutive corridor with the same orientation}.
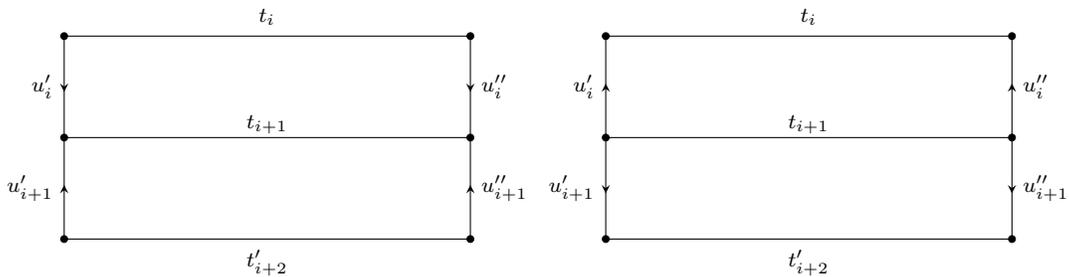
\begin{figure}[H]
\begin{center}
\begin{tikzpicture}[scale=0.9]

\node [above] at (3,3) {\tiny $t_{i}$};
\draw (0,3)--(6,3);

\draw [fill] (0,3) circle [radius=0.05];
\draw [fill] (6,3) circle [radius=0.05];

\node [left] at (0,2.25) {\tiny $u'_{i}$};
\draw [middlearrow={stealth}] (0,3)--(0,1.5);
\draw [middlearrow={stealth}] (6,3)--(6,1.5);
\node [right] at (6,2.25) {\tiny $u''_{i}$};

\draw (0,1.5)--(6,1.5);

\node [above] at (3,1.4) {\tiny $t_{i+1}$};
\draw [fill] (0,1.5) circle [radius=0.05];
\draw [fill] (6,1.5) circle [radius=0.05];

\node [left] at (0,0.75) {\tiny $u'_{i+1}$};
\draw [middlearrow={stealth reversed}] (0,1.5)--(0,0);
\draw [middlearrow={stealth reversed}] (6,1.5)--(6,0);
\node [right] at (6,0.75) {\tiny $u''_{i+1}$};

\draw (0,0)--(6,0);

\node [below] at (3,0) {\tiny $t'_{i+2}$};
\draw [fill] (0,0) circle [radius=0.05];
\draw [fill] (6,0) circle [radius=0.05];



\node [above] at (11,3) {\tiny $t_{i}$};
\draw (8,3)--(14,3);

\draw [fill] (8,3) circle [radius=0.05];
\draw [fill] (14,3) circle [radius=0.05];

\node [left] at (8,2.25) {\tiny $u'_{i}$};
\draw [middlearrow={stealth reversed}] (8,3)--(8,1.5);
\draw [middlearrow={stealth reversed}] (14,3)--(14,1.5);
\node [right] at (14,2.25) {\tiny $u''_{i}$};

\draw (8,1.5)--(14,1.5);

\node [above] at (11,1.4) {\tiny $t_{i+1}$};
\draw [fill] (8,1.5) circle [radius=0.05];
\draw [fill] (14,1.5) circle [radius=0.05];

\node [left] at (8,0.75) {\tiny $u'_{i+1}$};
\draw [middlearrow={stealth}] (8,1.5)--(8,0);
\draw [middlearrow={stealth}] (14,1.5)--(14,0);
\node [right] at (14,0.75) {\tiny $u''_{i+1}$};

\draw (8,0)--(14,0);

\node [below] at (11,0) {\tiny $t'_{i+2}$};
\draw [fill] (8,0) circle [radius=0.05];
\draw [fill] (14,0) circle [radius=0.05];

\end{tikzpicture}
\end{center}
\caption{Edges $u'_{i},u''_{i}$ and $u'_{i+1},u''_{i+1}$ have different orientations.}
\label{Edges $u'_{i},u''_{i}$ and $u'_{i+1},u''_{i+1}$ have different orientations.}
\end{figure}
\end{lemma}

\begin{proof} 
We prove the case when $u'_{i},u''_{i},u'_{i+1},u''_{i+1}$ are pointing toward $t_{i+i}$. The other case is similar. We claim that in this case, $C_{i}$ and $C_{i+1}$ cannot be both $\alpha$-corridors or $\beta$-corridors. Suppose $C_{i}$ and $C_{i+1}$ are both $\alpha$-corridors (respectively $\beta$-corridors), then there are two triangles that share a unique common edge on $t_{i+1}$:
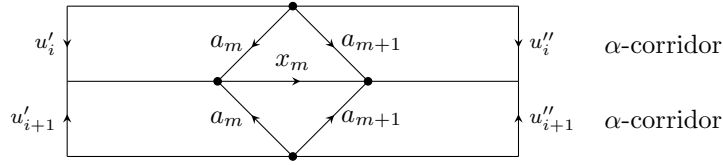
\begin{figure}[H]
\begin{center}
\begin{tikzpicture}

\draw (-2,1)--(4,1);

\draw [fill] (1,1) circle [radius=0.05];

\draw [middlearrow={stealth}] (-2,1)--(-2,0);
\draw [middlearrow={stealth}] (1,1)--(0,0);
\draw [middlearrow={stealth}] (1,1)--(2,0);
\draw [middlearrow={stealth}] (4,1)--(4,0);

\node [left] at (-2,0.5) {\tiny $u'_{i}$};
\node [left] at (0.5,0.5) {\footnotesize $a_{m}$};
\node [right] at (1.5,0.5) {\footnotesize $a_{m+1}$};
\node [right] at (4,0.5) {\tiny $u''_{i}$};
\node [right] at (5,0.5) {\footnotesize $\alpha$-corridor};

\draw (-2,0)--(0,0);
\draw [middlearrow={stealth}] (0,0)--(2,0);
\draw (2,0)--(4,0);

\draw [fill] (0,0) circle [radius=0.05];
\draw [fill] (2,0) circle [radius=0.05];
\node [above] at (1,0) {\footnotesize $x_{m}$};

\node [right] at (5,-0.5) {\footnotesize $\alpha$-corridor};

\draw [middlearrow={stealth reversed}] (-2,0)--(-2,-1);
\draw [middlearrow={stealth}] (1,-1)--(0,0);
\draw [middlearrow={stealth}] (1,-1)--(2,0);
\draw [middlearrow={stealth reversed}](4,0)--(4,-1);

\draw (-2,-1)--(4,-1);

\node [left] at (-2,-0.5) {\tiny $u'_{i+1}$};
\node [left] at (0.5,-0.5) {\footnotesize $a_{m}$};
\node [right] at (1.5,-0.5) {\footnotesize $a_{m+1}$};
\node [right] at (4,-0.5) {\tiny $u''_{i+1}$};

\draw [fill] (1,-1) circle [radius=0.05];
\end{tikzpicture}
\end{center} 
\caption{Two triangles from different $\alpha$-corridors with opposite orientations.}
\label{two triangles from different alpha-corridors with opposite orientations.}
\end{figure}
\noindent The situation as shown in Figure \ref{two triangles from different alpha-corridors with opposite orientations.} contradicts the fact that $T$ is part of a minimal van Kampen diagram since we can obtain a smaller van Kampen diagram by canceling these two triangles. This proves the claim.

Now suppose $C_{i}$ is an $\alpha$-corridor and $C_{i+1}$ is a $\beta$-corridor. By definition, each $j$-vertex on $t_{i}$ is adjacent to $j+1$ consecutive distinct vertices on $t_{i+1}$: 
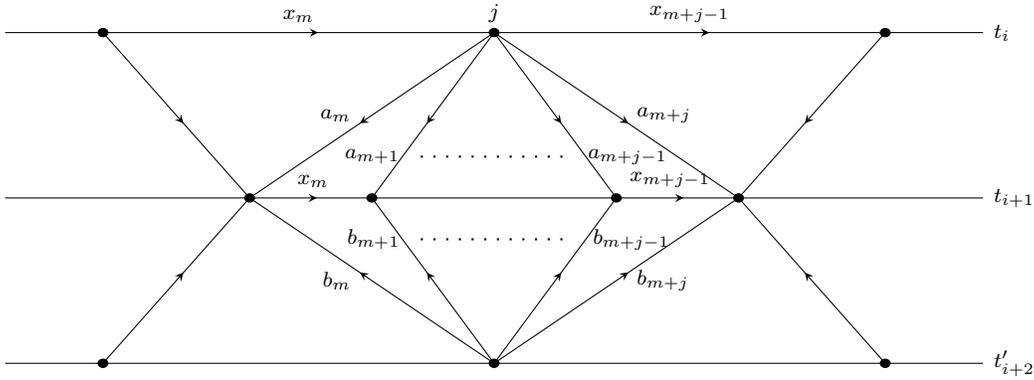
\begin{figure}[H]
\begin{center}
\begin{tikzpicture}[xscale=0.65, yscale=0.55]

\draw (-10,4)--(-8,4);
\draw [middlearrow={stealth}] (-8,4)--(0,4);
\draw [middlearrow={stealth}] (0,4)--(8,4);
\draw (8,4)--(10,4);

\node [above] at (-4,4) {\tiny $x_{m}$};
\node [above] at (0,4) {\tiny $j$};
\node [above] at (4,4) {\tiny $x_{m+j-1}$};
\node [right] at (10,4) {\tiny $t_{i}$};

\draw [fill] (-8,4) circle [radius=0.1];
\draw [fill] (0,4) circle [radius=0.1];
\draw [fill] (8,4) circle [radius=0.1];

\draw [middlearrow={stealth}] (-8,4)--(-5,0);
\draw [middlearrow={stealth}] (0,4)--(-5,0);
\draw [middlearrow={stealth}] (0,4)--(-2.5,0);
\draw [thick, loosely dotted] (-1.5,1)--(1.5,1);
\draw [middlearrow={stealth}] (0,4)--(2.5,0);
\draw [middlearrow={stealth}] (0,4)--(5,0);
\draw [middlearrow={stealth}] (8,4)--(5,0);

\node [left] at (-2.7,2) {\tiny $a_{m}$};
\node [left] at (-1.7,1) {\tiny $a_{m+1}$};
\node [right] at (1.7,1) {\tiny $a_{m+j-1}$};
\node [right] at (2.7,2) {\tiny $a_{m+j}$};

\draw (-10,0)--(-5,0);
\draw [middlearrow={stealth}] (-5,0)--(-2.5,0);
\draw (-2.5,0)--(2.5,0);
\draw [middlearrow={stealth}] (2.5,0)--(5,0);
\draw (5,0)--(10,0);

\node [above] at (-3.7,0) {\tiny $x_{m}$};
\node [above] at (3.6,0) {\tiny $x_{m+j-1}$};
\node [right] at (10,0) {\tiny $t_{i+1}$};

\draw [fill] (-5,0) circle [radius=0.1];
\draw [fill] (-2.5,0) circle [radius=0.1];
\draw [fill] (2.5,0) circle [radius=0.1];
\draw [fill] (5,0) circle [radius=0.1];


\draw [middlearrow={stealth}] (-8,-4)--(-5,0);
\draw [middlearrow={stealth}] (0,-4)--(-5,0);
\draw [middlearrow={stealth}] (0,-4)--(-2.5,0);
\draw [thick, loosely dotted] (-1.5,-1)--(1.5,-1);
\draw [middlearrow={stealth}] (0,-4)--(2.5,0);
\draw [middlearrow={stealth}] (0,-4)--(5,0);
\draw [middlearrow={stealth}] (8,-4)--(5,0);

\node [left] at (-2.7,-2) {\tiny $b_{m}$};
\node [left] at (-1.7,-1) {\tiny $b_{m+1}$};
\node [right] at (1.8,-1) {\tiny $b_{m+j-1}$};
\node [right] at (2.7,-2) {\tiny $b_{m+j}$};

\draw (-10,-4)--(-8,-4);
\draw (-8,-4)--(0,-4);
\draw (0,-4)--(8,-4);
\draw (8,-4)--(10,-4);

\node [right] at (10,-4) {\tiny $t'_{i+2}$};

\draw [fill] (-8,-4) circle [radius=0.1];
\draw [fill] (0,-4) circle [radius=0.1];
\draw [fill] (8,-4) circle [radius=0.1];

\end{tikzpicture}
\end{center}
\caption{Every pair of adjacent vertices of the distinct $j+1$ vertices on $t_{i+1}$ generates the same vertex on $t'_{i+2}$.}
\label{alpha corridor sits above beta corridor}
\end{figure}
\noindent From Figure \ref{alpha corridor sits above beta corridor} we see that every pair of adjacent vertices of these $j+1$ consecutive distinct vertices on $t_{i+1}$ generates the same vertex on $t'_{i+2}$ since $C_{i+1}$ is a $\beta$-corridor. That is, each $j$-vertex on $t_{i}$ creates $j-1$ vertices on $t_{i+1}$ and they are $0$-vertices. It follows that the number of vertices on $t'_{i+2}$ is less than the numbers of vertices on $t_{i+2}$ in Lemma \ref{two consecutive corridor with the same orientation}. Thus, we have
\begin{align*}
|t'_{i+2}|&=\text{(number of vertices on $t'_{i+2}$)}-1                    \\
           &\leq\text{(number of vertices on $t_{i+2}$)}-1     \\
           &=|t_{i+2}|.
\end{align*}
This completes the proof.
\end{proof}

Now, we prove Lemma \ref{stacks are cubic}.

\begin{proof}[Proof of Lemma \ref{stacks are cubic}.]
Let $T$ be a stack with labelings as shown in Figure \ref{Boundary words of corridors in a stack $T$.}. Recall that $t_{0},\cdots,t_{h}$ are words in the free group $F(x_{1},\cdots,x_{k})$ and $u'_{i},u''_{i}$ are both letters either in the corridor scheme $\alpha=\lbrace a_{1},\cdots,a_{k+1}\rbrace$ or $\beta=\lbrace b_{1},\cdots,b_{k+1}\rbrace$. Since the area of a stack $T$ is the sum of the areas of the corridors that contained in $T$, by Lemma \ref{Properties of one corridor in a stack} we have
\begin{align*}
\text{Area}(T)=|t_{0}|+2|t_{1}|+\cdots+2|t_{h-1}|+|t_{h}|\leq 2\left(|t_{0}|+\cdots+|t_{h}|\right).
\end{align*}
\noindent Thus, upper bounds on $|t_{0}|,\cdots,|t_{n}|$ give an upper on $\mathrm{Area}(T)$. In order to obtain an upper bound on $|t_{i}|$, we need a few assumptions. Recall that for a fixed $k$, each vertex of $T$ is at most a $k$-vertex. The assumptions we need here are
\begin{enumerate}
\item[(1)] the vertices on $t_{0}$ and the two vertices at the two ends of $t_{i}$ are $k$-vertices;
\item[(2)] the arrows on $u'_{i},u''_{i}$, $i=1,\cdots,h$, are pointing away from $t_{0}$ (Lemma \ref{two consecutive corridor with the same orientation} and Lemma \ref{two consecutive corridor with different orientations}).
\end{enumerate}
The first assumption gives us an upper bound on $|t_{0}|$. From Lemma \ref{Properties of one corridor in a stack} we know that the length of $t_{i+1}$ depends on the types of the vertices on $t_{i}$. The second assumption ensures that $|t_{i}|\leq|t_{i+1}|\leq|t_{i+2}|$. We compute $|t_{i}|$ as follows.

For $|t_{1}|$, since $|t_{0}|=l$ and all the vertices on $t_{0}$ are $k$-vertex, we have $|t_{1}|=k(l+1)$. 

On $t_{1}$, there are two $k$-vertices at the two ends. Lemma \ref{two consecutive corridor with the same orientation} tells us that every pair of adjacent vertices on $t_{0}$ generates either a $1$-vertex or a $3$-vertex on $t_{1}$, but we may assume that they are all $3$-vertices so that we can get the largest possible $|t_{2}|$. So the number of $3$-vertices on $t_{1}$ is $l$; other vertices on $t_{1}$ are $2$-vertices and there are $(k-1)(l+1)$ of them. 	Knowing the types of vertices on $t_{1}$ gives the length of $t_{2}$:
\begin{align*}
|t_{2}|=2\cdot k+l\cdot 3+(k-1)(l+1)\cdot 2=2kl+l+4k-2
\end{align*}

For $i\geq 2$, every pair of adjacent vertices on $t_{i-1}$ generates a $1$-vertex on $t_{i}$ by Lemma \ref{two consecutive corridor with the same orientation}, so the number of $1$-vertices on $t_{i}$ is $|t_{i-1}|$. Every $2$-vertex on $t_{i-1}$ creates a $2$-vertex on $t_{i}$ and the two $k$-vertices at the ends of $t_{i-1}$ creates $(k-1)$ $2$-vertices on $t_{i}$. So the number of $2$-vertices on $t_{i}$ is the number of $2$-vertices on $t_{i-1}$ plus $2(k-1)$:
\begin{align*}
(kl+l+k-1)+(i-1)\cdot2(k-1)=kl+l+(2i-1)k-2i+1.
\end{align*}
Having the information of the vertices on $t_{i}$ we get
\begin{align*}
|t_{i+1}|&=|t_{i-1}|\cdot 1+\left[kl+l+(2i-1)k-2i+1\right]\cdot 2+2\cdot k  \\
         &=|t_{i-1}|+2kl+2l+4ik-4i+2
\end{align*}
for $i=2,\cdots,h-1$. Let
\begin{align*}
d(i)=|t_{i+1}|-|t_{i-1}|=2kl+2l+4ik-4i+2, \ \ \ i=2,\cdots,h-1,
\end{align*}
then $\lbrace d(i)\rbrace$ is an arithmetic sequence whose difference is $4k-4$ and
$$
d(i+2)-d(i)=8k-8.
$$

When $i$ is even, we have
\begin{align*}
|t_{i+1}|
&=|t_{1}|+d(2)+d(4)+\cdots+d(i)              \\                                                 
&=|t_{1}|+\frac{i}{2}d(2)+i(i-2)(k-1),
\end{align*}
and when $i$ is odd, we have
\begin{align*}
|t_{i+1}|
&=|t_{2}|+d(3)+d(5)+\cdots+d(i)                                                 \\
&=|t_{2}|+\left(\frac{i-1}{2}\right)d(3)+(i-1)(i-3)(k-1).
\end{align*}

When $h$ is odd, we have
\begin{align*}
\text{Area}(T)
\leq 2\sum^{h}_{i=1}|t_{i}|
=2\left(|t_{0}|+|t_{1}|+|t_{2}|+\sum^{h-1}_{\substack{i=2 \\ \text{$i$ is even}}}|t_{i+1}|+\sum^{h-2}_{\substack{i=3 \\ \text{$i$ is odd}}}|t_{i+1}|\right).  
\end{align*}
Furthermore, 
\begin{align*}
\sum^{h-1}_{\substack{i=2 \\ \text{$i$ is even}}}|t_{i+1}|
&=\sum^{h-1}_{\substack{i=2 \\ \text{$i$ is even}}}\left[|t_{1}|+\frac{i}{2}d(2)+i(i-2)(k-1)\right]           \\
&\leq h|t_{1}|+\left[d(2)-2(k-1)\right]h^{2}+(k-1)h^{3}  \\
&\leq 12klh^{2}+kh^{3}.
\end{align*}  
and
\begin{align*}
\sum^{h-2}_{\substack{i=3 \\ \text{$i$ is odd}}}|t_{i+1}|
&=\sum^{h-2}_{\substack{i=3 \\ \text{$i$ is odd}}}\left[|t_{2}|+\left(\frac{i-1}{2}\right)d(3)+(i-1)(i-3)(k-1)\right]     \\                              
&\leq h|t_{2}|+d(3)h^{2}+kh^{3}+kh^{2}+2h         \\
&\leq 24klh^{2}+kh^{3}.                                                                 
\end{align*}
Thus, 
\begin{align*}
\text{Area}(T)
&\leq 2\left[|t_{0}|+|t_{1}|+|t_{2}|+(12klh^{2}+kh^{3})+(24klh^{2}+kh^{3})\right]      \\
&\leq 92k(h^{3}+lh^{2}).
\end{align*}
When $h$ is even, the computation is similar. Hence, $\text{Area}(T)\leq K(h^{3}+lh^{2})$, where $K$ is a positive constant which does not depend on $l$ and $h$.
\end{proof}

\begin{lemma}\label{The suspension of a path of length at least 3 is at most cubic}
Let $\Gamma$ be the graph as shown in Figure \ref{The suspension of a path of length at least $3$ with orientation.}. Then $\delta_{H_{\Gamma}}(n)\preceq n^{3}$. 
\end{lemma}

\begin{proof}
Let $w$ be a freely reduced word of length at most $n$ that represents the identity in $H_{\Gamma}$. Let $\Delta$ be a minimal van Kampen diagram for $w$. Choose corridor schemes $\alpha=\lbrace a_{1},\cdots,a_{k+1}\rbrace$ and $\beta=\lbrace b_{1},\cdots,b_{k+1}\rbrace$ for $\Gamma$, then the van Kampen diagram $\Delta$ consists of stacks, $T_{1},\cdots,T_{m}$. Denote the two legs of $T_{i}$ by $U'_{i}$ and $U''_{i}$, and the height of $T_{i}$ by $|U'_{i}|=|U''_{i}|=h_{i}$, $i=1,\cdots,m$. Note that $U'_{i}$ and $U''_{i}$ are words on $\partial\Delta$ that consist of letters in the corridor schemes $\alpha$ and $\beta$. The boundary word $w$ is a cyclic permutation of the following word, and we also denoted it by $w$:
\begin{align*}
w=A_{1}B_{1}A_{2}B_{2}\cdots A_{r}B_{s},
\end{align*}
where each of the words $A_{1},\cdots,A_{r}$ consists of one or more words from $U'_{1},\cdots,U'_{m}$ and $U''_{1},\cdots,U''_{m}$; each of the words $B_{1},\cdots,B_{s}$ is a word in the free group $F(x_{1},\cdots,x_{k})$. The length of each of the words $A_{1},\cdots,A_{r}$ is either the height of a single stack or the sum of the heights of multiple stacks; each of the words $B_{1},\cdots,B_{s}$ is a base of a stack or part of a base of a stack. Denote the length of the words $B_{1},\cdots,B_{s}$ by $l_{1},\cdots,l_{s}$, respectively. Let 
\begin{align*}
h=|A_{1}|+\cdots+|A_{r}|=2(h_{1}+\cdots+h_{m})=2\sum^{m}_{i=1}h_{i}
\end{align*}
and
\begin{align*}
l=|B_{1}|+\cdots+|B_{s}|=\sum^{s}_{i=1}l_{i}.
\end{align*}

Consider a stack $T'$ whose top is the word $B_{1}\cdots B_{s}$ of length $l$ and whose legs are $U'_{1}\cdots U'_{m}$ and $U''_{1}\cdots U''_{m}$ and the height $h=|U'_{1}\cdots U'_{m}|=|U''_{1}\cdots U''_{m}|$, assuming that all the the arrows on $U'_{i},U''_{i}$ are pointing away from the top; see Figure \ref{The stack $T'$ with labels on the top and the legs.}. 
\begin{figure}[H]
\begin{center}
\begin{tikzpicture}[scale=0.4]

\draw (0,10)--(2,10);
\draw (2,10)--(4,10);
\draw (4,10)--(8,10);
\draw (8,10)--(10,10);

\node [above] at (1,10) {\tiny $B_{1}$};
\node [above] at (3,10) {\tiny $B_{2}$};
\node [above] at (9,10) {\tiny $B_{s}$};

\draw [fill] (0,10) circle [radius=0.1];
\draw [fill] (2,10) circle [radius=0.1];
\draw [fill] (4,10) circle [radius=0.1];
\draw [fill] (8,10) circle [radius=0.1];
\draw [fill] (10,10) circle [radius=0.1];

\draw [middlearrow={stealth}] (0,10)--(-1,8);
\draw [middlearrow={stealth}] (-1,8)--(-2,6);
\draw (-2,6)--(-4,2);
\draw [middlearrow={stealth}] (-4,2)--(-5,0);

\node [left] at (-0.5,9){\tiny $U'_{1}$};
\node [left] at (-1.5,7){\tiny $U'_{2}$};
\node [left] at (-4.5,1){\tiny $U'_{m}$};

\draw [fill] (-1,8) circle [radius=0.1];
\draw [fill] (-2,6) circle [radius=0.1];
\draw [fill] (-4,2) circle [radius=0.1];
\draw [fill] (-5,0) circle [radius=0.1];

\draw [middlearrow={stealth}] (10,10)--(11,8);
\draw [middlearrow={stealth}] (11,8)--(12,6);
\draw (12,6)--(14,2);
\draw [middlearrow={stealth}] (14,2)--(15,0);

\node [right] at (10.5,9){\tiny $U''_{1}$};
\node [right] at (11.5,7){\tiny $U''_{2}$};
\node [right] at (14.5,1){\tiny $U''_{m}$};

\draw [fill] (11,8) circle [radius=0.1];
\draw [fill] (12,6) circle [radius=0.1];
\draw [fill] (14,2) circle [radius=0.1];
\draw [fill] (15,0) circle [radius=0.1];

\draw (-1,8)--(11,8);
\draw (-2,6)--(12,6);
\draw (-4,2)--(14,2);

\draw [thick, loosely dotted] (5,5.5)--(5,2.5);

\draw (-5,0)--(15,0);

\end{tikzpicture}	
\end{center}
\caption{The stack $T'$.}
\label{The stack $T'$ with labels on the top and the legs.}
\end{figure}

There are three possible cases that the van Kampen diagram $\Delta$ could be. The first case is that every stack has a base that is part of $\partial\Delta$, as shown in Figure \ref{Case 1: every stack has a base as part of Delta.}. The second case is that every stack has a base that is either on  $\partial\Delta$, or part of it is on $\partial\Delta$, as shown in Figure \ref{Case 2: at least one of the stacks such that part of its bases is part of Delta.}. The third case is that there is a stack whose bases are not part of $\partial\Delta$, as shown in Figure \ref{Case 3: at least one of the stacks whose bases are not part of Delta.}. We show that $\text{Area}(\Delta)\leq\text{Area}(T')$ in each of these three cases.
\begin{figure}[H]
\begin{center}
\begin{tikzpicture}[scale=0.7]
\draw (0,0)--(5,0)--(5,4)--(0,4)--(0,0);

\draw [fill] (0,0) circle [radius=0.1];
\draw [fill] (5,0) circle [radius=0.1];
\draw [fill] (5,4) circle [radius=0.1];
\draw [fill] (0,4) circle [radius=0.1];

\node at (2.5,2) {\tiny $T_{1}$};
\node [above] at (2.5,4) {\tiny $B_{1}$};
\node [left] at (0,2) {\tiny $U'_{1}$};
\node [right] at (5,2) {\tiny $U''_{1}$};

\draw (-5,0)--(2,0)--(2,-5)--(0,-5)--(-5,-2)--(-5,0);

\draw [fill] (-5,0) circle [radius=0.1];
\draw [fill] (2,-5) circle [radius=0.1];
\draw [fill] (0,-5) circle [radius=0.1];
\draw [fill] (-5,-2) circle [radius=0.1];

\node at (-1.5,-2.5) {\tiny $T_{2}$};
\node [above] at (-2.5,0) {\tiny $B_{4}$};
\node [left] at (-5,-1) {\tiny $U'_{2}$};
\node [below left] at (-2.5,-3.5) {\tiny $B_{2}$};
\node [below] at (1,-5) {\tiny $U''_{2}$};

\draw (2,0)--(7,0)--(7,-3)--(5,-7)--(2,-7)--(2,0);

\draw [fill] (7,0) circle [radius=0.1];
\draw [fill] (7,-3) circle [radius=0.1];
\draw [fill] (5,-7) circle [radius=0.1];
\draw [fill] (2,-7) circle [radius=0.1];

\node at (4.5,-3.5) {\tiny $T_{3}$};
\node [left] at (2,-6) {\tiny $B_{5}$};
\node [below] at (3.5,-7) {\tiny $U'_{3}$};
\node [below right] at (6,-5) {\tiny $B_{3}$};
\node [right] at (7,-1.5) {\tiny $U''_{3}$};
\node [above] at (6,0) {\tiny $B_{6}$};
\end{tikzpicture}
\end{center}
\caption{Case 1: Every stack $T_{i}$ has a base $B_{i}$ that is part of $\partial\Delta$.}
\label{Case 1: every stack has a base as part of Delta.}
\end{figure}
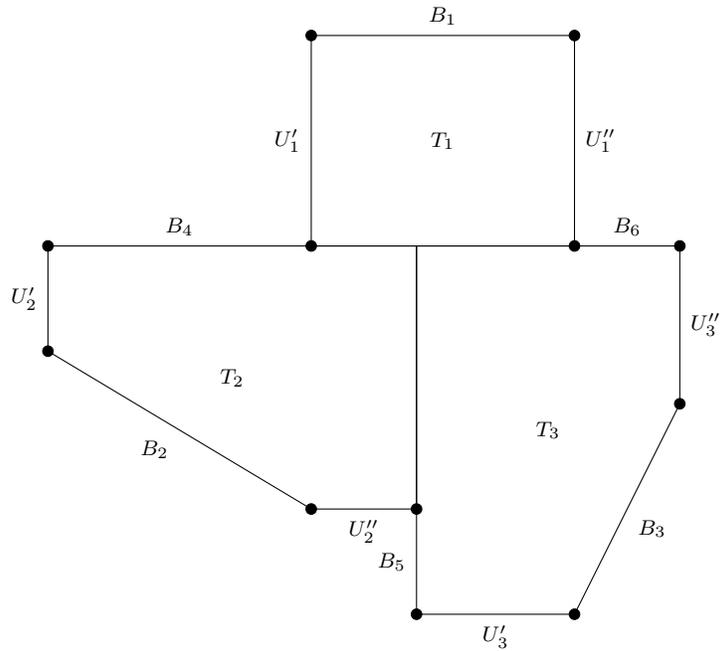
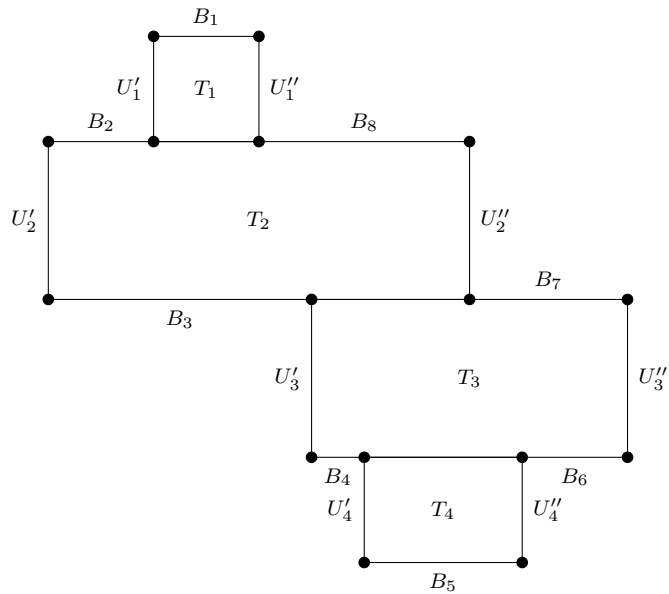
\begin{figure}[H]
\begin{center}
\begin{tikzpicture}[scale=0.7]

\draw (2,3)--(4,3)--(4,5)--(2,5)--(2,3);

\draw [fill] (2,3) circle [radius=0.1];
\draw [fill] (4,3) circle [radius=0.1];
\draw [fill] (4,5) circle [radius=0.1];
\draw [fill] (2,5) circle [radius=0.1];

\node at (3,4) {\tiny $T_{1}$};
\node [above] at (3,5) {\tiny $B_{1}$};
\node [left] at (2,4) {\tiny $U'_{1}$};
\node [right] at (4,4) {\tiny $U''_{1}$};

\draw (0,0)--(8,0)--(8,3)--(0,3)--(0,0);

\draw [fill] (0,3) circle [radius=0.1];
\draw [fill] (0,0) circle [radius=0.1];
\draw [fill] (5,0) circle [radius=0.1];
\draw [fill] (8,0) circle [radius=0.1];
\draw [fill] (8,3) circle [radius=0.1];

\node at (4,1.5) {\tiny $T_{2}$};
\node [above] at (1,3) {\tiny $B_{2}$};
\node [left] at (0,1.5) {\tiny $U'_{2}$};
\node [below] at (2.5,0) {\tiny $B_{3}$};
\node [right] at (8,1.5) {\tiny $U''_{2}$};
\node [above] at (6,3) {\tiny $B_{8}$};

\draw (5,0)--(11,0)--(11,-3)--(5,-3)--(5,0);

\draw [fill] (11,0) circle [radius=0.1];
\draw [fill] (11,-3) circle [radius=0.1];
\draw [fill] (9,-3) circle [radius=0.1];
\draw [fill] (5,-3) circle [radius=0.1];
\draw [fill] (6,-3) circle [radius=0.1];

\node at (8,-1.5) {\tiny $T_{3}$};
\node [left] at (5,-1.5) {\tiny $U'_{3}$};
\node [below] at (5.5,-3) {\tiny $B_{4}$};
\node [below] at (10,-3) {\tiny $B_{6}$};
\node [right] at (11,-1.5) {\tiny $U''_{3}$};
\node [above] at (9.5,0) {\tiny $B_{7}$};

\draw (6,-3)--(9,-3)--(9,-5)--(6,-5)--(6,-3);

\draw [fill] (6,-5) circle [radius=0.1];
\draw [fill] (9,-5) circle [radius=0.1];

\node at (7.5,-4) {\tiny $T_{4}$};
\node [left] at (6,-4) {\tiny $U'_{4}$};
\node [below] at (7.5,-5) {\tiny $B_{5}$};
\node [right] at (9,-4) {\tiny $U''_{4}$};
\end{tikzpicture}
\end{center}
\caption{Case 2: At least one of the stacks such that part of its bases is part of $\partial\Delta$.}
\label{Case 2: at least one of the stacks such that part of its bases is part of Delta.}
\end{figure}

\begin{figure}[H]
\begin{center}
\begin{tikzpicture}[scale=0.7]
\draw (0,5)--(6,5)--(6,8)--(0,8)--(0,5);

\draw [fill] (0,5) circle [radius=0.1];
\draw [fill] (2,5) circle [radius=0.1];
\draw [fill] (4,5) circle [radius=0.1];
\draw [fill] (6,5) circle [radius=0.1];
\draw [fill] (6,8) circle [radius=0.1];
\draw [fill] (0,8) circle [radius=0.1];

\node at (3,6.5) {\tiny $T_{1}$};
\node [above] at (3,8) {\tiny $B_{1}$};
\node [left] at (0,6.5) {\tiny $U'_{1}$};
\node [below] at (1,5) {\tiny $B_{2}$};
\node [below] at (5,5) {\tiny $B_{6}$};
\node [right] at (6,6.5) {\tiny $U''_{1}$};

\draw (2,3)--(4,3)--(4,5)--(2,5)--(2,3);

\node at (3,4) {\tiny $T_{2}$};
\node [left] at (2,4) {\tiny $U'_{2}$};
\node [right] at (4,4) {\tiny $U''_{2}$};


\draw (0,0)--(6,0)--(6,3)--(0,3)--(0,0);

\draw [fill] (2,3) circle [radius=0.1];
\draw [fill] (0,3) circle [radius=0.1];
\draw [fill] (0,0) circle [radius=0.1];
\draw [fill] (6,0) circle [radius=0.1];
\draw [fill] (6,3) circle [radius=0.1];
\draw [fill] (4,3) circle [radius=0.1];

\node at (3,1.5) {\tiny $T_{3}$};
\node [above] at (1,3) {\tiny $B_{3}$};
\node [left] at (0,1.5) {\tiny $U'_{3}$};
\node [below] at (3,0) {\tiny $B_{4}$};
\node [right] at (6,1.5) {\tiny $U''_{3}$};
\node [above] at (5,3) {\tiny $B_{5}$};

\end{tikzpicture}
\end{center}
\caption{Case 3: At least one of the stacks whose bases are not part of $\partial\Delta$.}
\label{Case 3: at least one of the stacks whose bases are not part of Delta.}
\end{figure}
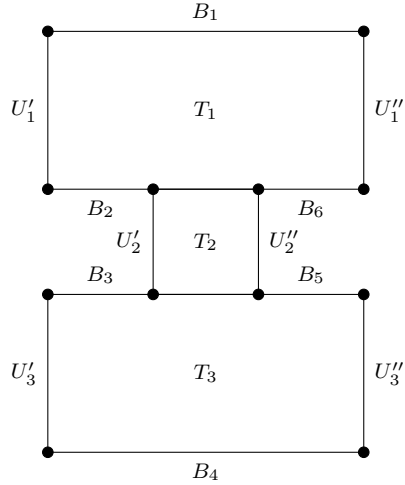

\textbf{Case 1.} Recall that each of $A_{i}$ is a leg of one or more stacks; each $B_{i}$ is part of a base of a stack or a base of a stack. By assumption, each stack $T_{i}$ has a base that is on $\partial\Delta$, say $B_{i}$ is a base of the stack $T_{i}$. If $B_{i}$ is the top of the stack $T_{i}$, then the top of the stack $T'_{i}$ in Figure \ref{Disjoing substactk T'_i and T'_j in T'.} is longer than $B_{i}$. Thus, $\text{Area}(T_{i})\leq\text{Area}(T'_{i})$. If $B_{i}$ is the bottom of the stack $T_{i}$, then consider a stack whose heights are $U'_{i},U''_{i}$ and whose top is $B_{i}$. The area of this stack is obviously greater than the area of $T_{i}$, but less than the area of $T'_{i}$ in Figure \ref{Disjoing substactk T'_i and T'_j in T'.}. Thus, for each stack $T_{i}$, $i=1,\cdots,m$, there is a stack $T'_{i}$ in $T'$ satisfying $\text{Area}(T_{i})\leq\text{Area}(T'_{i})$. These substacks $T'_{1},\cdots,T'_{m}$ are disjoint inside $T'$ because their legs are disjoint. Each substack $T'_{}$ is at different height inside $T'$, as shown in Figure \ref{Disjoing substactk T'_i and T'_j in T'.}:
\begin{figure}[H]
\begin{center}
\begin{tikzpicture}[scale=0.4]

\draw (0,14)--(16,14);
\draw [fill] (2,14) circle [radius=0.1];
\draw [fill] (4,14) circle [radius=0.1];
\node [above] at (3,14) {\tiny $B_{i}$};

\draw [fill] (10,14) circle [radius=0.1];
\draw [fill] (12,14) circle [radius=0.1];
\node [above] at (11,14) {\tiny $B_{j}$};

\draw (0,14)--(-1,12);
\draw (-1,12)--(-2,10);
\draw [middlearrow={stealth}] (-2,10)--(-3,8);
\draw [fill] (-2,10) circle [radius=0.1];
\draw [fill] (-3,8) circle [radius=0.1];
\node [left] at (-2.5,9) {\tiny $U'_{i}$};
\draw (-3,8)--(-4,6);
\draw [middlearrow={stealth}] (-4,6)--(-6,2);
\draw [fill] (-4,6) circle [radius=0.1];
\draw [fill] (-6,2) circle [radius=0.1];
\node [left] at (-5,4) {\tiny $U'_{j}$};
\draw (-6,2)--(-7,0);

\draw (16,14)--(17,12);
\draw (17,12)--(18,10);
\draw [middlearrow={stealth}] (18,10)--(19,8);
\draw [fill] (18,10) circle [radius=0.1];
\draw [fill] (19,8) circle [radius=0.1];
\node [right] at (18.5,9) {\tiny $U''_{i}$};
\draw (19,8)--(20,6);
\draw [middlearrow={stealth}] (20,6)--(22,2);
\draw [fill] (20,6) circle [radius=0.1];
\draw [fill] (22,2) circle [radius=0.1];
\node [right] at (21,4) {\tiny $U''_{j}$};
\draw (22,2)--(23,0);

\draw [loosely dotted] (2,14)--(0,10);
\draw [loosely dotted] (4,14)--(6,10);
\draw (0,10)--(6,10);
\draw (0,10)--(-1,8);
\draw (6,10)--(7,8);
\draw (-1,8)--(7,8);

\draw [loosely dotted] (0,10)--(-2,10);
\draw [loosely dotted] (-1,8)--(-3,8);
\draw [fill] (0,10) circle [radius=0.1];
\draw [fill] (-1,8) circle [radius=0.1];
\draw [loosely dotted] (6,10)--(18,10);
\draw [loosely dotted] (7,8)--(19,8);
\draw [fill] (6,10) circle [radius=0.1];
\draw [fill] (7,8) circle [radius=0.1];

\node at (3,9) {\tiny $T'_{i}$};

\draw [loosely dotted] (10,14)--(6,6);
\draw [loosely dotted] (12,14)--(16,6);
\draw (6,6)--(16,6);
\draw (6,6)--(4,2);
\draw (16,6)--(18,2);
\draw (4,2)--(18,2);

\draw [loosely dotted] (6,6)--(-4,6);
\draw [loosely dotted] (4,2)--(-6,2);
\draw [fill] (6,6) circle [radius=0.1];
\draw [fill] (4,2) circle [radius=0.1];
\draw [loosely dotted] (16,6)--(20,6);
\draw [loosely dotted] (18,2)--(22,2);
\draw [fill] (16,6) circle [radius=0.1];
\draw [fill] (18,2) circle [radius=0.1];

\node at (11,4) {\tiny $T'_{j}$};

\draw (-7,0)--(23,0);
\end{tikzpicture}
\end{center}
\caption{Disjoint substacks $T'_{i}$ and $T'_{j}$ in $T'$.}
\label{Disjoing substactk T'_i and T'_j in T'.}
\end{figure}
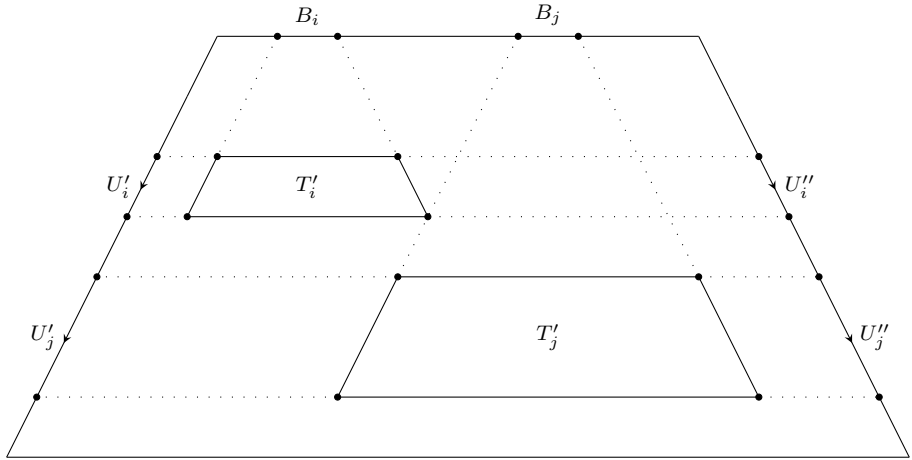
We have
\begin{align*}
\text{Area}(w)
&=\text{Area}(\Delta)                                              \\
&=\text{Area}(T_{1})+\cdots+\text{Area}(T_{m})                     \\
&\leq\text{Area}(T'_{1})+\cdots+\text{Area}(T'_{m})                     \\
&\leq\text{Area}(T')                                                    \\
&\leq C(h^{3}+lh^{2})                                             \\
&\leq 2Cn^{3} 
\end{align*} 
for some positive constant $C$, which does not depend on $|w|=n$. The second last inequality follows by Lemma \ref{stacks are cubic}, and the last inequality holds since $l\leq n$ and $h\leq n$. Thus, we prove the claim for the first case.

\textbf{Case 2.} Suppose that the van Kampen diagram $\Delta$ has some stacks whose bases are not all on $\partial\Delta$, as shown in Figure \ref{Case 2: at least one of the stacks such that part of its bases is part of Delta.}. Divide those stacks into smaller stacks as following:
\begin{figure}[H]
\begin{center}
\begin{tikzpicture}[scale=0.7]



\draw (2,0)--(2,5);
\draw (4,0)--(4,5);

\draw (6,0)--(6,-5);
\draw (9,0)--(9,-5);

\draw [fill] (2,0) circle [radius=0.1];
\draw [fill] (4,0) circle [radius=0.1];
\draw [fill] (6,0) circle [radius=0.1];
\draw [fill] (9,0) circle [radius=0.1];

\node at (1,1.5) {\tiny $\overline{T}_{1}$};
\node at (3,2.5) {\tiny $\overline{T}_{2}$};
\node at (6,1.5) {\tiny $\overline{T}_{3}$};
\node at (5.5,-1.5) {\tiny $\overline{T}_{4}$};
\node at (7.5,-2.5) {\tiny $\overline{T}_{5}$};
\node at (10,-1.5) {\tiny $\overline{T}_{6}$};

\draw (4,3)--(4,5)--(2,5)--(2,3);

\draw [fill] (2,3) circle [radius=0.1];
\draw [fill] (4,3) circle [radius=0.1];
\draw [fill] (4,5) circle [radius=0.1];
\draw [fill] (2,5) circle [radius=0.1];

\node [above] at (3,5) {\tiny $\overline{B}_{2}$};
\node [left] at (2,4) {\tiny $U'_{1}$};
\node [right] at (4,4) {\tiny $U''_{1}$};

\draw (0,0)--(8,0)--(8,3)--(4,3);
\draw (2,3)--(0,3)--(0,0);


\draw [fill] (0,3) circle [radius=0.1];
\draw [fill] (0,0) circle [radius=0.1];
\draw [fill] (5,0) circle [radius=0.1];
\draw [fill] (8,3) circle [radius=0.1];

\node [above] at (1,3) {\tiny $\overline{B}_{1}$};
\node [left] at (0,1.5) {\tiny $U'_{2}$};
\node [right] at (8,1.5) {\tiny $U''_{2}$};
\node [above] at (6,3) {\tiny $\overline{B}_{8}$};

\draw (5,0)--(11,0)--(11,-3)--(9,-3);
\draw (6,-3)--(5,-3)--(5,0);


\draw [fill] (11,0) circle [radius=0.1];
\draw [fill] (11,-3) circle [radius=0.1];
\draw [fill] (9,-3) circle [radius=0.1];
\draw [fill] (5,-3) circle [radius=0.1];
\draw [fill] (6,-3) circle [radius=0.1];

\node [left] at (5,-1.5) {\tiny $U'_{3}$};
\node [below] at (5.5,-3) {\tiny $\overline{B}_{4}$};
\node [below] at (10,-3) {\tiny $\overline{B}_{6}$};
\node [right] at (11,-1.5) {\tiny $U''_{3}$};

\draw (9,-3)--(9,-5)--(6,-5)--(6,-3);

\draw [fill] (6,-5) circle [radius=0.1];
\draw [fill] (9,-5) circle [radius=0.1];

\node [left] at (6,-4) {\tiny $U'_{4}$};
\node [below] at (7.5,-5) {\tiny $\overline{B}_{5}$};
\node [right] at (9,-4) {\tiny $U''_{4}$};
\end{tikzpicture}
\end{center}
\caption{}
\label{Divide the subdiagram Delta.}
\end{figure}

\noindent In Figure \ref{Divide the subdiagram Delta.}, each of the stacks $\overline{T}_{i}$ has at least one base $\overline{B}_{i}$ that is on $\partial\Delta$. Thus, Case 2 follows by Case 1.

\textbf{Case 3.} Suppose that there is a stack in $\Delta$ whose bases are not on $\partial\Delta$, as shown in Figure \ref{Case 3: at least one of the stacks whose bases are not part of Delta.}. Divide $\Delta$ and rearrange the stacks as follows: 
\begin{figure}[H]
\begin{center}
\begin{tikzpicture}[scale=0.7]

\path (2,0) rectangle (4,8);

\draw (2,0)--(2,8);
\draw (4,0)--(4,8);

\draw [fill] (2,0) circle [radius=0.1];
\draw [fill] (4,0) circle [radius=0.1];
\draw [fill] (2,8) circle [radius=0.1];
\draw [fill] (4,8) circle [radius=0.1];

\node [above] at (1,8) {\tiny $B'_{1}$};
\node [above] at (3,8) {\tiny $B''_{1}$};
\node [above] at (5,8) {\tiny $B'''_{1}$};

\node [below] at (1,0) {\tiny $B'_{4}$};
\node [below] at (3,0) {\tiny $B''_{4}$};
\node [below] at (5,0) {\tiny $B'''_{4}$};

\node at (1,6.5) {\tiny $\overline{T}_{1}$};
\node at (5,6.5) {\tiny $\overline{T}_{2}$};
\node at (3,4) {\tiny $\overline{T}_{3}$};
\node at (1,1.5) {\tiny $\overline{T}_{4}$};
\node at (5,1.5) {\tiny $\overline{T}_{5}$};

\draw (0,5)--(2,5);
\draw (4,5)--(6,5)--(6,8)--(0,8)--(0,5);


\draw [fill] (0,5) circle [radius=0.1];
\draw [fill] (2,5) circle [radius=0.1];
\draw [fill] (4,5) circle [radius=0.1];
\draw [fill] (6,5) circle [radius=0.1];
\draw [fill] (6,8) circle [radius=0.1];
\draw [fill] (0,8) circle [radius=0.1];

\node [left] at (0,6.5) {\tiny $U'_{1}$};
\node [below] at (1,5) {\tiny $B_{2}$};
\node [below] at (5,5) {\tiny $B_{6}$};
\node [right] at (6,6.5) {\tiny $U''_{1}$};

\draw (2,3)--(2,5);
\draw (4,3)--(4,5);


\node [left] at (2,4) {\tiny $U'_{2}$};
\node [right] at (4,4) {\tiny $U''_{2}$};

\draw (0,0)--(6,0)--(6,3)--(4,3);
\draw (2,3)--(0,3)--(0,0);

\draw [fill] (2,3) circle [radius=0.1];
\draw [fill] (0,3) circle [radius=0.1];
\draw [fill] (0,0) circle [radius=0.1];
\draw [fill] (6,0) circle [radius=0.1];
\draw [fill] (6,3) circle [radius=0.1];
\draw [fill] (4,3) circle [radius=0.1];

\node [above] at (1,3) {\tiny $B_{3}$};
\node [left] at (0,1.5) {\tiny $U'_{3}$};
\node [right] at (6,1.5) {\tiny $U''_{3}$};
\node [above] at (5,3) {\tiny $B_{5}$};

\end{tikzpicture}
\end{center}
\caption{}
\label{Divide the subdiagram Delta'' in Delta}
\end{figure}
\noindent Again, all the stacks in Figure \ref{Divide the subdiagram Delta'' in Delta} have a base that is on $\partial\Delta$. Thus, Case 3 follows by Case ~1. 

This completes the proof of the lemma. 
\end{proof}

\section{Graphs with $K_{4}$ subgraphs}\label{section graphs with K4 subgraphs}

In the previous sections, we considered finite simplicial graphs $\Gamma$ that do not contain $K_{4}$ subgraphs. In this section, we consider finite simplicial graphs that can contain $K_{4}$ subgraphs. The main difference is that the flag complexes on finite simplicial graphs with $K_{4}$ subgraphs are not $2$-dimensional.

Unfortunately, we have not been able to characterize the Dehn function $\delta_{H_{\Gamma}}$ when $\Delta_{\Gamma}$ is not $2$-dimensional. Instead, we obtain a lower bound for $\delta_{H_{\Gamma}}$ when $\Gamma$ contains induced subgraphs that satisfy the assumptions of Theorem \ref{Dehn functions for square boundary}.

\begin{definition}
We say that a subgroup $H$ is a \emph{retract} of a group $G$ if there is a homomorphism $r:G\rightarrow H$, such that $r:H\rightarrow H$ is the identity. We call the homomorphism $r$ a \emph{retraction}. 
\end{definition}

A standard fact about group retract is that if $H$ is a retract of a finitely presented group $G$, then $H$ is also finitely presented. The following lemma says that group retractions do not increase Dehn functions.

\begin{lemma}\textnormal{(\cite{brick}, Lemma 2.2)}\label{Dehn functions of group retract}
If $H$ is a retract of a finitely presented group $G$, then $\delta_{H}\preceq\delta_{G}$.
\end{lemma}

\begin{prop}\label{retract of Bestvina--Brady groups}
Let $\Gamma$ be a finite simplicial graph. If $\Gamma'$ is a connected induced subgraph of $\Gamma$, then $H_{\Gamma'}$ is a retract of $H_{\Gamma}$.
\end{prop}

\begin{proof}
Since $\Gamma'$ is an induced subgraph of $\Gamma$, $A_{\Gamma'}$ is a retract of $A_{\Gamma}$. Let $r:A_{\Gamma}\rightarrow A_{\Gamma'}$ be a retraction; define $r'=r\vert_{H_{\Gamma}}:H_{\Gamma}\rightarrow H_{\Gamma'}$. Since $\Gamma$ and $\Gamma'$ are connected, $H_{\Gamma}$ and $H_{\Gamma'}$ are finitely generated and their generating sets are sets of oriented edges of $\Gamma$ and $\Gamma'$, respectively. It suffices to show that $r'$ is the identity on the generating set of $H_{\Gamma'}$. Let $e$ be an oriented edge of $\Gamma'$ with initial vertex $v$ and terminal vertex $w$. The generator $e$ of $H_{\Gamma'}$ can be expressed in terms of the generators of $A_{\Gamma'}$, that is, $e=vw^{-1}$. Since $r:A_{\Gamma}\rightarrow A_{\Gamma'}$ is a retraction, we have
\begin{align*}
r'(e)=r(vw^{-1})=vw^{-1}=e.
\end{align*}
This shows that $r':H_{\Gamma}\rightarrow H_{\Gamma'}$ is a retraction.
\end{proof}

We establish the lower bound as promised.

\begin{prop}\label{Lower bound on K4 case}
Let $\Gamma$ be a finite simplicial graph such that $\Delta_{\Gamma}$ is simply-connected. If $\Gamma$ contains an induced subgraph $\Gamma'$ such that $\Delta_{\Gamma'}$ is a $2$-dimensional triangulated subdisk of $\Delta_{\Gamma}$ that has square boundary and $\dim_{I}(\Delta_{\Gamma'})=d$, $d\in\lbrace 0,1,2 \rbrace$, then $n^{d+2}\preceq\delta_{H_{\Gamma}}(n)$.
\end{prop}

\begin{proof}
By Theorem \ref{Dehn functions for square boundary}, we have $\delta_{H_{\Gamma'}}(n)\simeq n^{d+2}$. Since $\Gamma'$ is an induced subgraph of $\Gamma$, it follows from Proposition \ref{retract of Bestvina--Brady groups} and Lemma \ref{Dehn functions of group retract} that $H_{\Gamma'}$ is a retract of $H_{\Gamma}$ and $n^{d+2}\simeq\delta_{H_{\Gamma'}}(n)\preceq\delta_{H_{\Gamma}}(n)$. 
\end{proof}

We remark that Proposition \ref{Lower bound on K4 case} can be used to give the cubic lower bound in Theorem \ref{Dehn functions for square boundary} for the case $d=1$. We close this section with an example.

\begin{example}\label{Double of K4s}
Let $\Gamma$ be the graph as shown in Figure \ref{example of induced subgraph}:
\begin{figure}[H]
\begin{center}
\begin{tikzpicture}[scale=0.3]
\draw (3,7)--(-4,0);
\draw (3,7)--(0,3);
\draw (3,7)--(0,1.5);
\draw (3,7)--(0,0);
\draw (3,7)--(2,0);
\draw (3,7)--(4,0);
\draw (3,7)--(6,0);
\draw (3,7)--(6,1.5);
\draw (3,7)--(6,3);
\draw (3,7)--(10,0);

\draw (0,3)--(0,0)--(-4,0)--(0,3);

\draw (-4,0)--(0,1.5);
	
\draw (0,0)--(6,0);
	
\draw (10,0)--(6,1.5);
	
\draw (6,3)--(6,0)--(10,0)--(6,3);

\draw [fill] (-4,0) circle [radius=0.1];
\draw [fill] (0,0) circle [radius=0.1];
\draw [fill] (2,0) circle [radius=0.1];
\draw [fill] (4,0) circle [radius=0.1];
\draw [fill] (6,0) circle [radius=0.1];
\draw [fill] (10,0) circle [radius=0.1];

\draw [fill] (0,1.5) circle [radius=0.1];
\draw [fill] (0,3) circle [radius=0.1];
	
\draw [fill] (6,1.5) circle [radius=0.1];
\draw [fill] (6,3) circle [radius=0.1];

\draw [fill] (3,7) circle [radius=0.1];

\draw (3,-7)--(-4,0);
\draw (3,-7)--(0,-3);
\draw (3,-7)--(0,-1.5);
\draw (3,-7)--(0,0);
\draw (3,-7)--(2,0);
\draw (3,-7)--(4,0);
\draw (3,-7)--(6,0);
\draw (3,-7)--(6,-1.5);
\draw (3,-7)--(6,-3);
\draw (3,-7)--(10,0);

\draw (0,-3)--(0,0)--(-4,0)--(0,-3);
 
\draw (-4,0)--(0,-1.5);
\draw (10,0)--(6,-1.5);

\draw (6,-3)--(6,0)--(10,0)--(6,-3);

\draw [fill] (0,-1.5) circle [radius=0.1];
\draw [fill] (0,-3) circle [radius=0.1];
\draw [fill] (6,-1.5) circle [radius=0.1];
\draw [fill] (6,-3) circle [radius=0.1];	
\draw [fill] (3,-7) circle [radius=0.1];

\end{tikzpicture}
\end{center}
\caption{The graph $\Gamma$.}
\label{example of induced subgraph}
\end{figure}
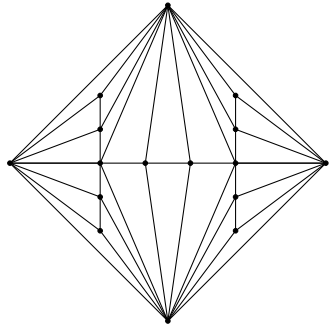

\noindent The flag complex $\Delta_{\Gamma}$ is not $2$-dimensional since $\Gamma$ contains $K_{4}$ subgraphs. Observe that $\Gamma$ contains an induced subgraph $\Gamma'$ that is the suspension of a path of length $3$. Hence, by Theorem \ref{Dehn functions for square boundary} and Proposition \ref{retract of Bestvina--Brady groups}, we have $n^{3}\simeq\delta_{H_{\Gamma''}}(n)\preceq\delta_{H_{\Gamma}}(n)$.
\end{example}

\bibliographystyle{plain}
\bibliography{references}

\begin{thebibliography}{10}

\bibitem{abddy}
Aaron Abrams, Noel Brady, Pallavi Dani, Moon Duchin, and Robert Young.
\newblock Pushing fillings in right-angled {A}rtin groups.
\newblock {\em J. Lond. Math. Soc. (2)}, 87(3):663--688, 2013.

\bibitem{bestvinaandbrady}
Mladen Bestvina and Noel Brady.
\newblock Morse theory and finiteness properties of groups.
\newblock {\em Invent. Math.}, 129(3):445--470, 1997.

\bibitem{bierinormalsubgroupsindualitygroupsandingroupsofcohomologicaldimension2}
Robert Bieri.
\newblock Normal subgroups in duality groups and in groups of cohomological
  dimension {$2$}.
\newblock {\em J. Pure Appl. Algebra}, 7(1):35--51, 1976.

\bibitem{bradyandforester}
Noel Brady and Max Forester.
\newblock Snowflake geometry in {${\rm CAT}(0)$} groups.
\newblock {\em J. Topol.}, 10(4):883--920, 2017.

\bibitem{bradyrileyshortthegeometryofthewordproblemforfinitelygeneratedgroups}
Noel Brady, Tim Riley, and Hamish Short.
\newblock {\em The geometry of the word problem for finitely generated groups}.
\newblock Advanced Courses in Mathematics. CRM Barcelona. Birkh\"{a}user
  Verlag, Basel, 2007.
\newblock Papers from the Advanced Course held in Barcelona, July 5--15, 2005.

\bibitem{bradyandsorokodehnfunctionsofsubgroupsofraags}
Noel {Brady} and Ignat {Soroko}.
\newblock {Dehn functions of subgroups of right-angled Artin groups}.
\newblock {\em arXiv e-prints}, page arXiv:1709.04066, Sep 2017.

\bibitem{brick}
Stephen~G. Brick.
\newblock On {D}ehn functions and products of groups.
\newblock {\em Trans. Amer. Math. Soc.}, 335(1):369--384, 1993.

\bibitem{bridsongeometryofwordproblem}
Martin~R. Bridson.
\newblock The geometry of the word problem.
\newblock In {\em Invitations to geometry and topology}, volume~7 of {\em Oxf.
  Grad. Texts Math.}, pages 29--91. Oxford Univ. Press, Oxford, 2002.

\bibitem{bridsonandhaefliger}
Martin~R. Bridson and Andr\'{e} Haefliger.
\newblock {\em Metric spaces of non-positive curvature}, volume 319 of {\em
  Grundlehren der Mathematischen Wissenschaften [Fundamental Principles of
  Mathematical Sciences]}.
\newblock Springer-Verlag, Berlin, 1999.

\bibitem{carterandforester}
William Carter and Max Forester.
\newblock The {D}ehn functions of {S}tallings-{B}ieri groups.
\newblock {\em Math. Ann.}, 368(1-2):671--683, 2017.

\bibitem{charneyanintroductiontoraag}
Ruth Charney.
\newblock An introduction to right-angled {A}rtin groups.
\newblock {\em Geom. Dedicata}, 125:141--158, 2007.

\bibitem{dicksandleary}
Warren Dicks and Ian~J. Leary.
\newblock Presentations for subgroups of {A}rtin groups.
\newblock {\em Proc. Amer. Math. Soc.}, 127(2):343--348, 1999.

\bibitem{dison}
Will Dison.
\newblock An isoperimetric function for {B}estvina-{B}rady groups.
\newblock {\em Bull. Lond. Math. Soc.}, 40(3):384--394, 2008.

\bibitem{gromovhyperbolicgroup}
M.~Gromov.
\newblock Hyperbolic groups.
\newblock In {\em Essays in group theory}, volume~8 of {\em Math. Sci. Res.
  Inst. Publ.}, pages 75--263. Springer, New York, 1987.

\bibitem{stefansuciualgebraicinvariantsforbbgroups}
Stefan Papadima and Alexander Suciu.
\newblock Algebraic invariants for {B}estvina-{B}rady groups.
\newblock {\em J. Lond. Math. Soc. (2)}, 76(2):273--292, 2007.

\bibitem{stallingsafinitelypresentedgroupwhose3dimentionalintegralhomologyisnotfinitelygenerated}
John Stallings.
\newblock A finitely presented group whose 3-dimensional integral homology is
  not finitely generated.
\newblock {\em Amer. J. Math.}, 85:541--543, 1963.

\end{thebibliography}
\end{document}